\documentclass{article}
\usepackage[utf8]{inputenc}
\usepackage[margin=1in]{geometry}
\usepackage{xcolor}
 \pdfoutput=1
\usepackage{amsmath,amssymb}
\usepackage{multicol}
\usepackage{array}
\usepackage{framed}
\usepackage{graphicx}
\usepackage{setspace}
\usepackage{amssymb,amsmath,amsthm,amsfonts,amscd,centernot}
\usepackage{hyperref}
\usepackage{color}
\usepackage{constants}
\usepackage{booktabs}
\usepackage{tabularx}
\usepackage[retainorgcmds]{IEEEtrantools}
\usepackage[notref,notcite,final]{showkeys}
\usepackage[final]{pdfpages}
\usepackage{fancyhdr}
\usepackage{caption}
\usepackage{enumerate}
\usepackage[T1]{fontenc}

\usepackage{tikz}
\usepackage{tikz-cd}
\usetikzlibrary{cd}

\newtheorem{theorem}{Theorem}[section]
\newtheorem*{theorem*}{Theorem}
\newtheorem{lemma}[theorem]{Lemma}
\newtheorem*{lemma*}{Lemma}
\newtheorem{proposition}[theorem]{Proposition}
\newtheorem{claim}[theorem]{Claim}
\newtheorem{corollary}[theorem]{Corollary}

\newtheorem{Definition}[theorem]{Definition}
\newenvironment{definition}
{\begin{Definition}\rm}{\end{Definition}}
\newtheorem{Example}[theorem]{Example}

\newtheorem{Algorithm}[theorem]{Algorithm}
\newenvironment{algorithm}
{\begin{Algorithm}\rm}{\end{Algorithm}}
\newtheorem{Heuristic}[theorem]{Heuristic}
\newenvironment{heuristic}
{\begin{Heuristic}\rm}{\end{Heuristic}}

\theoremstyle{definition}

\newcommand{\E}{\mathcal{E}}
\newcommand{\Tam}{\textup{Tam}}
\newcommand{\Des}{\textup{Des}}
\newcommand{\cov}{\textup{cov}}
\newcommand{\MC}{\textup{MC}}
\newcommand{\Av}{\textup{Av}}
\newcommand*{\Nearrow}{\rotatebox[origin=c]{45}{\(\Longrightarrow\)}}

\makeatletter
\newcommand{\tpitchfork}{%
  \vbox{
    \baselineskip\z@skip
    \lineskip-.52ex
    \lineskiplimit\maxdimen
    \m@th
    \ialign{##\crcr\hidewidth\smash{$-$}\hidewidth\crcr$\pitchfork$\crcr}
  }%
}
\makeatother

\DeclareMathOperator{\ord}{\textrm{ord}}

\newconstantfamily{abcon}{symbol=c}

\newcommand{\dfn}[1]{\textcolor{blue}{\emph{#1}}} 
\newcommand\numberthis{\addtocounter{equation}{1}\tag{\theequation}}
\graphicspath{ {./images/} }

\title{Tighter Bounds on the Expected Absorbing Time of Ungarian Markov Chains}
\author{Eric Shen}
\date{May 2024}

\begin{document}

\maketitle

\begin{abstract}
    In $2023$, Defant and Li defined the Ungarian Markov chain $\mathbf{U}_L$ associated to a finite lattice $L$. This Markov chain has state space $L$, and from any state $x \in L$ transitions to the meet of $\{x\} \cup T$, where $T$ is a randomly selected subset of the elements of $L$ covered by $x$. For any lattice $L$, let $\E(L)$ be the expected number of steps until the maximal element of $L$ transitions into the minimal element in the Ungarian Markov chain. We show that $\mathcal{E}(L)$ is linear in $n$ when $L$ is the weak order on the symmetric group $S_n$, and satisfies an $n^{1-o(1)}$ lower bound when $L$ is the $n^\text{th}$ Tamari lattice. This completely resolves a conjecture by Defant and Li and partially resolves another.
\end{abstract}

\section{Introduction}\label{sIntro}

Consider a positive integer $n$ and any permutation $\sigma \in S_n$. An index $i \in [n - 1]$ is a \dfn{descent} of $\sigma$ if $\sigma(i) > \sigma(i+1)$. Let $\Des(\sigma)$ denote the set of descents of $\sigma$. Given a permutation $\sigma \in S_n$ and a subset $T \subset \Des(\sigma)$, the \dfn{Ungar move} on $\sigma$ associated with $T$ consists of reversing every descent in $T$, where consecutive descents in $T$ are treated as blocks to be reversed. For example, given $\sigma = 416523$, if only index $1$ is selected then the Ungar move sends
$\textcolor{black}{41}6523 \rightarrow \textcolor{black}{14}6523$, while if both indices $3$ and $4$ are selected, then the Ungar move sends $41\textcolor{black}{652}3 \rightarrow 41\textcolor{black}{256}3$. We call an Ungar move on $\sigma$ associated to $T$ \dfn{nontrivial} if $T \neq \varnothing$, and the \dfn{maximal Ungar move} on $\sigma$ is the move associated to $T = \Des(\sigma)$.

Ungar moves were first studied by Ungar \cite{ungar1982} to answer a question posed by Scott \cite{scott1970ungarquestion}, on the minimum number of distinct slopes formed by the lines connecting each pair of points within a set of $n$ distinct points (not all collinear) in a plane. To this end, Ungar proved the following bounds on the number of Ungar moves necessary to transform a given permutation into the identity.

\begin{theorem}(\cite{ungar1982})
    Let $n \geq 4$ be an integer.
    \begin{enumerate}[(a)]
        \item At most $n - 1$ \textit{maximal} Ungar moves suffice to send any permutation to the identity.
        \item Consider a sequence of Ungar moves which send $n(n-1) \cdots 1$ to $12 \cdots n$. Given that the first move in this sequence is nontrivial and not maximal, the total number of moves in the sequence must be at least $n$ if $n$ is even, and at least $n - 1$ if $n$ is odd.
    \end{enumerate}
\end{theorem}

In \cite{defant2023ungarian}, Defant and Li studied the Ungar moves in a probabilistic setting by introducing random Ungar moves, defined as follows. Fix $p \in (0, 1]$. Then, given any $\sigma \in S_n$, randomly select a subset $T \subset \Des(\sigma)$ by including any element of $\Des(\sigma)$ into $T$ with independent probability $p$. A \dfn{random Ungar move} on $\sigma$ is now the Ungar move associated with $T$. Note that any nontrivial Ungar move decreases the number of inversions, and hence repeatedly applying any sequence of nontrivial Ungar moves to any permutation will eventually reduce it to the identity. Motivated by this, Defant and Li raised the question of the expected number of random Ungar moves that must be applied to the decreasing permutation $n(n-1) \cdots 1$ until we reach the identity permutation $12 \cdots n$. To this end, they defined the \dfn{Ungarian Markov chain} $\textbf{U}_{S_n}$ on $S_n$ (with parameter $p$) as the Markov chain whose states are the elements of $S_n$, and transitions are given by the random Ungar moves \cite{defant2023ungarian}. They also defined $\E_p(S_n)$ to be the expected number of steps for the permutation $n(n-1) \cdots 1$ to transition to the identity $12 \cdots n$. In particular, they derived the following bound on $\E_p(S_n)$:
\begin{theorem}
        \cite[Theorem 1.9]{defant2023ungarian} Fix a $p \in (0, 1)$. Then
    \begin{align*}
        n - 1 + o_p(1) &\leq \E_p(S_n) \leq \frac{8}{p}n \log n + O_p(n)
    \end{align*}
    as $n \rightarrow \infty$.
\end{theorem}
Defant and Li also conjectured that for fixed $p \in (0, 1)$, $\E_p(S_n)$ is linear in $n$. The first main result of this paper proves this conjecture.
\begin{theorem}\label{thmMainTheorem1}
    Fix a $p \in (0, 1)$. Then
    \begin{align*}
        \E_p(S_n) \leq \bigg(\frac{1+\sqrt{1-p}}{p}\bigg)n + o_p(n).
    \end{align*}
    as $n \rightarrow \infty$.
\end{theorem}
\subsection{Ungarian Markov Chains on Lattices}\label{ssIntroLattice}

Throughout this article, we will assume all partially ordered sets (posets) to be finite. A \dfn{lattice} $L$ is a poset such that any two elements $x, y \in L$ have a greatest lower bound, denoted with $x \wedge y$ and called their ``\dfn{meet},'' and a lowest upper bound, denoted with $x \vee y$ and called their ``\dfn{join}.'' In \cite{defant2023ungarian} Defant and Li noted that by considering $S_n$ under the right weak order, the definition of the Ungar move generalizes readily to any lattice. In particular, for any $y \in L$, define $\cov_L(y) = \{x \in L : x \lessdot y\}$ to be the set of elements that $y$ covers in $L$. Then picking a subset $T \subset \Des(\sigma)$ and swapping all the selected descents is equivalent to picking a subset $T \subset \cov_{S_n}(\sigma)$ and sending $\sigma \mapsto \bigwedge (\{\sigma\} \cup T)$.

In this manner Defant and Li generalized the definition of random Ungar moves and the Ungarian Markov chain as follows.
\begin{definition}
    Fix some $p \in (0, 1]$ and a lattice $L$. For any element $x \in L$, a \dfn{random Ungar move} on $x$ is an operation that randomly picks a subset $T \subset \cov_L(x)$ by including each element of $\cov_L(x)$ with independent probability $p$, and then maps $x \mapsto \bigwedge (\{x\} \cup T)$. The \dfn{Ungarian Markov chain} $\textbf{U}_L$ on $L$ is defined as the Markov chain with state space $L$, satisfying that for all $x, y \in L$, the transition probability of $x \rightarrow y$ is given by
    \begin{align*}
        \mathbb{P}(x \rightarrow y) &= \sum_{\substack{T \subset \cov_L(x) \\ \bigwedge (\{x\} \cup T) = y}} p^{|T|}(1-p)^{|\cov_L(x)| - |T|}.
    \end{align*}
\end{definition}

It is well-known that every finite lattice $L$ has a minimal element $\hat{0}$ and a maximal element $\hat{1}$. Evidently from any state in the Ungarian Markov chain $\textbf{U}_L$ on a lattice $L$, we may reach $\hat{0}$ in a finite number of steps (i.e., the Markov chain is absorbing). Hence we may generally define $\E_p(L)$ to be the expected number of steps until $\hat{1}$ transitions to $\hat{0}$ in $\textbf{U}_L$.

In \cite{defant2023ungarian}, aside from the lattices $S_n$, Defant and Li also studied the Ungarian Markov chains of another family of lattices: the Tamari lattices $\Tam_n$, first defined by Dov Tamari in \cite{tamari1962}. The $n^\text{th}$ Tamari lattice may be interpreted as a sublattice of $S_n$ as follows. Call a permutation $\sigma \in S_n$ \dfn{$312$-avoiding} if there do not exist indices $1 \leq i_1 < i_2 < i_3 \leq n$ such that $\sigma(i_1) > \sigma(i_3) > \sigma(i_2)$. Let $\Av_n(312)$ denote the subset of $S_n$ consisting of the $312$-avoiding permutations. Consider $S_n$ as a lattice under the right weak order, and let $\Av_n(312)$ inherit a poset structure from $S_n$. Then in fact $\Av_n(312)$ is a sublattice of $S_n$ (i.e. it is closed under the meet and join operations), and the sublattice $\Av_n(312)$ is isomorphic to $\Tam_n$ as a poset \cite[Theorem 7.8]{reading2006cambrian1}.

The second main objective of this paper is to study the asymptotics of $\E_p(\Tam_n)$. In \cite{defant2023ungarian}, Defant and Li derived an explicit asymptotic upper bound on $\E_p(\Tam_n)$; to introduce it we will first define some functions. Let $\zeta_p(n)$ be the probability that the maximum of $n$ independent geometric random variables, each with expected value $1/p$, is attained uniquely. In \cite{thomas1990geommax}, Bruss and O'Cinneide proved that if we define
\begin{align*}
    \Upsilon_p(x) &= \begin{cases} px\sum_{k \in \mathbb{Z}}(1-p)^k e^{-(1-p)^kx} & p < 1 \\ 0 & \text{else}, \end{cases}
\end{align*}
then $\lim_{n \rightarrow \infty} \zeta_p(n) - \Upsilon_p(n) = 0$. Now, let $\zeta_p^- = \liminf_{n \rightarrow \infty} \zeta_p(n)$ and $\zeta_p^+ = \limsup_{n \rightarrow \infty} \zeta_p(n)$. Note that $\Upsilon_p((1-p)x) = \Upsilon_p(x)$ for all $x > 0$, so thus
\begin{align*}
    \zeta_p^+ &= \max_{0 < x < 1} \Upsilon_p(x), \\
    \zeta_p^- &= \min_{0 < x < 1} \Upsilon_p(x) \geq p(1-p)e^{p-1}.
\end{align*}
Now, Defant and Li established the following upper bound on $\E_p(\Tam_n)$:
\begin{theorem}
    \cite[Theorem 1.21]{defant2023ungarian} For fixed $p \in (0, 1)$, we have that
    \begin{align*}
    \E_p(\Tam_n) &\leq \frac{2}{p}\bigg(\sqrt{\zeta_p^+(1+\zeta_p^+)}-\zeta_p^+\bigg)n + o_p(n).
\end{align*}
\end{theorem}
Defant and Li also presented the question of establishing any nontrivial lower bound on $\E_p(\Tam_n)$. The second main result of this paper establishes the following bound on $\E_p(\Tam_n)$, which notably for fixed $p$ is $n^{1 - o(1)}$.
\begin{theorem}\label{thmMainTheorem2}
There exists an absolute constant $\Cl[abcon]{1}$ such that for all $p \in (0, 1)$ and $n \geq 16$, we have that
    \begin{align*}
        \E_p(\Tam_n) &\geq \zeta_p^-n\exp\Big(-p^8\exp(\Cr{1}/p^2)(\log \log n)^4\Big)
    \end{align*}
\end{theorem}

The paper is now organized as follows. Section \ref{sPrelim} presents introductory material on statistics and well-known Markov processes (notably, the multicorner growth Totally Asymmetric Exclusion Process) that will be used to prove Theorems \ref{thmMainTheorem1} and \ref{thmMainTheorem2}. Section \ref{sSnProof} proves Theorem \ref{thmMainTheorem1}. Section \ref{sTamariBackground} presents an interpretation of $\Tam_n$ as a poset on $n$-vertex ordered forests that will be used to prove Theorem \ref{thmMainTheorem2}. Section \ref{sTamnProof} proves Theorem \ref{thmMainTheorem2}.

\subsection{Acknowledgements}

This work was done at the University of Minnesota Duluth with support from Jane Street Capital and the National Security Agency. This work was also supported by NSF Grant 2140043. The author would like to thank Joe Gallian and Colin Defant for organizing the Duluth REU and proofreading the paper. The author would also like to thank Evan Chen, Daniel Zhu, Mitchell Lee, and Rupert Li for their many helpful comments on the paper, as well as Noah Kravitz for many helpful discussions.

\section{Preliminaries}\label{sPrelim}

Throughout this paper, the sequence $c_1, c_2, \dots$ denotes a sequence of positive, absolute constants. Moreover, any implied constants are assumed to be positive and absolute, unless otherwise indicated. Implied constants that are not absolute are indicated by subscripts denoting the dependence.

\subsection{Background on Posets}

For further reference, see \cite{stanley1986enumerativecombo}.

Let $P$ be a poset. A \dfn{subposet} of $P$ is a subset $P' \subset P$ equipped with a partial order $\leq'$, such that for any $x, y \in P'$, we have that $x \leq' y$ in $P'$ if and only if $x \leq y$ in $P$. For any $x, y \in P$, we say that $x$ \dfn{covers} $y$ (which we denote using $x \gtrdot y$) if $x > y$ and there does not exist a $z \in P$ such that $x > z > y$. The \dfn{Hasse diagram} of a poset $P$ is a drawing, where the vertices are the elements of $P$, and for any two $x, y \in P$, if $x \gtrdot y$, then the corresponding vertex of $x$ is connected to the corresponding vertex of $y$ via an edge where $x$ is drawn higher (vertically) than $y$. Now, a \dfn{chain} of $P$ is a subset $\mathcal{C} \subset P$ such that all its elements are comparable; such a chain is \dfn{maximal} if there does not exist another chain $\mathcal{C}'$ for which $\mathcal{C} \subsetneq \mathcal{C}'$. The \dfn{length} of such a chain $\mathcal{C}$ is $|\mathcal{C}| - 1$. Finally, a finite lattice is \dfn{graded} if all of its maximal chains are of the same length.

\subsection{Probabilistic Background}

In this section we will prove two useful probability lemmas that will be used in the proofs of Theorem \ref{thmMainTheorem1} and \ref{thmMainTheorem2}. These two lemmas will both follow from standard applications of well-known results in statistics (e.g., the Markov bound, Chernoff bounds, random walk theory).

\subsubsection{Sum of Geometric Variables}

Throughout the proof of Theorem \ref{thmMainTheorem2}, we will often use the following tail bounds on a sum of independent geometric random variables. Here we use the convention that a \dfn{geometric random variable with parameter $p$} is a random variable $X$ which satisfies that for all positive integers $k$, $\mathbb{P}(X = k) = (1-p)^{k-1}p$.
\begin{lemma}\label{lemmaSumGeomBound}
    Fix a $p \in (0, 1]$. Let $G_1, G_2, \dots, G_k$ be independent geometric random variables with parameter $p$. Then for all $t > 0$, we have that
    \begin{align*}
        \mathbb{P}\bigg(\sum_{i=1}^k G_i > \frac{k}{p} + t\sqrt{\frac{k}{p^3}} \bigg) \leq e^{-\frac{t^2}{2p + 2t\sqrt{p/k}}} \quad \textup{and} \quad 
        \mathbb{P}\bigg(\sum_{i=1}^k G_i < \frac{k}{p} - t\sqrt{\frac{k}{p^3}} \bigg) \leq e^{-\frac{t^2}{2p - t\sqrt{p/k}}}.
    \end{align*}
\end{lemma}
\begin{proof}
    We only prove the first bound, as the proof for the other bound is analogous. Extend $G_1, G_2, \thinspace \dots$ to an infinite sequence. Let
    \begin{align*}
        B_i &= \begin{cases} 1 & i = \sum_{j = 1}^m G_j \text{ for some }m \\
        0 & \text{else.} \end{cases}
    \end{align*}
    Then evidently each $B_i$ is given by an independent Bernoulli variable with parameter $p$ (i.e. $\mathbb{P}(B_i = 1) = p$). Let $\mu = k + t\sqrt{k/p}$ and $\delta = \frac{t/\sqrt{kp}}{1 + t/\sqrt{kp}}$; note that $(1 - \delta)\mu = k$. We thus have that by a multiplicative Chernoff bound,
    \begin{align*}
        \mathbb{P}\bigg(\sum_{i=1}^k G_i > \frac{k}{p} + t\sqrt{\frac{k}{p^3}} \bigg) = \mathbb{P}\Bigg(\sum_{i=1}^{  \lfloor k/p + t\sqrt{k/p^3}  \rfloor } B_i < k \Bigg) \leq e^{-\delta^2 \mu /2} = e^{-\frac{t^2}{2p + 2t\sqrt{p/k}}},
    \end{align*}
    as desired.
\end{proof}

\subsubsection{Simple Random Walks}

Consider an infinite sequence $X_1, X_2, \ldots$ of independent and identically distributed (i.i.d.) random variables in $\mathbb{Z}$, where
\begin{align*}
    X_i &= \begin{cases} 1 & \text{with probability }\frac{1}{2} \\ -1 & \text{with probability }\frac{1}{2}. \end{cases}
\end{align*}

The \dfn{simple random walk on $\mathbb{Z}$} is the random process given by the infinite sequence of random variables $S_1, S_2, \thinspace \ldots$ where $S_n = X_1 + \cdots + X_n$. The \dfn{hitting time} of any nonzero integer $m$ is the random variable $\tau_m$, denoting the smallest positive integer $t$ for which $S_t = m$. Note that $\tau_m$ is a sum of $m$ i.i.d.\ copies of $\tau_1$ for all positive integers $m$ (and by symmetry, $\tau_m = \tau_{-m}$). Now, it is well known that the generating function $\mathbb{E}[z^{\tau_1}]$ of $\tau_1$ satisfies
\begin{align*}
    \mathbb{E}[z^{\tau_1}]  = \sum_{i=0}^{\infty} \mathbb{P}(\tau_1 = n)z^n &= \frac{1 - \sqrt{1-z^2}}{z} \\
    &= \sum_{n=0}^{\infty} (-1)^n\binom{1/2}{n+1}z^{2n+1}.
\end{align*}
In particular, by elementary calculus we may find that $(-1)^n\binom{1/2}{n+1} = \Theta(n^{-3/2})$.
Using the information above, we may derive the following bound.
\begin{lemma}\label{lemmaRandomWalk}
    There exists an absolute constant $\Cl[abcon]{hittingtime} \leq \frac{1}{2}$ such that for all positive integers $n$, 
    \begin{align*}
        \mathbb{P}(\tau_m \geq n^2) &\geq 1 - e^{-\Cr{hittingtime}m/n}.
    \end{align*}
\end{lemma}
\begin{proof}
    By the above, we have that there exists an absolute constant $\Cl[abcon]{hittingtime0}$ such that for all positive integers $n$,
    \begin{align*}
        \mathbb{P}(\tau_1 = 2n - 1) &\geq \Cr{hittingtime0}n^{-3/2}.
    \end{align*}
    Thus, there exists an absolute constant $\Cr{hittingtime} \leq \frac{1}{2}$ such that for all positive integers $n$,
    \begin{align*}
        \mathbb{P}(\tau_1 \geq n^2) &\geq \frac{\Cr{hittingtime}}{n}.
    \end{align*}
    Using the inequality $1 - x \leq e^{-x}$ (which holds for all $x \in [0, 1]$), we find that for all positive integers $n$,
    \begin{align*}
        \mathbb{P}(\tau_1 < n^2) \leq 1 - \frac{\Cr{hittingtime}}{n} \leq  e^{-\Cr{hittingtime}/n}.
    \end{align*}
    Now, recall that $\tau_m$ may be viewed as a sum of $m$ i.i.d.\ copies of $\tau_1$. Thus we have that
    \begin{align*}
        \mathbb{P}(\tau_m < n^2) \leq \mathbb{P}(\tau_1 < n^2)^m \leq e^{-\Cr{hittingtime}m/n},
    \end{align*}
    and thus
    \begin{align*}
        \mathbb{P}(\tau_m \geq n^2) &\geq 1 - e^{-\Cr{hittingtime}m/n}
    \end{align*}
    as desired.
\end{proof}

Now, fix some parameter $q \in (0, 1/2)$. Consider an infinite sequence $Y_1, Y_2, \ldots$ of i.i.d.\ random variables in $\mathbb{Z}$, where
\begin{align*}
    Y_i &= \begin{cases} 1 & \text{with probability }q \\ 0 & \text{with probability }1 - 2q \\ -1 & \text{with probability }q. \end{cases}
\end{align*}

The \dfn{lazy simple random walk on $\mathbb{Z}$ with parameter $q$} is the random process given by the infinite sequence of random variables $T_1, T_2, \thinspace \ldots$ where $T_n = Y_1 + \cdots + Y_n$. We may again define the \dfn{hitting time} of any nonzero integer $m$ as the random variable $\tau'_m$ denoting the smallest positive integer $t$ for which $T_t = m$. The following corollary will be useful in the proof of Theorem \ref{thmMainTheorem2}.

\begin{corollary}\label{corLazyLemmaRandomWalk}
    Let $\Cr{hittingtime}$ be the absolute constant used in the statement of Lemma \ref{lemmaRandomWalk}. Then for all positive integers $m$ and integers $n > 1$, 
    \begin{align*}
        \mathbb{P}(\tau'_m \geq n^2) &\geq 1 - e^{-\Cr{hittingtime}m/n}.
    \end{align*}
\end{corollary}
\begin{proof}
    Let $\tau^\ast_m$ be the number of indices $i$ between $1$ and $\tau'_m$ inclusive for which $Y_i \neq 0$. Then evidently $\tau^\ast_m \leq \tau'_m$, and moreover $\tau^\ast_m$ shares the same distribution as $\tau_m$ (as $\tau^\ast_m$ considers only the nonzero $Y_i$ when calculating the hitting time.) Hence, we have that
    \begin{align*}
        \mathbb{P}(\tau'_m \geq n^2) \geq \mathbb{P}(\tau^\ast_m &\geq n^2) \geq 1 - e^{-\Cr{hittingtime}m/n},
    \end{align*}
    as desired.
\end{proof}

\subsection{Ungarian Markov Chains on Posets of Order Ideals}

Given any lattice $L$ and element $x \in L$, let the random variable $T_L(x)$ denote the number of steps elapsed in the Ungarian Markov chain $\textbf{U}_L$ of $L$ until $x$ has transitioned into $\hat{0}$. Define $T_L := T_L(\hat{1})$ and $\E_L(x) := \mathbb{E}(T_L(x))$. By definition, $\E_L(\hat{1}) = \E_p(L)$.

Consider any poset $P$. An \dfn{order ideal} of $P$ is a subset $I \subset P$ such that for any $x, y \in P$ satisfying $x \geq y$, we have that $x \in I$ implies $y \in I$. Establish a partial order on the order ideals of $P$ by stating that $I \leq J$ if $I \subset J$; under this partial order, the order ideals of a poset $P$ form a (distributive) lattice, which we will denote as $J(P)$.

Consider any order ideal $I$ of a poset $P$. Evidently $\cov_{J(P)}(I)$ is precisely the set of order ideals $J$ that are obtained by removing a maximal element of $I$. Hence the Ungarian Markov chain $\textbf{U}_{J(P)}$ is stochastically equivalent to the Markov chain with state space $J(P)$ which on each step, given any $I \in J(P)$, deletes a randomly selected subset of the maximal elements of $I$. 

Now, fix a parameter $p$ and a poset $P$. Simulate $\textbf{U}_{J(P)}$ with parameter $p$, starting with the order ideal $I = P$. For any nonnegative integer $k$, let $I_k$ denote the order ideal which we obtain after $k$ steps. For any $x \in P$, let the random variable $G_x$ denote the number of nonnegative integers $k$ such that $x$ is a maximal element of $I_k$; observe that the random variables $G_x$ are independent and all stochastically equivalent to a geometric random variable with parameter $p$. Now, define $\MC(P)$ to be the set of maximal chains of $P$. By considering the last element $x_1 \in P$ to be removed from the order ideal, then the last element $x_2$ among those covering $x_1$ to be removed from the order ideal, and so on, we may iteratively construct a (maximal) chain $C \in \MC(P)$ satisfying $T(J(P)) = \sum_{x \in C} G_x$. Thus we always have
\begin{align*}
    T(J(P)) &= \max_{C \in \MC(P)}\sum_{x \in C} G_x.
\end{align*}
This interpretation reformulates the Ungarian Markov chain $\textbf{U}_{J(P)}$ as a variant of the well-studied \dfn{last-passage percolation with geometric weights} random process. In particular, under this formulation we readily retrieve the following corollary.
\begin{corollary}\label{corDistribSubPoset}
    For any order ideal $J' \in J(P)$ and any $t \geq 0$, we have that
    \begin{align*}
        \mathbb{P}(T_{J(P)}(J') \geq t) \leq \mathbb{P}(T(J(P)) \geq t).
    \end{align*}
\end{corollary}

\begin{proof}
    Interpret $J' \in J(P)$ as an order ideal $P'$ of $P$. Then evidently the subposet $\{x \in J(P) \colon x \leq J'\}$ of $J(P)$ is isomorphic to $J(P')$ as a poset. Now, note that every maximal chain of $P'$ is contained within a maximal chain of $P$. Thus, by simulating the Ungarian Markov chain $\textbf{U}_{J(P)}$ using the geometric random variables $G_x$, we find that
    \begin{align*}
        T_{J(P)}(J') &= T(J(P')) \\
        &= \max_{C \in MC(P')} \sum_{x \in C} G_x \\
        &\leq \max_{C \in MC(P)} \sum_{x \in C} G_x \\
        &= T(J(P)).
    \end{align*}
    From this, we readily obtain that for all $t > 0$,
    \begin{align*}
        \mathbb{P}(T_{J(P)}(J') \geq t) &\leq \mathbb{P}(T(J(P)) \geq t),
    \end{align*}
    as desired.
\end{proof}

Now, for any $n, m \in \mathbb{Z}^+$, define the poset $R_{n, m}$ to be the product of a chain of length $n-1$ and a chain of length $m-1$. Fix a parameter $p \in (0, 1)$. Then the last-passage percolation with geometric weights random process with parameter $p$ on $R_{n, m}$ can be interpreted using the \dfn{multicorner growth Totally Asymmmetric Simple Exclusion Process (TASEP)} with parameter $p$ as follows. The multicorner growth TASEP is a sequence of random Young diagrams $\{\lambda_k\}_{k \in \mathbb{Z}_{\geq 0}}$, where $\lambda_0 = \varnothing$ is the empty set and $\lambda_{k+1}$ is constructed by adding in each external corner of $\lambda_k$ with independent probability $p$. For each $k$, let $\lambda_k' \subset \lambda_k$ be the Young sub-diagram consisting only of the first $n$ rows and $m$ columns of $\lambda_k$. Again, simulate the Ungarian Markov chain $\textbf{U}_{J(R_{n,m})}$ starting at $R_{n, m} \in J(R_{n, m})$, and let the random subset $I_k \subset R_{n, m}$ denote the order ideal we obtain after $k$ steps. Then by considering the Hasse diagram of $R_{n, m}$, we may find that the complement of $I_k$ in $R_{n, m}$ may be viewed as a Young diagram, which is identically distributed (as a random Young diagram) as $\lambda_k'$. This correspondence is illustrated more clearly in Figure \ref{figureYoungDiagram}. In particular, note that the growth of the Young diagrams $\{\lambda_k\}$ outside of the first $n$ rows and $m$ columns do not influence the growth of the Young diagrams $\{\lambda_k'\}$, in the sense that the probability distribution of $\lambda_{k+1}'$ conditioned on $\lambda_k$ is the same as the distribution of $\lambda_{k+1}'$ conditioned on $\lambda_k'$. Hence $T_{J(R_{m,n})}$ is precisely the smallest index $k$ such that $\lambda_k'$ is an $n \times m$ grid of squares. For more information on the multicorner growth TASEP, see \cite[Chapters 4, 5]{romik2014subsequences}.

Now, an important ingredient towards the proof of Theorem \ref{thmMainTheorem1} will be the following convergence of the rescaled distributions of $T_{J(R_{m,n})}$, taken from \cite[Theorem 5.31]{romik2014subsequences} but originally appearing as \cite[Theorem 1.2]{johansson2000multicorner}.
\begin{theorem}\label{thmMultiCornerConvergence}
    Consider any $0 < p < 1$. For any $x, y > 0$, define
    \begin{align*}
        \Phi_p(x, y) &= \frac{1}{p}(x + y + 2\sqrt{(1-p)xy}), \\
        \eta_p(x, y) &= \frac{(1-p)^{1/6}}{p}(xy)^{-1/6}(\sqrt{x} + \sqrt{(1-p)y})^{2/3}(\sqrt{y}+\sqrt{(1-p)x})^{2/3}.
    \end{align*}
    Now, let $\{m_k\}_{k=1}^{\infty}$ and $\{n_k\}_{k=1}^{\infty}$ be sequences of positive integers, satisfying that
    \begin{align*}
        \lim_{k \rightarrow \infty} m_k = &\lim_{k \rightarrow \infty} n_k = \infty, \\
        0 < \liminf_{k \rightarrow \infty} \frac{m_k}{n_k} &< \limsup_{k \rightarrow \infty} \frac{m_k}{n_k} < \infty.
    \end{align*}
    Then we have that for all positive reals $t$,
    \begin{align*}
        \lim_{k \rightarrow \infty} \mathbb{P}\bigg(\frac{T(J(R_{m_k, n_k})) - \Phi_p(m_k, n_k)}{\eta_p(m_k, n_k)} \leq t \bigg) &= F_2(t),
    \end{align*}
    where $F_2(t)$ is the Tracy-Widom distribution.
\end{theorem}
The following asymptotic of $F_2(t)$ will also be useful to us.
\begin{theorem}\label{thmTracyWidomApprox}
    \cite[Equation 25]{baik2008tracywidom} In the $t \rightarrow +\infty$ limit, we have
    \begin{align*}
        F_2(t) = 1 - \frac{1}{32 \pi t^{3/2}}e^{-4t^{3/2}/3}(1 + O(t^{-3/2})).
    \end{align*}
\end{theorem}

\begin{figure}[ht]
    \begin{center}
\begin{tikzpicture}[scale=0.5]

\draw[black, thin] (0,0)--(2,0);
\draw[black, thin] (2,0)--(4,0);
\draw[black, thin] (4,0)--(6,0);
\draw[black, thin] (0,-2)--(2,-2);
\draw[black, thin] (2,-2)--(4,-2);
\draw[black, thin] (4,-2)--(6,-2);
\draw[black, thin] (0,-4)--(2,-4);
\draw[black, thin] (2,-4)--(4,-4);
\draw[black, thin] (4,-4)--(6,-4);
\draw[black, thin] (0,0)--(0,-2);
\draw[black, thin] (2,0)--(2,-2);
\draw[black, thin] (4,0)--(4,-2);
\draw[black, thin] (6,0)--(6,-2);
\draw[black, thin] (0,-2)--(0,-4);
\draw[black, thin] (2,-2)--(2,-4);
\draw[black, thin] (4,-2)--(4,-4);
\draw[black, thin] (6,-2)--(6,-4);

\filldraw[black] (0,0) circle (3pt);
\filldraw[black] (2,0) circle (3pt);
\filldraw[black] (4,0) circle (3pt);
\filldraw[black] (0,-2) circle (3pt);
\filldraw[black] (2,-2) circle (3pt);
\filldraw[black] (0,-4) circle (3pt);
\filldraw[blue] (6,0) circle (3pt);
\filldraw[blue] (6,-2) circle (3pt);
\filldraw[blue] (4,-2) circle (3pt);
\filldraw[blue] (6,-4) circle (3pt);
\filldraw[blue] (4,-4) circle (3pt);
\filldraw[blue] (2,-4) circle (3pt);

\filldraw[red] (-1,1) circle (0.5pt);
\filldraw[red] (1,1) circle (0.5pt);
\filldraw[red] (3,1) circle (0.5pt);
\filldraw[red] (5,1) circle (0.5pt);
\filldraw[red] (-1,-1) circle (0.5pt);
\filldraw[red] (1,-1) circle (0.5pt);
\filldraw[red] (3,-1) circle (0.5pt);
\filldraw[red] (5,-1) circle (0.5pt);
\filldraw[red] (-1,-3) circle (0.5pt);
\filldraw[red] (1,-3) circle (0.5pt);
\filldraw[red] (3,-3) circle (0.5pt);
\filldraw[red] (-1,-5) circle (0.5pt);
\filldraw[red] (1,-5) circle (0.5pt);
\filldraw[red] (-1,-7) circle (0.5pt);
\filldraw[red] (1,-7) circle (0.5pt);

\draw[red, thick] (-1,1)--(1,1);
\draw[red, thick] (1,1)--(3,1);
\draw[red, thick] (3,1)--(5,1);
\draw[red, thick] (-1,-1)--(1,-1);
\draw[red, thick] (1,-1)--(3,-1);
\draw[red, thick] (3,-1)--(5,-1);
\draw[red, thick] (-1,-3)--(1,-3);
\draw[red, thick] (1,-3)--(3,-3);
\draw[red, thick] (-1,-5)--(1,-5);
\draw[red, thick] (-1,-7)--(1,-7);

\draw[red, thick] (-1,1)--(-1,-1);
\draw[red, thick] (-1,-1)--(-1,-3);
\draw[red, thick] (-1,-3)--(-1,-5);
\draw[red, thick] (-1,-5)--(-1,-7);
\draw[red, thick] (1,1)--(1,-1);
\draw[red, thick] (1,-1)--(1,-3);
\draw[red, thick] (1,-3)--(1,-5);
\draw[red, thick] (1,-5)--(1,-7);
\draw[red, thick] (3,1)--(3,-1);
\draw[red, thick] (3,-1)--(3,-3);
\draw[red, thick] (5,1)--(5,-1);

\end{tikzpicture}
\end{center}
    \captionsetup{width=0.8\textwidth}
    \caption{An example of an order ideal of $R_{3, 4}$ and a corresponding Young diagram. The blue \textit{dots} form the order ideal, while the red \textit{boxes} display the Young diagram. Note that the Young diagram may extend beyond the $3 \times 4$ box of dots forming $R_{3, 4}$, but the growth of the Young diagram in this region does not influence the evolution of the corresponding order ideal.}
    \label{figureYoungDiagram}
\end{figure}

\section{Proof of Theorem \ref{thmMainTheorem1}}\label{sSnProof}

We will now prove a more general version of Theorem \ref{thmMainTheorem1}. Note that the ideas in the following proof are modeled after Ungar's argument in \cite{ungar1982}.

\begin{theorem}\label{thmBuffedMainTheorem1}
    Fix a parameter $p \in (0, 1)$ and a positive integer $n$. Then for any $\sigma \in S_n$,
    \begin{align*}
        \E_{S_n}(\sigma) &\leq \bigg(\frac{1+\sqrt{1-p}}{p}\bigg)n + o_p(n).
    \end{align*}
\end{theorem}

Consider any integer $1 \leq k \leq n-1$. Run the Ungarian Markov chain $\textbf{U}_{S_n}$ starting at $\sigma$, and let the random variable $t_{k, \sigma} = t_k$ denote the number of steps until $\sigma$ satisfies that for any $i \in [n]$, we have $\sigma(i) \leq k$ if and only if $i \leq k$.
Our key claim is as follows.
\begin{claim}
    For all nonnegative integers $t$, we have that
    \begin{align*}
        \mathbb{P}(t_{k, \sigma} \geq t) \leq \mathbb{P}(T(J(R_{k, n-k})) \geq t).
    \end{align*}
\end{claim}
\begin{proof}
    For any permutation $\sigma \in S_n$, let $\pi_k(\sigma)$ be the string of length $n$, where the $i^\text{th}$ character of $\pi_k(\sigma)$ is ``E'' if $\sigma(i) \leq k$ and ``N'' if $\sigma(i) > k$. Any such string corresponds to a path in $\mathbb{Z}^2$ from $(0, 0)$ to $(k, n-k)$, where the $i^\text{th}$ step in the path moves up $1$ unit if the $i^\text{th}$ character is ``N'' and right $1$ unit if the $i^\text{th}$ character is ``E''. We may further correspond each such string to an element $I_k(\sigma)$ of $J(R_{k, n - k})$ by taking the set of half-integer lattice points in $[0.5, k-0.5] \times [0.5, n-k-0.5]$ to the right of or below the path. For example, for $n = 7$ and $k = 4$, the permutation $\sigma = 1725364$ satisfies $\pi_4(\sigma) = \text{``ENENENE''}$, with $I_4(\sigma)$ as shown in Figure \ref{figureNEPath} below.

\begin{figure}[ht]
    \captionsetup{width=0.8\textwidth}
\begin{center}
\begin{tikzpicture}[scale=0.8]

\draw[gray!30, thin] (-1,-0.5)--(4,-0.5);
\draw[gray!30, thin] (-1,0.5)--(4,0.5);
\draw[gray!30, thin] (-1,1.5)--(4,1.5);
\draw[gray!30, thin] (-1,2.5)--(4,2.5);
\draw[gray!30, thin] (-0.5,-1)--(-0.5,3);
\draw[gray!30, thin] (0.5,-1)--(0.5,3);
\draw[gray!30, thin] (1.5,-1)--(1.5,3);
\draw[gray!30, thin] (2.5,-1)--(2.5,3);
\draw[gray!30, thin] (3.5,-1)--(3.5,3);

\draw[red, thick] (-0.5,-0.5)--(0.5,-0.5);
\draw[red, thick] (0.5,-0.5)--(0.5,0.5);
\draw[red, thick] (0.5,0.5)--(1.5,0.5);
\draw[red, thick] (1.5,0.5)--(1.5,1.5);
\draw[red, thick] (1.5,1.5)--(2.5,1.5);
\draw[red, thick] (2.5,1.5)--(2.5,2.5);
\draw[red, thick] (2.5,2.5)--(3.5,2.5);

\filldraw[black] (-0.5, -0.5) circle (1.5pt);
\filldraw[black] (0.5, -0.5) circle (1.5pt);
\filldraw[black] (1.5, -0.5) circle (1.5pt);
\filldraw[black] (2.5, -0.5) circle (1.5pt);
\filldraw[black] (3.5, -0.5) circle (1.5pt);
\filldraw[black] (-0.5, 0.5) circle (1.5pt);
\filldraw[black] (0.5, 0.5) circle (1.5pt);
\filldraw[black] (1.5, 0.5) circle (1.5pt);
\filldraw[black] (2.5, 0.5) circle (1.5pt);
\filldraw[black] (3.5, 0.5) circle (1.5pt);
\filldraw[black] (-0.5, 1.5) circle (1.5pt);
\filldraw[black] (0.5, 1.5) circle (1.5pt);
\filldraw[black] (1.5, 1.5) circle (1.5pt);
\filldraw[black] (2.5, 1.5) circle (1.5pt);
\filldraw[black] (3.5, 1.5) circle (1.5pt);
\filldraw[black] (-0.5, 2.5) circle (1.5pt);
\filldraw[black] (0.5, 2.5) circle (1.5pt);
\filldraw[black] (1.5, 2.5) circle (1.5pt);
\filldraw[black] (2.5, 2.5) circle (1.5pt);
\filldraw[black] (3.5, 2.5) circle (1.5pt);

\draw[red] (0,2) circle (3pt);
\draw[red] (1,2) circle (3pt);
\draw[red] (2,2) circle (3pt);
\draw[red] (0,1) circle (3pt);
\draw[red] (1,1) circle (3pt);
\draw[red] (0,0) circle (3pt);
\filldraw[red] (3,2) circle (3pt);
\filldraw[red] (3,1) circle (3pt);
\filldraw[red] (2,1) circle (3pt);
\filldraw[red] (3,0) circle (3pt);
\filldraw[red] (2,0) circle (3pt);
\filldraw[red] (1,0) circle (3pt);

\node[anchor=north east] (O) at (-0.5, -0.5) {$\textcolor{black}{(0, 0)}$};
\node[anchor=north west] (O) at (3.5, -0.5) {$\textcolor{black}{(4, 0)}$};
\node[anchor=south east] (O) at (-0.5, 2.5) {$\textcolor{black}{(0, 3)}$};

\end{tikzpicture}
\end{center}
    
    \caption{For $\sigma = 1725364 \in S_7$, we have $I_4(\sigma) = \text{``ENENENE''}$. The corresponding path in $\mathbb{Z}^2$ is shown in \textcolor{red}{red}, and the corresponding order ideal $I_4(\sigma) \subset J(R_{4, 3})$ is the subset of \textcolor{red}{red} filled in points above.}
    \label{figureNEPath}
\end{figure}
    
    Now, we will prove the strengthened claim that for all $\sigma \in S_n$ and nonnegative integers $t$,
    \begin{align*}
        \mathbb{P}(t_{k, \sigma} \geq t) &\leq \mathbb{P}(T_{J(R_{k, n-k})}(I_k(\sigma)) \geq t).
    \end{align*}

    Note that this is indeed stronger than the original claim by Corollary \ref{corDistribSubPoset}. We prove this by induction on $t$; the claim is trivial for $t = 0$, so we prove the claim for some $t$ assuming it is true for all $t' < t$. Now consider any permutation $\sigma$. Note that any pair of consecutive elements $(x, x + 1) \in [n]^2$ with $\sigma(x) > k$ and $\sigma(x+1) \leq k$ forms a descent, whilst any pair of consecutive elements $(x, x+1) \in [n]^2$ with $\sigma(x) \leq k$ and $\sigma(x+1) > k$ does not form a descent. Call all such former pairs \dfn{critical}. Now, for any subset $S'$ of the set of critical pairs $S$, let $\sigma_{S'}$ be the permutation achieved when only the pairs in $S'$ are swapped in $\sigma$. Define $I := I_k(\sigma)$ and $I_{S'} := I_{k}(\sigma_{S'})$ for any subset $S' \subset S$. Moreover for any permutation $\tau$ let $I_{\tau} := I_{k}(\tau)$, and for any permutation $\tau$ achieved by applying an Ungar move to $\sigma$, let $S_{\tau} \subset S$ denote the subset of critical pairs which were chosen to be swapped in that Ungar move. Then, by the above properties it is clear that:
    \begin{enumerate}
        \item The elements that $I$ can transition to in the Ungarian Markov chain $\textbf{U}_{J(R_{k, n-k})}$ are precisely the elements $I_{S'}$ for the subsets $S' \subset S$.
        \item Given any subset $S' \subset S$, the probability that $\sigma$ transitions in $\textbf{U}_{S_n}$ to some permutation $\tau$ with $S_\tau = S'$ is precisely the same as the probability that $I$ transitions to $I_{S'}$ in $J(R_{k, n-k})$.
        \item For any $\tau$ satisfying $S_{\tau} = S'$, we have that $\tau \leq \sigma_{S'}$ in $S_n$, hence $I_{\tau} \leq I_{S'}$ in $J(R_{k, n-k})$. Thus, by Corollary \ref{corDistribSubPoset} we have that for all $t \geq 0$, $\mathbb{P}(T_{J(R_{k, n-k})}(I_{\tau}) \geq t) \leq \mathbb{P}(T_{J(R_{k, n-k})}(I_{S'}) \geq t)$.
    \end{enumerate}
    Now, for any permutation $\sigma'$, let $Q_{\sigma'}$ be the set of permutations that can be obtained by applying a (possibly trivial) Ungar move to $\sigma'$. We now have:
    \begin{align}
        \mathbb{P}(t_{k, \sigma} \geq t) &= \sum_{S' \subset S} \sum_{\substack{\tau \in Q_{\sigma} \\S_\tau = S'}} \mathbb{P}(\sigma \rightarrow \tau) \mathbb{P}(t_{k, \tau} \geq t-1) \nonumber\\
        &\leq \sum_{S' \subset S} \sum_{\substack{\tau \in Q_{\sigma} \\S_\tau = S'}} \mathbb{P}(\sigma \rightarrow \tau) \mathbb{P}(T_{J(R_{k, n-k})}(I_\tau) \geq t-1)\tag{Inductive Hypothesis}\\
        &\leq \sum_{S' \subset S} \mathbb{P}(T_{J(R_{k, n-k})}(I_{S'}) \geq t-1) \sum_{\substack{\tau \in Q_{\sigma} \\S_\tau = S'}} \mathbb{P}(\sigma \rightarrow \tau) \tag{Observation 3}\\
        &\leq \sum_{S' \subset S} \mathbb{P}(T_{J(R_{k, n-k})}(I_{S'}) \geq t-1) \mathbb{P}(I \rightarrow I_{S'}) \tag{Observation 2}\\
        &= \mathbb{P}(T_{J(R_{k, n-k})}(I) \geq t). \tag{Observation 1}
    \end{align}
    This thus proves the inductive step and hence the claim.
\end{proof}

We now prove Theorem \ref{thmBuffedMainTheorem1}. Consider any odd $m \in \mathbb{Z}^+$, and for all integers $c$ let $a_{n, m, c} = \lceil \frac{c(n-1)}{m} \rceil$. Suppose $n > 2m$; then evidently for each $1 \leq k \leq n-1$ there exists a unique integer $1 \leq c \leq m$ such that
\begin{align*}
     a_{n, m, c-1} + 1 \leq k \leq a_{n, m, c}.
\end{align*}
Note that we always have $a_{n,m,c} + a_{n,m,m-c} \geq n-1$, so the above inequality implies $n - k \leq a_{n, m, m-c+1}$. Hence for all $k$ within the range above, we have that
\begin{align*}
        \mathbb{P}(t_k \geq t) &\leq \mathbb{P}(T(J(R_{k, n-k}))) \geq t) \\
        &\leq \mathbb{P}(T(J(R_{a_{n, m, c}, a_{n, m, m - c+1}}))) \geq t),
\end{align*}
since $J(R_{k, n-k})$ is a subposet of $J(R_{a_{n, m, c}, a_{n, m, m - c+1}})$.

Now, fix $c$, $m$, and $p$, and recall the definitions of $\Phi_p$ and $\eta_p$ from Theorem \ref{thmMultiCornerConvergence}. For $c, m, p$ fixed, we have that $\eta_p(a_{n, m, c}, a_{n, m, m - c + 1}) = \Theta_{c, m, p}(n^{1/3})$, where the implied constant possibly depends on $c$, $m$, and $p$. By taking the maximum of this constant across all choices of $c \in [m]$, we may assume the constant depends only on $m, p$. Now, we have that for all positive $x, y$,
\begin{align}
    \Phi_p(x, y) = \frac{1}{p}(x + y + 2\sqrt{(1-p)xy}) &\leq \frac{(x+y)(1 + \sqrt{1-p})}{p} \tag{AM-GM}.
\end{align}
Hence,
\begin{align*}
    \Phi_p(a_{n, m, c}, a_{n, m, m-c+1}) &\leq \bigg((n-1)\bigg(\frac{m+1}{m}\bigg)+2\bigg)\bigg(\frac{1+\sqrt{1-p}}{p}\bigg).
\end{align*}
Now run $\textbf{U}_{S_n}$ starting at $\sigma$. For each $k \in [n-1]$, let $A_k$ be the event that
\begin{align*}
    t_k > \bigg((n-1)\bigg(\frac{m+1}{m}\bigg)+2\bigg)\bigg(\frac{1+\sqrt{1-p}}{p}\bigg) + \sqrt{n}.
\end{align*}
Then, by Theorem \ref{thmMultiCornerConvergence} and Theorem \ref{thmTracyWidomApprox}, we find that for all sufficiently large $n$, there exists constants $c_{1, m, p}, c_{2, m, p}$, possibly dependent on $m, p$ such that:
\begin{align*}
    \mathbb{P}(A_k) &\leq \frac{c_{1, m, p}}{n^{1/4}}e^{-c_{2, m, p}n^{1/4}}.
\end{align*}
Hence, by the union bound,
\begin{align}\label{eq1}
    \mathbb{P}\bigg(\bigcup_{i=1}^{n-1} A_i\bigg) &\leq c_{1, m, p}n^{3/4}e^{-c_{2, m, p}n^{1/4}}.
\end{align}
Let $A = \bigcup_{i=1}^{n-1} A_i$, and let $B$ be the event that $T_{S_n}(\sigma) > n^2/p$. Observe that if for all $1 \leq k \leq n-1$, the elements less than or equal to $k$ are all left of the elements greater than $k$, then the current permutation must be the identity. Hence, $\neg A$ implies that $T_{S_n}(\sigma) \leq ((n-1)(1+1/m)+2)(1+\sqrt{1-p})/p + \sqrt{n}$, and thus for all sufficiently large $n$, $B \subset A$. To estimate the contribution of $B$ to $\E_{\sigma}(S_n)$, we now use the following lemma, first proved (with some algebraic modifications) in \cite[Section 3]{defant2023ungarian}.
\begin{lemma}\label{lemmaBBoundSn}
    \cite[Section 3]{defant2023ungarian} We have that
    \begin{align*}
        \mathbb{P}(B)\mathbb{E}(T_{S_n}(\sigma) | B) = O_p(n^2e^{-n^2p/6}).
    \end{align*}
\end{lemma}
\begin{proof}
    Recall that $S_n$ is a graded poset where every maximal chain is of length $\binom{n}{2} < \frac{n^2}{2}$. Hence, any chain from $\sigma$ to the identity permutation has less than $\frac{n^2}{2}$ elements. Each step, the probability that we transition to a strictly lower element in the lattice is at least $p$, so we evidently have that $T_{S_n}(\sigma) \leq \sum_{i=1}^{\binom{n}{2}} G_i$, where each $G_i$ is an independent geometric random variable with parameter $p$. Applying Lemma \ref{lemmaSumGeomBound} (a) with $k = \frac{n^2}{2}$ and $t = \frac{z-n^2/2p}{n/\sqrt{2p^3}}$, we find that for any integer $z > n^2/p$,
    \begin{align*}
        \mathbb{P}(T_{S_n}(\sigma) \geq z) &\leq \mathbb{P}\bigg(\sum_{i=1}^{\binom{n}{2}}G_i \geq z \bigg) \leq \exp\bigg(-\frac{(zp-n^2/2)^2}{2zp}\bigg) \leq \exp\bigg(-\frac{zp}{8}\bigg).
    \end{align*}
    Hence we find that
\begin{align*}
    \mathbb{P}(B)\mathbb{E}(T_{S_n}(\sigma) | B) &= \sum_{z = \lfloor n^2/p \rfloor + 1}^{\infty} z \mathbb{P}(T_{S_n}(\sigma) = z) \\
    &\leq \sum_{z = \lfloor n^2/p \rfloor+1}^{\infty} z\mathbb{P}(T_{S_n}(\sigma) \geq z) \\
    &\leq \sum_{z = \lfloor n^2/p \rfloor+1}^{\infty} ze^{-zp/8}\\
    &\leq \int_{n^2/p}^\infty xe^{-xp/8} dx\\
    &= O_p(n^2e^{-n^2/8}),
\end{align*}
as desired.
\end{proof} 
Applying Lemma \ref{lemmaBBoundSn} and Equation \ref{eq1}, we now find that for all sufficiently large $n$,
\begin{align*}
    \E_{S_n}(\sigma) &= \mathbb{P}(\neg A)\mathbb{E}(T_{S_n}(\sigma) | \neg A) + \mathbb{P}(A - B)\mathbb{E}(T_{S_n}(\sigma) | A - B) + \mathbb{P}(B)\mathbb{E}(T_{S_n}(\sigma) | B) \\
    &\leq \bigg(\bigg((n-1)\bigg(\frac{m+1}{m}\bigg)+2\bigg)\bigg(\frac{1+\sqrt{1-p}}{p}\bigg) + \sqrt{n}\bigg) + c_{1, m, p}n^{3/4}e^{-c_{2, m, p}n^{1/4}}\bigg(\frac{n^2}{p}\bigg) + O_p(n^2e^{-n^2/8}) \\
    &= \bigg(\frac{1 + 1\sqrt{1-p}}{p}\bigg)\bigg(\frac{n(m+1)}{m}\bigg) + o_{p,m}(n).
\end{align*}
The above statement is true for all odd $m \in \mathbb{Z}^+$, so
\begin{align*}
    \E_{S_n}(\sigma) &\leq \bigg(\frac{1 + \sqrt{1-p}}{p}\bigg)n + o_p(n),
\end{align*}
as desired.

\section{Background on Tamari Lattices}\label{sTamariBackground}

In Section \ref{sTamnProof} we will prove Theorem \ref{thmMainTheorem2}. Before we do so it will be useful to reformulate $\Tam_n$ as a poset on the ordered forests on $n$ vertices. In Subsection \ref{ssBackground312Avoiding} we will first give more background on the poset structure of $\Av_n(312)$. Then in Subsection \ref{ssBackgroundTamariTree} we will define a poset $(\mathcal{F}_{\ord}(n), \leq)$ on the ordered forests on $n$ vertices. We will prove $(\mathcal{F}_{\ord}(n), \leq) \cong \Tam_n$ by establishing a poset isomorphism between $(\mathcal{F}_{\ord}(n), \leq)$ and $\Av_n(312)$ in Theorem \ref{thmRecontextualizeTamari}. Note that much of the material of this section (notably, the isomorphism $(\mathcal{F}_{\ord}(n), \leq) \cong \Tam_n$) is based on known results in the literature regarding the Tamari lattice. For further reference on $(\mathcal{F}_{\ord}(n), \leq)$, see \cite[Section 6.2.3, Exercise 32]{donKnuth}.

\subsection{The Sublattice of 312-Avoiding Permutations}\label{ssBackground312Avoiding}

The \dfn{right weak order} on $S_n$ is the partial order where for any permutation $\sigma$ and descent $i \in \Des(\sigma)$, the permutation $\sigma' = \sigma \circ \begin{pmatrix} i & i + 1 \end{pmatrix}$ is covered by $\sigma$. It is well-known that the right weak order on $S_n$ is a lattice \cite{bjornerBrentiCoxeterGroupBook}; thus from now on we will use $S_n$ to implicitly denote this lattice. Given a permutation $\sigma \in S_n$, say a triple of indices $(i_1, i_2, i_3) \subset [n]^3$ with $i_1 < i_2 < i_3$ forms of \dfn{$312$-pattern} with respect to $\sigma$ if $\sigma(i_1) > \sigma(i_3) > \sigma(i_2)$. Call a permutation $\sigma \in S_n$ \dfn{$312$-avoiding} if no triple of indices forms a $312$-pattern with respect to $\sigma$. Let $\Av_n(312)$ denote the subset of $S_n$ consisting of the $312$-avoiding permutations, and let $\Av_n(312)$ inherit a poset structure from $S_n$. Then as discussed in Subsection \ref{ssIntroLattice}, $\Av_n(312)$ is a sublattice of $S_n$, and $\Av_n(312) \cong \Tam_n$.

Now, to further describe the poset structure on $\Av_n(312)$, we will consider the following \dfn{projection operator} $\pi_\downarrow : S_n \rightarrow \Av_n(312)$, first defined by Defant in \cite{defant2022tamari}. Given any permutation $\sigma \in S_n$, if there exist indices $i, j$ such that $i + 1 < j$ and $\sigma(i + 1) < \sigma(j) < \sigma(i)$ we can perform an \dfn{allowable swap} by swapping the entries of $\sigma(i)$ and $\sigma(i+1)$. Then $\pi_\downarrow$ sends $\sigma$ to the permutation obtained by starting at $\sigma$ and repeatedly applying allowable swaps until no more can be performed. Clearly the resulting permutation $\pi_\downarrow(\sigma)$ is in $\Av_n(312)$, and as shown in \cite[Section 3]{defant2022tamari}, the resulting permutation $\pi_\downarrow(\sigma)$ is well-defined and independent of the order in which we perform allowable swaps. In fact, Defant proved the following commutation relation between $\pi_\downarrow$ and the meet operation $\wedge$:
\begin{lemma}[Lemma 3.1, \cite{defant2022tamari}]\label{lemmaProjectionMeet}
    Consider any positive integer $n$. Given any subset $T \subset S_n$, we have that
    \begin{align*}
        \pi_\downarrow\Big(\bigwedge T \Big) = \bigwedge \{\pi_\downarrow(\sigma) : \sigma \in T\}.
    \end{align*}
\end{lemma}
Here both meet operations are taken in $S_n$, but since $\Av_n(312)$ is a sublattice of $S_n$ (as discussed in Subsection \ref{ssthmMainTheorem2Intro}), the meet operation on the right hand side can be taken in $\Av_n(312)$ as well.

Now, the following proposition characterizes how $\pi_\downarrow$ interacts with the covering relation on $S_n$.
\begin{proposition}\label{propProjectionSnCovering}
    Consider any permutation $\sigma \in \Av_n(312)$, and consider a descent $i \in \Des(\sigma)$. Let $j \in [n]$ be the minimal index such that $\sigma(k) \geq \sigma(i)$ for all $k \in [j, i]$. Let $\tau = \sigma \circ (i, i+1)$ be the permutation obtained from $\sigma$ by swapping $\sigma(i)$ and $\sigma(i+1)$. Then
    \begin{align*}
        \pi_\downarrow (\tau) &= \sigma \circ \begin{pmatrix} i+1 & i & \ldots & j+1 & j \end{pmatrix}.
    \end{align*}
\end{proposition}
\begin{proof}
    Note that if $j < i$, then we have that $\sigma(i-1) > \sigma(i)$. So the triplet $(i-1, i, i+1)$ forms a $312$-pattern in $\tau = \sigma \circ (i, i + 1)$, so we may admissably swap $i-1$ and $i$ to obtain $\sigma \circ \begin{pmatrix} i+1 & i  & i-1 \end{pmatrix}$. Continuing inductively, we find that for any $k > j$, the permutation $\sigma \circ \begin{pmatrix} i+1 & i & \ldots & k+1 & k \end{pmatrix}$ contains the $312$-pattern $(k-1, k, i+1)$, so we may admissably swap $k - 1$ and $k$. Hence we may retrieve $\sigma \circ \begin{pmatrix} i+1 & i & \ldots & j+1 & j \end{pmatrix}$ by performing admissible swaps on $\sigma$.

    Now let $\tau_0 = \sigma \circ \begin{pmatrix} i+1 & i & \ldots & j+1 & j \end{pmatrix}$, and suppose $\tau_0$ contains a $312$-pattern $(i_1, i_2, i_3)$. In the above paragraph we were repeatedly swapping the index mapping to $\sigma(i+1)$ with an adjacent index, so any $312$-pattern in $\tau_0$ must contain the index mapping to $\sigma(i+1)$, which is $j$. We also only swapped the indices within the interval $[j, i+1]$, so such a $312$-pattern must also contain another index within this interval. Thus the $312$-pattern must satisfy $i_2 = j$ and $j < i_3 \leq i+1$. By the minimality of $j$ we have that $\tau_0(j-1) = \sigma(j-1) < \sigma(i) \leq \tau(i_3)$, so we must have $i_1 < j-1$. But then we have
    \begin{align*}
        \tau_0(i_1) > \tau_0(i_3) > \tau_0(j-1), \\
        \sigma(i_1) > \tau_0(i_3) > \sigma(j-1).
    \end{align*}
    Since $\sigma^{-1}(\tau_0(i_3)) \in [j, i]$, we have that $(i_1, j-1, \sigma^{-1}(\tau_0(i_3)))$ is a $312$-pattern in $\sigma$, a contradiction. Hence $\tau_0$ is $312$-avoiding, so $\pi_\downarrow (\tau) = \tau_0$ as desired.
\end{proof}

Using Proposition \ref{propProjectionSnCovering}, we may now derive the following relationship between $\cov_{S_n}$ and $\cov_{\Av_n(312)}$.
\begin{lemma}\label{lemmaCoveringBijection}
    Given any permutation $\sigma \in \Av_n(312)$, we have that the map
    \begin{align*}
        \pi_\downarrow : \cov_{S_n}(\sigma) &\rightarrow \cov_{\Av_n(312)}(\sigma) \\
        \tau &\rightarrow \pi_\downarrow(\tau)
    \end{align*}
    is a bijection.
\end{lemma}
\begin{proof}
    We first claim $\pi_\downarrow$ is a well-defined map; i.e. that for any $\tau \in \cov_{S_n}(\sigma)$, we have that $\tau_0 := \pi_\downarrow(\tau)$ is covered by $\sigma$ in $\Av_n(312)$. To do this, observe that any $\tau \in \cov_{S_n}(\sigma)$ is of the form $\tau = \sigma \circ \begin{pmatrix} i + 1 & i \end{pmatrix}$ for some $i \in \Des(\sigma)$. Thus by Proposition \ref{propProjectionSnCovering} we have that $\tau_0 = \sigma \circ \begin{pmatrix} i+1 & i & \ldots & j+1 & j \end{pmatrix}$ for some $j$. Note that the relative ordering of the indices mapping to $[n] - \{\sigma(i+1)\}$ is the same in $\sigma$ and $\tau_0$. Thus any $\tau' \in S_n$ satisfying $\sigma \geq \tau' \geq \tau_0$ in $S_n$ must be obtainable by repeatedly performing swaps on $\sigma$, where one of the swapped indices maps to $\sigma(i+1)$. Notably this implies that the interval $[\tau_0, \sigma] \subset S_n$ is a chain, whose only $312$-avoiding elements are its endpoints. Hence we have that $\tau_0 \in \cov_{\Av_n(312)}(\sigma)$ as desired.

    Now, we first prove $\pi_\downarrow$ is surjective. Indeed given any $\tau_0 \in \cov_{\Av_n(312)}(\sigma)$, we have that since $\tau_0 < \sigma$, there exists some $\tau \in \cov_{S_n}(\sigma)$ such that $\tau_0 \leq \tau$. Now by Lemma \ref{lemmaProjectionMeet}, we have that $\pi_\downarrow(\tau) = \pi_\downarrow(\tau_0 \wedge \tau) = \tau_0 \wedge \pi_\downarrow(\tau)$, so $\pi_\downarrow(\tau) \geq \tau_0$. Since $\tau_0, \pi_\downarrow(\tau) \in \cov_{\Av_n(312)}(\sigma)$, we find that $\pi_\downarrow(\tau) = \tau_0$ as desired.

    We finally show $\pi_\downarrow$ is injective. Note that any two distinct elements $\tau, \tau'$ of $\cov_{S_n}(\sigma)$ must be of the form $\tau = \sigma \circ \begin{pmatrix} i & i+1 \end{pmatrix}$ and $\tau' = \sigma \circ \begin{pmatrix} j & j+1 \end{pmatrix}$ for some distinct $i, j \in \Des(\sigma)$. Now by Proposition \ref{propProjectionSnCovering}, $\pi_\downarrow(\tau)$ is obtainable from $\sigma$ by repeatedly swapping the index mapping to $\sigma(i+1)$ with an adjacent index, while $\pi_\downarrow(\tau')$ is obtainable from $\sigma$ by repeatedly swapping the index mapping to $\sigma(j+1)$ with an adjacent index. This clearly implies that $\pi_\downarrow(\tau) \neq \pi_\downarrow(\tau')$, as desired.
\end{proof}

By combining Lemma \ref{lemmaProjectionMeet} and Lemma \ref{lemmaCoveringBijection}, we may now reinterpret Ungar moves on $\Av_n(312)$ as follows. Given any permutation $\sigma \in \Av_n(312)$ and subset $T \subset \cov_{S_n}(\sigma)$, we have that
\begin{align*}
    \pi_\downarrow \Big(\bigwedge_{S_n(312)}(T \cup \{\sigma\})\Big) = \bigwedge_{\Av_n(312)}\{\pi_\downarrow(\tau) : \tau \in T\} \cup \{\sigma\}.
\end{align*}

Hence for any $\sigma \in \Av_n(312)$, applying a random Ungar move to $\sigma$ in $\Av_n(312)$ is equivalent to applying a random Ungar move to $\sigma$ in $S_n$, then applying the $\pi_\downarrow$ operator.

\subsection{A Graph-Theoretic Formulation of the Tamari Lattice}\label{ssBackgroundTamariTree}

In the proof of Theorem \ref{thmMainTheorem2}, it will be useful to work with a graph-theoretic interpretation of $\Tam_n$. Before we present this interpretation, we will first review some relevant prerequisite information.

A \dfn{rooted tree} is a (graph-theoretic) tree where we specify one vertex to be the ``root.'' Within a rooted tree, we may naturally assign an orientation to each of the edges by having them point ``away'' from the root. If two vertices $v, w$ are connected by an edge pointing towards $w$, then we will call $v$ the \dfn{parent} of $w$, and $w$ the \dfn{child} of $v$. If a vertex has no children, then we call it a \dfn{leaf}. An \dfn{ordered tree} is a rooted tree where we also specify an ordering of the children of each node. We may view ordered trees as trees that we may draw on a plane, such that each parent is ``higher'' than their child, and the ordering of the children of each vertex from left to right is precisely the ordering we have defined. Hence ordered trees are also called plane trees.

A \dfn{forest} is a disjoint union of trees. For any forest $F$, let $V(F)$ denote its set of vertices. An \dfn{ordered forest} is a forest equipped with an ordering on each tree, along with an ordering of the trees of the forest from left to right. A map $\varphi : V(F) \rightarrow V(F')$ is an \dfn{isomorphism of ordered forests} if it is a graph isomorphism which preserves the ordering of the roots and the natural orientations of the edges. We will denote the set of ordered forests on $n$ vertices up to isomorphism as $\mathcal{F}_{\ord}(n)$.

The \dfn{(left-to-right) preorder traversal} of an ordered forest on $n$ vertices is a labeling of the vertices from $1$ to $n$, iteratively assigned as follows.
\begin{enumerate}
    \item We assign the label $1$ to the root of the leftmost tree.
    \item If the label $i$ was just assigned to a non-leaf vertex, we next assign $i + 1$ to its leftmost child.
    \item If the label $j$ was just assigned to a leaf vertex, and $i$ is the largest index less than $j$ which has at least one still-unlabeled child, then we label the leftmost unlabeled child of $i$ with the label $j+1$.
    \item If after assigning the label $i$, all vertices in the tree containing vertex $i$ are labeled, then we next assign $i + 1$ to the root of the tree immediately right of the tree containing vertex $i$.
\end{enumerate}
The \dfn{right-to-left preorder traversal} is defined analogously as above, except the instances of ``right'' and ``left'' are swapped. We will use ``preorder traversal'' to denote the left-to-right preorder traversal, unless indicated otherwise. See Figure \ref{figureTamariPicture1} for an example of an ordered forest with the preorder traversal labeling. Clearly an isomorphism $\varphi : V(F) \rightarrow V(F')$ of ordered forests preserves the preorder traversal labeling. Note that throughout the rest of the paper, we will label the vertices of a forest using the preorder traversal labeling, and we will identify each vertex with its label.

\begin{figure}[ht]
    \captionsetup{width=0.8\textwidth}
    \begin{center}
\begin{tikzpicture}
\filldraw[black] (3.25, 3) circle (1.5pt) node[anchor=south]{$4$};
\filldraw[black] (2,2) circle (1.5pt) node[anchor=south east]{$5$};
\filldraw[black] (1,1) circle (1.5pt) node[anchor=north west]{$6$};
\filldraw[black] (0.5,0) circle (1.5pt) node[anchor=east]{$7$};
\filldraw[black] (2,1) circle (1.5pt) node[anchor=north]{$8$};
\filldraw[black] (3,1) circle (1.5pt) node[anchor=north]{$9$};
\filldraw[black] (4.5,2) circle (1.5pt) node[anchor=south west]{$10$};
\filldraw[black] (4,1) circle (1.5pt) node[anchor=north]{$11$};
\filldraw[black] (5,1) circle (1.5pt) node[anchor=north]{$12$};

\draw[black, thin] (0.5,0)--(1,1);
\draw[black, thin] (1,1)--(2,2);
\draw[black, thin] (2,2)--(3.25,3);
\draw[black, thin] (2,2)--(2,1);
\draw[black, thin] (2,2)--(3,1);
\draw[black, thin] (4.5,2)--(4,1);
\draw[black, thin] (4.5,2)--(5,1);
\draw[black, thin] (3.25,3)--(4.5,2);

\filldraw[black] (-0.5,3) circle (1.5pt) node[anchor=south]{$1$};
\filldraw[black] (-1.5,2) circle (1.5pt) node[anchor=north]{$2$};
\filldraw[black] (0.5,2) circle (1.5pt) node[anchor=north]{$3$};

\draw[black, thin] (-0.5,3)--(0.5,2);
\draw[black, thin] (-0.5,3)--(-1.5,2);

\end{tikzpicture}
\end{center}
    
    \caption{An example of an ordered forest in $\mathcal{F}_{\text{ord}}(12)$. The labels in the diagram are given by the preorder traversal.}
    \label{figureTamariPicture1}
\end{figure}

We now define a partial order $\leq$ on $\mathcal{F}_{\ord}(n)$. We will later show that the poset $(\mathcal{F}_{\ord}(n), \leq)$ is isomorphic to $\Tam_n$. To do so we first define an \dfn{operation} on vertices of an ordered forest as follows.
\begin{definition}
    Consider any ordered forest $F$, and a vertex $v$ of $F$. If $v$ has a parent, denote it with $w$, and if $v$ has any children, denote its rightmost child with $v'$. An \dfn{operation} on $v$ maps $F$ to a new ordered forest, defined as follows.
    \begin{enumerate}
        \item If $v$ is a leaf, the operation maps $F$ to itself.
        \item Assume $v$ is not a leaf (so $v'$ exists).
        \begin{enumerate}
            \item If $w$ exists, delete the edge from $v \rightarrow v'$, and draw a new edge from $w \rightarrow v'$, such that $v'$ is immediately right of $v$ as a child of $w$.
            \item If $w$ does not exist, delete the edge from $v \rightarrow v'$. Then $v'$ and its descendants form a new tree within the forest; order this tree to be immediately right of the tree containing $v$.
        \end{enumerate}
    \end{enumerate}
\end{definition}

Given any forest $F$ and vertex $v$ of $F$, we will denote the forest obtained by operating on $v \in F$ by $F[v]$. See Figure \ref{figureTamariPicture2} for some examples of the effects of various operations on a given ordered forest. Observe that the preorder traversal labeling is preserved under any operation. We now define the poset structure on $\mathcal{F}_{\ord}(n)$ as follows.

\begin{figure}[ht]
    \captionsetup{width=0.8\textwidth}
\begin{tabular}{ | m{5cm} | m{5cm}| m{5cm}| } 
\hline
\begin{center}
\begin{tikzpicture}[scale=0.8]

\filldraw[black] (3.25,3) circle (1.5pt) node[anchor=south]{$1$};
\filldraw[black] (2,2) circle (1.5pt) node[anchor=south east]{$2$};
\filldraw[black] (1,1) circle (1.5pt) node[anchor=south east]{$3$};
\filldraw[black] (0.5,0) circle (1.5pt) node[anchor=east]{$4$};
\filldraw[black] (2,1) circle (1.5pt) node[anchor=north]{$5$};
\filldraw[black] (3,1) circle (1.5pt) node[anchor=north]{$6$};
\filldraw[black] (4.5,2) circle (1.5pt) node[anchor=south west]{$7$};
\filldraw[black] (4,1) circle (1.5pt) node[anchor=north]{$8$};
\filldraw[black] (5,1) circle (1.5pt) node[anchor=north]{$9$};

\draw[black, thin] (0.5,0)--(1,1);
\draw[black, thin] (1,1)--(2,2);
\draw[black, thin] (2,2)--(3.25,3);
\draw[black, thin] (2,2)--(2,1);
\draw[black, thin] (2,2)--(3,1);
\draw[black, thin] (4.5,2)--(4,1);
\draw[black, thin] (4.5,2)--(5,1);
\draw[black, thin] (3.25,3)--(4.5,2);

\end{tikzpicture}
\end{center}
 & 
 \begin{center}

\begin{tikzpicture}[scale=0.8]

\filldraw[black] (3.25,3) circle (1.5pt) node[anchor=south]{$1$};
\filldraw[black] (2,2) circle (1.5pt) node[anchor=south east]{$2$};
\filldraw[black] (1.5,1) circle (1.5pt) node[anchor=south east]{$3$};
\filldraw[black] (1,0) circle (1.5pt) node[anchor=east]{$4$};
\filldraw[black] (2.5,1) circle (1.5pt) node[anchor=north]{$5$};
\filldraw[black] (3.25,2) circle (1.5pt) node[anchor=north]{$6$};
\filldraw[black] (4.5,2) circle (1.5pt) node[anchor=south west]{$7$};
\filldraw[black] (4,1) circle (1.5pt) node[anchor=north]{$8$};
\filldraw[black] (5,1) circle (1.5pt) node[anchor=north]{$9$};

\draw[black, thin] (1,0)--(1.5,1);
\draw[black, thin] (1.5,1)--(2,2);
\draw[black, thin] (2,2)--(3.25,3);
\draw[black, thin] (2,2)--(2.5,1);
\draw[black, thin] (3.25,3)--(3.25,2);
\draw[black, thin] (4.5,2)--(4,1);
\draw[black, thin] (4.5,2)--(5,1);
\draw[black, thin] (3.25,3)--(4.5,2);

\end{tikzpicture}
 \end{center}
& \begin{center}

\begin{tikzpicture}[scale=0.8]

\filldraw[black] (2,3) circle (1.5pt) node[anchor=south]{$1$};
\filldraw[black] (2,2) circle (1.5pt) node[anchor=south east]{$2$};
\filldraw[black] (1,1) circle (1.5pt) node[anchor=south east]{$3$};
\filldraw[black] (0.5,0) circle (1.5pt) node[anchor=east]{$4$};
\filldraw[black] (2,1) circle (1.5pt) node[anchor=north]{$5$};
\filldraw[black] (3,1) circle (1.5pt) node[anchor=north]{$6$};
\filldraw[black] (4.5,3) circle (1.5pt) node[anchor=south west]{$7$};
\filldraw[black] (4,2) circle (1.5pt) node[anchor=north]{$8$};
\filldraw[black] (5,2) circle (1.5pt) node[anchor=north]{$9$};

\draw[black, thin] (0.5,0)--(1,1);
\draw[black, thin] (1,1)--(2,2);
\draw[black, thin] (2,2)--(2,3);
\draw[black, thin] (2,2)--(2,1);
\draw[black, thin] (2,2)--(3,1);
\draw[black, thin] (4.5,3)--(4,2);
\draw[black, thin] (4.5,3)--(5,2);

\end{tikzpicture}
\end{center}
\\ 
\hline
\end{tabular}
    
    \caption{The tree on the left is labelled via the preorder traversal. The middle tree is obtained by operating on vertex $2$ of the left tree, and the right tree is obtained by operating on vertex $1$ of the left tree. Note that the preorder traversal labeling is preserved in both operations.}
    \label{figureTamariPicture2}
\end{figure}

\begin{definition}\label{defForestPoset}
    Consider any two ordered forests $F, F' \in \mathcal{F}_{\ord}(n)$.
    \begin{itemize}
        \item Say that $F' \lessdot F$ if we may obtain $F'$ by starting with $F$ and operating on any one of its non-leaf vertices.
        \item Say that $F' \leq F$ if there exists a finite sequence of forests $F' = F_1, F_2, \ldots F_{k-1}, F_k = F$ of length at least one, such that $F_i \in \mathcal{F}_{\ord}(n)$ for all $i \in [k]$, and $F_i \lessdot F_{i+1}$ for all $i \in [k-1]$.
    \end{itemize}
\end{definition}

\begin{theorem}
    As defined above, $\leq$ is indeed a valid partial order on $\mathcal{F}_{\ord}(n)$. Moreover, $\lessdot$ is precisely the covering relation corresponding to $\leq$.
\end{theorem}
\begin{proof}
    Consider an ordered forest $F$. Label its vertices via the preorder traversal on $F$. Given any vertex $i \in V(F)$, let $d(i)$ denote the number of descendants of $i$. Note that operating on a non-leaf vertex $i$ decreases $d(i)$ but does not change $d(j)$ for any other vertex $j \in F$. This implies that $\leq$ is antisymmetric, since if there exists a sequence of ordered forests $F_1 \gtrdot F_2 \gtrdot \ldots \gtrdot F_m$ where $F_m = F_1$, then the sums $\sum_{i \in V(F_k)} d(i)$ must be strictly decreasing, so we must have $m = 1$. By definition $\leq$ satisfies reflexivity and transitivity, so $\leq$ is indeed a valid partial order.

    We now show that $\lessdot$ is the covering relation corresponding to $\leq$. It suffices to show that given any sequence of ordered forests on $n$ vertices $F_1, \ldots F_m$ satisfying $F_1 \gtrdot F_2 \gtrdot \ldots \gtrdot F_m$ and $F_1 \gtrdot F_m$, we have that $m = 2$. Again, since operating on a vertex $i \in V(F)$ decreases $d(i)$ but not $d(j)$ for any $j \neq i$, we have that for all $k \in [m-1]$, $F_{k+1}$ must be obtained from $F_k$ by operating on the same vertex $i$ (according to the preorder traversal labelling). But there exists a unique forest which one may obtain by operating on $i \in V(F_1)$, so we must have $F_m = F_2$, hence $m = 2$ as desired.
\end{proof}

We now prove that the poset $(\mathcal{F}_{\ord}(n), \leq)$ is isomorphic to $\Tam_n$.

\begin{theorem}\label{thmRecontextualizeTamari}
    The poset $(\mathcal{F}_{\ord}(n), \leq)$ is isomorphic to $\Tam_n$.
\end{theorem}

\begin{proof}[Proof of Theorem \ref{thmRecontextualizeTamari}]
    Recall that the sublattice $\Av_n(312) \subset S_n$ of $312$-avoiding permutations is isomorphic to $\Tam_n$ as a lattice. It thus suffices to construct a poset isomorphism
    \begin{align*}
        \Phi : \Av_n(312) \rightarrow (\mathcal{F}_{\ord}(n), \leq).
    \end{align*}
    We will construct $\Phi$ as follows. Given a $312$-avoiding permutation $\sigma$, consider its plot; that is, the set of points $p_i = (i, \sigma(i))$ in $\mathbb{R}^2$. Now draw a graph with vertices $\{p_i\}$, such that given any two integers $i, j \in [n]$ with $i < j$, the points $p_i$ and $p_j$ are connected by a directed edge $p_j \rightarrow p_i$ if
\begin{itemize}
    \item $\sigma(i) > \sigma(j)$, and
    \item no points in the plot (other than $p_i, p_j$ themselves) are contained within the rectangle $[i, j] \times [\sigma(i), \sigma(j)]$.
\end{itemize}
Observe that since $\sigma$ is $312$-avoiding, the in-degree of any vertex in this graph is at most $1$. Since this graph clearly cannot have a directed cycle, it cannot have an undirected cycle either. So this graph is a forest. Now, let $\Phi(\sigma)$ be a plane forest with vertices denoted $q_1, \ldots, q_n$, satisfying the following properties:
\begin{itemize}
    \item vertices $q_i, q_j$ are connected by an edge if and only if $p_i, p_j$ are connected by a (directed) edge;
    \item for any two integers $i, j \in [n]$ satisfying $i < j$, if the vertices $q_i, q_j$ are connected by an edge, then $q_i$ lies vertically underneath $q_j$ in the plane;
    \item for any three integers $i, j, k \in [n]$ satisfying $i < j < k$, if the pairs of vertices $(q_i, q_k)$ and $(q_j, q_k)$ are connected by edges, then the vertex $q_i$ lies to the left of the vertex $q_j$.
    \item for any two integers $i, j \in [n]$ satisfying $i < j$, if $q_i$ and $q_j$ are the roots of the trees they are in (i.e. the rectangles $[i, n] \times [1, \sigma(i)]$ and $[j, n] \times [1, \sigma(j)]$ contain only the points $p_i$ and $p_j$ respectively), then the tree rooted at $q_i$ is left of the tree rooted at $q_j$.
\end{itemize}
Since the graph formed by the vertices $p_i$ was a forest, the graph $\Phi(\sigma)$ must be an ordered forest. Evidently such an ordered forest is unique (up to an isomorphism on ordered forests), so $\Phi$ is a well-defined map. As an example, Figure \ref{figureTamariPermutation1} demonstrates how $\Phi$ acts on the permutation $\sigma = 342651 \in S_6$.

\begin{figure}[ht]
    \captionsetup{width=0.95\textwidth}

{\centering
\begin{tabular}{ | m{6cm} | m{6cm}| } 
\hline
\begin{center}
\begin{tikzpicture}[scale=0.8]

\draw[gray, thick] (0.5,0.5)--(6.5,0.5);
\draw[gray!30, thin] (0.5,1.5)--(6.5,1.5);
\draw[gray!30, thin] (0.5,2.5)--(6.5,2.5);
\draw[gray!30, thin] (0.5,3.5)--(6.5,3.5);
\draw[gray!30, thin] (0.5,4.5)--(6.5,4.5);
\draw[gray!30, thin] (0.5,5.5)--(6.5,5.5);
\draw[gray, thick] (0.5,6.5)--(6.5,6.5);
\draw[gray, thick] (0.5,0.5)--(0.5,6.5);
\draw[gray!30, thin] (1.5,0.5)--(1.5,6.5);
\draw[gray!30, thin] (2.5,0.5)--(2.5,6.5);
\draw[gray!30, thin] (3.5,0.5)--(3.5,6.5);
\draw[gray!30, thin] (4.5,0.5)--(4.5,6.5);
\draw[gray!30, thin] (5.5,0.5)--(5.5,6.5);
\draw[gray, thick] (6.5,0.5)--(6.5,6.5);

\filldraw[black] (1,3) circle (1.5pt) node[anchor=south]{$p_1$};
\filldraw[black] (2,4) circle (1.5pt) node[anchor=south east]{$p_2$};
\filldraw[black] (3,2) circle (1.5pt) node[anchor=north east]{$p_3$};
\filldraw[black] (4,6) circle (1.5pt) node[anchor=east]{$p_4$};
\filldraw[black] (5,5) circle (1.5pt) node[anchor=south west]{$p_5$};
\filldraw[black] (6,1) circle (1.5pt) node[anchor=north]{$p_6$};

\node[anchor=north] (O) at (1, 0.5) {$1$};
\node[anchor=north] (O) at (2, 0.5) {$2$};
\node[anchor=north] (O) at (3, 0.5) {$3$};
\node[anchor=north] (O) at (4, 0.5) {$4$};
\node[anchor=north] (O) at (5, 0.5) {$5$};
\node[anchor=north] (O) at (6, 0.5) {$6$};

\node[anchor=east] (O) at (0.5, 1) {$1$};
\node[anchor=east] (O) at (0.5, 2) {$2$};
\node[anchor=east] (O) at (0.5, 3) {$3$};
\node[anchor=east] (O) at (0.5, 4) {$4$};
\node[anchor=east] (O) at (0.5, 5) {$5$};
\node[anchor=east] (O) at (0.5, 6) {$6$};

\draw[<-, black, thin] (1,3)--(3,2);
\draw[<-, black, thin] (2,4)--(3,2);
\draw[<-, black, thin] (3,2)--(6,1);
\draw[<-, black, thin] (4,6)--(5,5);
\draw[<-, black, thin] (5,5)--(6,1);

\end{tikzpicture}
\end{center}
 & 
 \begin{center}

\begin{tikzpicture}[scale=0.9]

\filldraw[black] (0,0) circle (1.5pt) node[anchor=east]{$q_1$};
\filldraw[black] (2,0) circle (1.5pt) node[anchor=west]{$q_2$};
\filldraw[black] (1,2) circle (1.5pt) node[anchor=south east]{$q_3$};
\filldraw[black] (4,0) circle (1.5pt) node[anchor=west]{$q_4$};
\filldraw[black] (4,2) circle (1.5pt) node[anchor=west]{$q_5$};
\filldraw[black] (2,4) circle (1.5pt) node[anchor=south]{$q_6$};

\draw[black, thin] (4,0)--(4,2);
\draw[black, thin] (4,2)--(2,4);
\draw[black, thin] (2,0)--(1,2);
\draw[black, thin] (0,0)--(1,2);
\draw[black, thin] (1,2)--(2,4);

\end{tikzpicture}
 \end{center}
\\ 
\hline
\end{tabular}\par
}
    
    \caption{On the left is the plot of the permutation $\sigma = 342651 \in S_6$, with directed edges drawn between the vertices of the plot as described above. On the right is the forest $\Phi(\sigma)$.}
    \label{figureTamariPermutation1}
\end{figure}

Now, we can establish an inverse correspondence $\varphi : (\mathcal{F}_{\ord}(n), \leq) \rightarrow \Av_n(312)$ as follows. Consider an ordered forest $F \in (\mathcal{F}_{\ord}(n), \leq)$. For any vertex $v \in F$, let $l(v)$ be its left-to-right preorder traversal label, and let $r(v)$ be its right-to-left preorder traversal label. Then let $\varphi(F)$ be the permutation $\sigma$ satisfying that $\sigma(n+1-r(v)) = l(v)$ for all $v \in V$. Since the ordered sets $\{l(v)\}, \{r(v)\}$ are both permutations of $[n]$, we have that $\varphi(F)$ is indeed a permutation. Also note that for any two vertices $v, w \in V(F)$, if $l(v) > l(w)$ and $r(v) > r(w)$, then $w$ must be an ancestor of $v$. Thus there do not exist vertices $v_1, v_2, v_3 \in V(F)$ such that $r(v_1) > r(v_2) > r(v_3)$ but $l(v_1) > l(v_3) > l(v_2)$. This is equivalent to $\varphi(F)$ being $312$-avoiding, so $\varphi$ indeed forms a well-defined map from $(\mathcal{F}_{\ord}(n), \leq)$ to $\Av_n(312)$. Now, one can readily check that $\varphi$ and $\Phi$ are inverses to each other. Thus $\Phi$ is a bijection, as desired. As an example, for the permutation $\sigma = 342651 \in S_6$, Figure \ref{figureTamariPermutation2} shows how $\varphi(\Phi(\sigma)) = \sigma$.

\begin{figure}[ht]
    \captionsetup{width=0.95\textwidth}

{\centering
\begin{tabular}{ | m{6cm} | m{6cm}| } 
\hline
 \begin{center}

\begin{tikzpicture}[scale=0.9]

\filldraw[black] (0,0) circle (1.5pt) 
node[anchor=west]{$\textcolor{red}{6}$}
node[anchor=east]{$\textcolor{blue}{3}$};
\filldraw[black] (2,0) circle (1.5pt)
node[anchor=west]{$\textcolor{red}{5}$}
node[anchor=east]{$\textcolor{blue}{4}$};
\filldraw[black] (1,2) circle (1.5pt)
node[anchor=west]{$\textcolor{red}{4}$}
node[anchor=east]{$\textcolor{blue}{2}$};
\filldraw[black] (4,0) circle (1.5pt)
node[anchor=west]{$\textcolor{red}{3}$}
node[anchor=east]{$\textcolor{blue}{6}$};
\filldraw[black] (4,2) circle (1.5pt)
node[anchor=west]{$\textcolor{red}{2}$}
node[anchor=east]{$\textcolor{blue}{5}$};
\filldraw[black] (2,4) circle (1.5pt)
node[anchor=west]{$\textcolor{red}{1}$}
node[anchor=east]{$\textcolor{blue}{1}$};

\draw[black, thin] (4,0)--(4,2);
\draw[black, thin] (4,2)--(2,4);
\draw[black, thin] (2,0)--(1,2);
\draw[black, thin] (0,0)--(1,2);
\draw[black, thin] (1,2)--(2,4);

\end{tikzpicture}
 \end{center}
 &
\begin{center}
\begin{tikzpicture}[scale=0.9]

\draw[gray, thick] (0.5,0.5)--(6.5,0.5);
\draw[gray!30, thin] (0.5,1.5)--(6.5,1.5);
\draw[gray!30, thin] (0.5,2.5)--(6.5,2.5);
\draw[gray!30, thin] (0.5,3.5)--(6.5,3.5);
\draw[gray!30, thin] (0.5,4.5)--(6.5,4.5);
\draw[gray!30, thin] (0.5,5.5)--(6.5,5.5);
\draw[gray, thick] (0.5,6.5)--(6.5,6.5);
\draw[gray, thick] (0.5,0.5)--(0.5,6.5);
\draw[gray!30, thin] (1.5,0.5)--(1.5,6.5);
\draw[gray!30, thin] (2.5,0.5)--(2.5,6.5);
\draw[gray!30, thin] (3.5,0.5)--(3.5,6.5);
\draw[gray!30, thin] (4.5,0.5)--(4.5,6.5);
\draw[gray!30, thin] (5.5,0.5)--(5.5,6.5);
\draw[gray, thick] (6.5,0.5)--(6.5,6.5);

\filldraw[black] (1,3) circle (3pt);
\filldraw[black] (2,4) circle (3pt);
\filldraw[black] (3,2) circle (3pt);
\filldraw[black] (4,6) circle (3pt);
\filldraw[black] (5,5) circle (3pt);
\filldraw[black] (6,1) circle (3pt);

\node[anchor=north] (O) at (1, 0.5) {$\textcolor{black}{1}$};
\node[anchor=north] (O) at (2, 0.5) {$\textcolor{black}{2}$};
\node[anchor=north] (O) at (3, 0.5) {$\textcolor{black}{3}$};
\node[anchor=north] (O) at (4, 0.5) {$\textcolor{black}{4}$};
\node[anchor=north] (O) at (5, 0.5) {$\textcolor{black}{5}$};
\node[anchor=north] (O) at (6, 0.5) {$\textcolor{black}{6}$};

\node[anchor=north] (O) at (1, 0) {$(7 - \textcolor{red}{6})$};
\node[anchor=north] (O) at (2, -0.5) {$(7 - \textcolor{red}{5})$};
\node[anchor=north] (O) at (3, 0) {$(7 - \textcolor{red}{4})$};
\node[anchor=north] (O) at (4, -0.5) {$(7 - \textcolor{red}{3})$};
\node[anchor=north] (O) at (5, 0) {$(7 - \textcolor{red}{2})$};
\node[anchor=north] (O) at (6, -0.5) {$(7 - \textcolor{red}{1})$};

\node[anchor=east] (O) at (0.5, 1) {$\textcolor{blue}{1}$};
\node[anchor=east] (O) at (0.5, 2) {$\textcolor{blue}{2}$};
\node[anchor=east] (O) at (0.5, 3) {$\textcolor{blue}{3}$};
\node[anchor=east] (O) at (0.5, 4) {$\textcolor{blue}{4}$};
\node[anchor=east] (O) at (0.5, 5) {$\textcolor{blue}{5}$};
\node[anchor=east] (O) at (0.5, 6) {$\textcolor{blue}{6}$};

\draw[gray, dashed, thin] (1,3)--(3,2);
\draw[gray, dashed, thin] (2,4)--(3,2);
\draw[gray, dashed, thin] (3,2)--(6,1);
\draw[gray, dashed, thin] (4,6)--(5,5);
\draw[gray, dashed, thin] (5,5)--(6,1);

\end{tikzpicture}
\end{center}
\\ 
\hline
\end{tabular}\par
}
    
    \caption{On the left is the forest $F = \Phi(342651) \in \mathcal{F}_{\text{ord}}$ from Figure \ref{figureTamariPermutation1}. For each $v \in V(F)$, the left \textcolor{blue}{blue} label is $l(v)$, while the right \textcolor{red}{red} label is $r(v)$. On the right is the permutation $\varphi(\Phi(342651)) = 342651$. To emphasize that $\Phi$ and $\varphi$ are inverse bijections, dashed lines are drawn on the right which indicate how the pairs of vertices in $V(F)$ are mapped.}
    \label{figureTamariPermutation2}
\end{figure}

Now, consider any $312$-avoiding permutation $\sigma$, and any $i \in [n]$. Let $\Phi(\sigma) = F$, and let $i$ correspond to the vertex $q_i \in V(F)$. Observe that $i-1 \in \Des(\sigma)$ if and only if $q_i$ is a non-leaf vertex. Moreover, for any $i \in [n]$ with $i-1 \in \Des(\sigma)$, let $\tau_i := \sigma \circ \begin{pmatrix} i-1 & i \end{pmatrix}$, and let $\sigma[i] := \pi_\downarrow(\tau_i)$. Then by Proposition \ref{propProjectionSnCovering}, we have that any $i$ with $i-1 \in \Des(\sigma)$ satisfies $\Phi(\sigma[i]) = \Phi(\sigma)[q_i]$. Hence for any $\sigma \in \Av_n(312)$, we have that $\Phi$ bijects $\cov_{\Av_n(312)}(\sigma)$ and $\cov_{(\mathcal{F}_{\ord}(n), \leq)}(\Phi(\sigma))$. So $\Phi$ is indeed a poset isomorphism, as desired.
\end{proof}

From now on we will identify $\Tam_n$ with $(\mathcal{F}_{\ord}(n), \leq)$. By combining the identity $\Phi(\sigma[i]) = \Phi(\sigma)[q_i]$ from the proof of Theorem \ref{thmRecontextualizeTamari}, the characterization of Ungar moves on $\Av_n(312)$ in Subsection \ref{ssBackground312Avoiding}, and the characterization of Ungar moves on $S_n$ in Section \ref{sIntro}, we may characterize the Ungar moves on $(\mathcal{F}_{\ord}(n), \leq)$ as follows.

\begin{corollary}\label{corTamnTreeUngarMoves}
    Consider an ordered forest $F \in (\mathcal{F}_{\ord}(n), \leq)$, and identify the vertices of $F$ with their preorder traversal labels. For any vertices $i_1, i_2, \ldots i_m \in V(F)$ whose labels satisfy $i_1 < i_2 < \ldots < i_m$, we have that
    \begin{align*}
        F \wedge \bigwedge_{k=1}^m F[i_k] = (\ldots((F[i_1])[i_2]) \ldots [i_m])
    \end{align*}
    where the right hand side denotes the forest obtained by operating on $i_1$ in $F$, then $i_2$ in $F[i_1]$, and so on.

    Notably, fix a parameter $p \in (0, 1]$, and consider an ordered forest $F \in (\mathcal{F}_{\ord}(n), \leq)$. Pick each vertex of $V(F)$ with independent probability $p$, and let $\{i_1, \ldots i_m\}$ be the list of picked vertices ordered from smallest to largest label. Then a \textit{random Ungar move} on $F$ sends
    \begin{align*}
        F \mapsto (\ldots((F[i_1])[i_2]) \ldots [i_m]).
    \end{align*}
\end{corollary}

Throughout the rest of the paper, we will write that during a given Ungar move, a vertex $v$ \dfn{receives} an operation (or is \dfn{operated on}) if an Ungar move is applied to the given forest and $v$ is one of the vertices selected in said move. Also note that for any ordered forest $F$ and leaf vertex $v \in F$, we have $F[v] = F$. Hence even though every vertex of $F$ may be picked in the formulation of the random Ungar moves on $(\mathcal{F}_{\ord}(n), \leq)$ in Corollary \ref{corTamnTreeUngarMoves}, only the operations on the non-leaf vertices affect the resulting forest.

By Theorem \ref{thmRecontextualizeTamari}, it is clear that the ordered forest corresponding to the maximal element $\hat{1} \in \Tam_n$ is the tree which consists of a path on $n$ vertices, while the ordered forest corresponding to the minimal element $\hat{0} \in \Tam_n$ is the forest which consists of $n$ vertices and no edges. Now, throughout the rest of the paper, we will use the following properties of the aforementioned operations and preorder traversal labelings:
\begin{proposition}\label{propTreeProperties}
Identify the labels of an ordered forest with their preorder traversal labeling. Then the preorder traversal labeling satisfies the following properties.
    \begin{enumerate}[(a)]
    \item The labels of a vertex and its descendants form a contiguous subset of $[n]$.
    \item For any integer $m \in [n]$ and ordered forest $F \in \Tam_n$, let $F\langle m, n \rangle$ be the induced subgraph of $F$ formed by the vertices with indices in $[m, n]$. Given an ordered forest $F \in \Tam_n$, consider the ordered forest $G \in \Tam_{n-m+1}$ isomorphic to $F \langle m, n \rangle$, and consider any other ordered forest $G' \in \Tam_{n-m+1}$. Then the probability that applying a random Ungar move to $F$ produces a forest $F' \in \Tam_n$ satisfying $F'\langle m, n \rangle \cong G'$ is equal to the probability that applying a random Ungar move to $G$ produces $G'$.
    \item Consider any two vertices indexed $k, l$ where $k < l$, and consider an infinite sequence of Ungar moves applied to $\hat{1} \in \Tam_n$. For every vertex $i$, let $h_i$ denote the smallest integer $t$ in which $i$ receives an operation on the $t^\text{th}$ step. Suppose that:
    \begin{align}\label{eqProp49Condition}
        h_k > h_l \qquad \text{and} \qquad \forall i \in [k+1, l-1], h_l \geq h_i.
    \end{align}
    Then $l$ will become a child of $k$ after $h_l$ moves.
\end{enumerate}
\end{proposition}

\begin{proof}[Proof of Proposition \ref{propTreeProperties}]
    Property $(a)$ follows by the definition of the preorder traversal labeling.

    To prove property $(b)$, observe that for any $m \in [n]$ and $i < m$, we have that $F[i]\langle m, n \rangle \cong F\langle m, n \rangle$; in other words, operating on vertex $i$ does not affect $F\langle m, n \rangle$. Thus property (b) follows immediately from combining this observation with the description of random Ungar moves on $F$ given in Corollary \ref{corTamnTreeUngarMoves}.

    We finally prove property $(c)$. Let $F_t$ be the forest obtained after applying the first $t$ Ungar moves to $\hat{1}$. For any ordered forest $F \in \Tam_n$ and pair of integers $i, j \in [n]$ with $i < j$, define the \dfn{rightmost skeleton of $[i, j]$ in $F$,} $R_{i, j}(F)$, to be the maximal subsequence $a_1, a_2, \ldots a_k$ of $i, i+1, \ldots j$, such that:
    \begin{itemize}
        \item we have $a_1 = i$, and;
        \item for any $l \geq 2$, we have that $a_l$ is the rightmost vertex of $a_{l-1}$
    \end{itemize}
    Now, the condition \eqref{eqProp49Condition} guarantees that for $t < h_l$, the rightmost skeleton $R_{i, j}(F_t)$ is only changed when a vertex $i \in [k + 1, l]$ is operated on. Specifically, we have that any vertex $i \in [k + 1, l - 1]$ satisfies $i \in R_{k, l}(F_t)$ for $t \leq h_i - 1$, and $i \not \in R_{k, l}(F_t)$ for $h_i \leq t \leq h_l$. Let $F'$ be the forest obtained from $F_{h_l - 1}$ by only operating on the vertices chosen in the $h_l^\text{th}$ Ungar move with index $< l$. Then the above implies $R_{k, l}(F') = \{k, l\}$, so $l$ must be a child of $k$ in $F'$. Since $F_{h_l}$ may be obtained from $F'$ by operating on vertices with index $\geq l$, we find that $l$ is also a child of $k$ in $F_{h_l}$, as desired.
\end{proof}

\section{Proof of Theorem \ref{thmMainTheorem2}.}\label{sTamnProof}

The proof of Theorem \ref{thmMainTheorem2} is rather technical. Hence we will first outline the proof in Subsection \ref{ssthmMainTheorem2Outline}, then return to the proof of Theorem \ref{thmMainTheorem2} in Subsection \ref{ssthmMainTheorem2Intro}. Note that we will define many random variables in Subsection \ref{ssthmMainTheorem2Outline}, but these variables will not necessarily be used outside of said Subsection. We will also use the standard abbreviation ``w.h.p.'' to denote ``with high probability''. We will use this term loosely and without proof in Subsection \ref{ssthmMainTheorem2Outline} since it is only a proof outline; all the results will be made rigorous in the later parts of Section \ref{sTamnProof}.

\subsection{Proof Outline for Theorem \ref{thmMainTheorem2}}\label{ssthmMainTheorem2Outline}

The strategy for the proof of Theorem \ref{thmMainTheorem2} is as follows. Consider an ordered tree $T \in \Tam_n$, and run the Ungarian Markov chain $\textbf{U}_{\Tam_n}$ starting at $T$. Let the random variable $T_{n, T}$ denote the number of steps until the vertex $1$ has no more children. Now, our main claim (Lemma \ref{lemmaMain}) in the proof of Theorem \ref{thmMainTheorem2} may roughly be stated as follows: if $T$ is a tree with many children, each with many children of their own, then $T_{n, T}$ is at least $n^{1-o(1)}$ with high probability.

Given a tree $T \in \Tam_n$ rooted at the vertex $1$, we may produce a lower bound for $T_{n, T}$ by piecing together the following two heuristics.
\begin{heuristic}\label{heuristic1}
    Consider an integer $k > 0$. Let $i$ be the $k+1^\text{th}$ rightmost child of $1$, let $j$ be the $k^\text{th}$ rightmost child of $1$, and suppose $i$ has at least $k$ children. Simulate the Ungarian Markov chain $\textbf{U}_{\Tam_n}$ starting at $T$. Let the random variable $\tau_0$ denote the number of steps until $j$ is no longer connected to $1$, and let $\tau$ denote the number of steps until $i$ is no longer connected to $1$. Also let $m$ denote the number of times the vertex $i$ has been operated on after $\tau_0$ steps. Since $1$ must have been operated on at least $k$ times after $\tau_0$ steps, we have w.h.p. that $m = \Omega_p(k)$. Hence we have that $\tau - \tau_0 \geq \min(m, k)$, so w.h.p. we have that $\tau - \tau_0 = \Omega_p(k)$. See Figure \ref{figureHeuristic1} for a visual depiction of how $T$ changes after $\tau_0$ and $\tau$ Ungar moves.
\end{heuristic}
\begin{figure}[ht]
    \captionsetup{width=0.9\textwidth}
    \begin{center}
\begin{tikzpicture}[scale=0.8]
\filldraw[black] (-6, 0) circle (1.5pt) node[anchor=south]{$1$};
\filldraw[black] (-7.5, -1) circle (1.5pt) node[anchor=south]{$i$};
\filldraw[black] (-6.75, -1) circle (1.5pt) node[anchor=north]{$j$};
\filldraw[black] (-6, -1) circle (1.5pt);
\filldraw[black] (-4.5, -1) circle (1.5pt);
\filldraw[red] (-8.25, -2) circle (1.5pt);
\filldraw[red] (-6.75, -2) circle (1.5pt);

\node (Dot1) at (-5.25, -1.35) {$\underbrace{\thickspace\thickspace\thickspace \ldots \ldots \thickspace\thickspace\thickspace}_{k-1 \text{ vertices}}$};
\node (Dot2) at (-7.5, -2.35) {$\underbrace{\thickspace\thickspace\thickspace \ldots \ldots \thickspace\thickspace\thickspace}_{\geq k \text{ vertices}}$};

\draw[black, thin] (-6,0)--(-7.5,-1);
\draw[black, thin] (-6,0)--(-6.75,-1);
\draw[black, thin] (-6,0)--(-6,-1);
\draw[black, thin] (-6,0)--(-4.5,-1);
\draw[black, thin] (-7.5, -1)--(-8.25, -2);
\draw[black, thin] (-7.5, -1)--(-6.75, -2);

\node (Arrow) at (-3, -1) {$\Longrightarrow$};
\node (Label) at (-3, -0.7) {$\tau_0 \text{ moves}$};

\filldraw[black] (-1, 0) circle (1.5pt) node[anchor=south]{$1$};
\filldraw[black] (-1.5, -1) circle (1.5pt) node[anchor=east]{$i$};
\filldraw[red] (-0.5, -1) circle (1.5pt);
\filldraw[red] (1, -1) circle (1.5pt);
\filldraw[red] (-2.25, -2) circle (1.5pt);
\filldraw[red] (-0.75, -2) circle (1.5pt);
\filldraw[black] (0, 0) circle (1.5pt) node[anchor=south]{$j$};
\filldraw[black] (0.5, 0) circle (1.5pt);
\filldraw[black] (2, 0) circle (1.5pt);

\node (Dot3) at (-1.5, -2.35) {$\underbrace{\thickspace\thickspace\thickspace \ldots \ldots \thickspace\thickspace\thickspace}_{\geq k-\Omega_p(k) \text{ vertices}}$};
\node (Dot4) at (0.25, -1.35) {$\underbrace{\thickspace\thickspace\thickspace \ldots \ldots \thickspace\thickspace\thickspace}_{\Omega_p(k) \text{ vertices}}$};
\node (Dot5) at (1.25, 0.35) {$\overbrace{\thickspace\thickspace\thickspace \ldots \ldots \thickspace\thickspace\thickspace}^{k-1 \text{ vertices}}$};;

\draw[black, thin] (-1, 0)--(-1.5, -1);
\draw[black, thin] (-1, 0)--(-0.5, -1);
\draw[black, thin] (-1, 0)--(1, -1);
\draw[black, thin] (-1.5, -1)--(-2.25, -2);
\draw[black, thin] (-1.5, -1)--(-0.75, -2);

\node (Arrow2) at (3.5, -1) {$\Longrightarrow$};
\node (Label2) at (3.5, -0.3) {$\tau - \tau_0$};
\node (Label3) at (3.5, -0.7) {$\text{moves}$};

\filldraw[black] (5, -1) circle (1.5pt) node[anchor=south]{$1$};
\filldraw[black] (5.5, -1) circle (1.5pt) node[anchor=south]{$i$};
\filldraw[red] (6.25, -2) circle (1.5pt);
\filldraw[red] (4.75, -2) circle (1.5pt);
\filldraw[red] (6.75, -1.5) circle (1.5pt);
\filldraw[red] (8.25, -1.5) circle (1.5pt);
\filldraw[black] (6.5, -0.5) circle (1.5pt) node[anchor=south]{$j$};
\filldraw[black] (7, -0.5) circle (1.5pt);
\filldraw[black] (8.5, -0.5) circle (1.5pt);

\draw[black, thin] (5.5, -1)--(6.25, -2);
\draw[black, thin] (5.5, -1)--(4.75, -2);

\node (Dot6) at (5.5, -2.35) {$\underbrace{\thickspace\thickspace\thickspace \ldots \ldots \thickspace\thickspace\thickspace}_{\geq k-\Omega_p(k) \text{ vertices}}$};
\node (Dot7) at (7.5, -1.15) {$\overbrace{\thickspace\thickspace\thickspace \ldots \ldots \thickspace\thickspace\thickspace}^{\Omega_p(k) \text{ vertices}}$};
\node (Dot8) at (7.75, -0.15) {$\overbrace{\thickspace\thickspace\thickspace \ldots \ldots \thickspace\thickspace\thickspace}^{k-1 \text{ vertices}}$};

\end{tikzpicture}
\end{center}
    
    \caption{A diagram showing the estimation of $\tau - \tau_0$ employed in Heuristic \ref{heuristic1} Not all edges and vertices of $T$ are shown, only the ones relevant to the Heuristic. As shown, Heuristic \ref{heuristic1} mainly bounds $\tau - \tau_0$ by counting the number of operations $i$ receives in the first $\tau_0$ moves. Here the vertices which are initially children of $i$ are colored \textcolor{red}{red}.}
    \label{figureHeuristic1}
\end{figure}

\begin{heuristic}\label{heuristic2}
    Consider an integer $k > 0$. Let $i$ be the $k+1^\text{th}$ rightmost child of $1$, let $j$ be the $k^\text{th}$ rightmost child of $1$, and suppose $i$ has at least $k^2$ children. Note that:
    \begin{itemize}
        \item so long as $i$ still has at least one child and $i$ is still a child of $1$, we have that operating on $i$ increases the number of children of $1$ to the right of $i$ by one;
        \item so long as $i$ is still a child of $1$, we have that operating on $1$ decreases the number of children of $1$ to the right of $i$ by one.
    \end{itemize}
    Now simulate $\textbf{U}_{\Tam_n}$ starting at $T$. Let the random variable $\tau_0$ denote the number of moves until vertex $i$ is no longer connected to vertex $j$, and let $\tau$ denote the number of moves until vertex $1$ is no longer connected to vertex $i$. For each $t \in \mathbb{Z}^+$, let the random variables $M(t)$ and $N(t)$ respectively denote the number of operations which vertices $1$ and $i$ receive after $t$ moves. Let the random variable $\tau'$ denote the smallest $t > \tau_0$ such that $(M(t) - M(\tau_0)) - (N(t) - N(\tau_0)) \geq N(\tau_0)$. Then we have:
    \begin{itemize}
        \item after $\tau_0$ moves, there are at least $\min(N(\tau_0), k^2)$ children of $1$ to the right of $i$;
        \item if after $t > \tau_0$ steps $i$ is still a child of $1$ and $i$ still has children, then the number of children of $1$ is at least:
        \begin{align*}
            N(\tau) - (M(t) - M(\tau_0));
        \end{align*}
        \item since vertex $1$ must be operated on at least $k$ times after $\tau_0$ moves, we have that $\tau_0 \geq k$, and hence w.h.p. we have $N(\tau_0) = \Omega_p(k)$;
        \item conditioned on any given value of $\tau_0$, we have that $(M(t) - M(\tau_0)) - (N(t) - N(\tau_0))$ follows a lazy simple random walk process for $t \geq \tau_0$. Hence by Lemma \ref{lemmaRandomWalk} we have w.h.p. that $\tau' = \Omega_p(k^2)$.
    \end{itemize}
    Now, after $\tau$ steps, we must either have that $N(\tau) - (M(t) - M(\tau_0)) \leq 0$ or that the vertex $i$ no longer has children. The second scenario implies that each of the $\geq k^2$ original children of $i$ were at one point a child of $1$ right of $i$ and left of $j$. Thus $\tau \geq \min(\tau_0 + \tau', k^2 + \tau_0)$, so w.h.p. we have that $\tau - \tau_0 = \Omega_p(k^2)$. See Figure \ref{figureHeuristic2} for a visual depiction of how $T$ changes after $\tau_0$ and $\tau$ Ungar moves; notably, observe that the bound on $\tau - \tau_0$ is improved by incorporating the effect of the operations $i$ receives after the $\tau_0^\text{th}$ move.
\end{heuristic}
\begin{figure}[ht]
    \captionsetup{width=0.9\textwidth}
    \begin{center}
\begin{tikzpicture}[scale=0.8]
\filldraw[black] (-6, 0) circle (1.5pt) node[anchor=south]{$1$};
\filldraw[black] (-7.5, -1) circle (1.5pt) node[anchor=south]{$i$};
\filldraw[black] (-6.75, -1) circle (1.5pt) node[anchor=north]{$j$};
\filldraw[black] (-6, -1) circle (1.5pt);
\filldraw[black] (-4.5, -1) circle (1.5pt);
\filldraw[red] (-8.25, -2) circle (1.5pt);
\filldraw[red] (-6.75, -2) circle (1.5pt);

\node (Dot1) at (-5.25, -1.35) {$\underbrace{\thickspace\thickspace\thickspace \ldots \ldots \thickspace\thickspace\thickspace}_{k-1 \text{ vertices}}$};
\node (Dot2) at (-7.5, -2.35) {$\underbrace{\thickspace\thickspace\thickspace \ldots \ldots \thickspace\thickspace\thickspace}_{\geq k^2 \text{ vertices}}$};

\draw[black, thin] (-6,0)--(-7.5,-1);
\draw[black, thin] (-6,0)--(-6.75,-1);
\draw[black, thin] (-6,0)--(-6,-1);
\draw[black, thin] (-6,0)--(-4.5,-1);
\draw[black, thin] (-7.5, -1)--(-8.25, -2);
\draw[black, thin] (-7.5, -1)--(-6.75, -2);

\node (Arrow) at (-3, -1) {$\Longrightarrow$};
\node (Label) at (-3, -0.7) {$\tau_0 \text{ moves}$};

\filldraw[black] (-1, 0) circle (1.5pt) node[anchor=south]{$1$};
\filldraw[black] (-1.5, -1) circle (1.5pt) node[anchor=east]{$i$};
\filldraw[red] (0, -1) circle (1.5pt);
\filldraw[red] (1.5, -1) circle (1.5pt);
\filldraw[red] (-2.25, -2) circle (1.5pt);
\filldraw[red] (-0.75, -2) circle (1.5pt);
\filldraw[black] (0, 0) circle (1.5pt) node[anchor=south]{$j$};
\filldraw[black] (0.5, 0) circle (1.5pt);
\filldraw[black] (2, 0) circle (1.5pt);

\node (Dot3) at (-1.5, -2.35) {$\underbrace{\thickspace\thickspace\thickspace \ldots \ldots \thickspace\thickspace\thickspace}_{\geq k^2-\Omega_p(k) \text{ vertices}}$};
\node (Dot4) at (0.75, -1.35) {$\underbrace{\thickspace\thickspace\thickspace \ldots \ldots \thickspace\thickspace\thickspace}_{\Omega_p(k) \text{ vertices}}$};
\node (Dot5) at (1.25, 0.35) {$\overbrace{\thickspace\thickspace\thickspace \ldots \ldots \thickspace\thickspace\thickspace}^{k-1 \text{ vertices}}$};
\node (BigArrow) at (-0.75, -1) {$\textcolor{red}{\Nearrow}$};

\draw[black, thin] (-1, 0)--(-1.5, -1);
\draw[black, thin] (-1, 0)--(0, -1);
\draw[black, thin] (-1, 0)--(1.5, -1);
\draw[black, thin] (-1.5, -1)--(-2.25, -2);
\draw[black, thin] (-1.5, -1)--(-0.75, -2);

\node (Arrow2) at (3.5, -1) {$\Longrightarrow$};
\node (Label2) at (3.5, -0.3) {$\tau - \tau_0$};
\node (Label3) at (3.5, -0.7) {$\text{moves}$};

\filldraw[black] (5, -1) circle (1.5pt) node[anchor=south]{$1$};
\filldraw[black] (5.5, -1) circle (1.5pt) node[anchor=south]{$i$};
\filldraw[red] (6.25, -2) circle (1.5pt);
\filldraw[red] (4.75, -2) circle (1.5pt);
\filldraw[red] (6.75, -1.5) circle (1.5pt);
\filldraw[red] (8.25, -1.5) circle (1.5pt);
\filldraw[black] (6.5, -0.5) circle (1.5pt) node[anchor=south]{$j$};
\filldraw[black] (7, -0.5) circle (1.5pt);
\filldraw[black] (8.5, -0.5) circle (1.5pt);

\draw[black, thin] (5.5, -1)--(6.25, -2);
\draw[black, thin] (5.5, -1)--(4.75, -2);

\node (Dot6) at (5.5, -2.35) {$\underbrace{\thickspace\thickspace\thickspace \ldots \ldots \thickspace\thickspace\thickspace}_{\geq k^2-\Omega_p(k^2) \text{ vertices}}$};
\node (Dot7) at (7.5, -1.15) {$\overbrace{\thickspace\thickspace\thickspace \ldots \ldots \thickspace\thickspace\thickspace}^{\Omega_p(k^2) \text{ vertices}}$};
\node (Dot8) at (7.75, -0.15) {$\overbrace{\thickspace\thickspace\thickspace \ldots \ldots \thickspace\thickspace\thickspace}^{k-1 \text{ vertices}}$};

\end{tikzpicture}
\end{center}
    
    \caption{A diagram showing the estimation of $\tau - \tau_0$ employed in Heuristic \ref{heuristic2} The main difference between Heuristics \ref{heuristic1} and \ref{heuristic2} is that Heuristic \ref{heuristic2} incorporates the effect of vertex $i$ operating even after move $\tau_0$. This is shown by the diagonal \textcolor{red}{red} arrow $\textcolor{red}{\Nearrow}$ in the diagram.}
    \label{figureHeuristic2}
\end{figure}

As a demonstration, suppose $T$ is a tree rooted at vertex $1$ with $\geq \log \log n$ children, each of which has $\sim n/\log \log n$ children. See Figure \ref{figureHeuristicTree} for a depiction of $T$. Suppose the $j$th child of $1$ from the right is labelled $i_j$. Simulate $\textbf{U}_{\Tam_n}$ starting at $T$, and let the random variable $t_j$ denote the number of steps until $i_j$ is disconnected from $1$. Then for any positive integer $k < \log \log n$,
\begin{enumerate}
    \item Heuristic \ref{heuristic1} above implies that w.h.p. either $t_{k+1} - t_k \gtrsim n/\log \log n$ or $t_{k+1} - t_k = \Omega_p(t_k)$, so the $t_k$ grow exponentially;
    \item Heuristic \ref{heuristic2} above implies that w.h.p. either $t_{k+1} - t_k \gtrsim n/\log \log n$ or $t_{k+1} - t_k = \Omega_p(t_k^2)$, so $\log t_k$ grows exponentially.
\end{enumerate}
Thus Heuristic \ref{heuristic2} implies that $t_{\lfloor \log \log n \rfloor} \gtrsim n/\log \log n$, so $\mathcal{E}(\Tam_n; F) \gtrsim n/\log \log n$, while Heuristic \ref{heuristic1} suggests a slower (but nevertheless exponential) growth from the $t_k$. In any case, \textit{either} heuristic here is enough to guarantee that $T_{n, T}$ is at least $n^{1-o(1)}$ w.h.p.

\begin{figure}[ht]
    \captionsetup{width=0.8\textwidth}
    \begin{center}
\begin{tikzpicture}[scale=1]
\filldraw[black] (0, 0) circle (1.5pt) node[anchor=south]{$1$};
\filldraw[black] (-1.5, -1) circle (1.5pt);
\filldraw[black] (1.5, -1) circle (1.5pt);
\filldraw[black] (-2.5, -2) circle (1.5pt);
\filldraw[black] (-1, -2) circle (1.5pt);
\filldraw[black] (1, -2) circle (1.5pt);
\filldraw[black] (2.5, -2) circle (1.5pt);

\draw[black, thin] (0, 0)--(-1.5, -1);
\draw[black, thin] (0, 0)--(1.5, -1);
\draw[black, thin] (-1.5, -1)--(-1, -2);
\draw[black, thin] (-1.5, -1)--(-2.5, -2);
\draw[black, thin] (1.5, -1)--(1, -2);
\draw[black, thin] (1.5, -1)--(2.5, -2);

\node (Dot3) at (0, -1.3) {$\underbrace{\thickspace \ldots \ldots \ldots \ldots \ldots \thickspace}_{\geq \log \log n \text{ vertices}}$};
\node (Dot4) at (-1.75, -2.3) {$\underbrace{\thickspace  \ldots \ldots \thickspace}_{\sim n/\log \log n \text{ vertices}}$};
\node (Dot5) at (1.75, -2.3) {$\underbrace{\thickspace  \ldots \ldots \thickspace}_{\sim n/\log \log n \text{ vertices}}$};
\end{tikzpicture}
\end{center}
    
    \caption{A diagram of the ``optimal'' tree $T$ for the application of Heuristics \ref{heuristic1} and \ref{heuristic2}}
    \label{figureHeuristicTree}
\end{figure}

The tree $T$ chosen for this demonstration is ideal for the application of Heuristics \ref{heuristic1} and \ref{heuristic2}, as $T$ has height $2$. In practice the height of $T$ will increase with $|T|$, so the children of the root of $T$ will have less children of their own. Thus for a more ``generic'' $T$, to estimate $T_n$ we will need to group the children of the vertices together, and apply Heuristics \ref{heuristic1} and \ref{heuristic2} to these ``groups'' of vertices. In this setting Heuristic \ref{heuristic1} will be more useful to us, as Heuristic \ref{heuristic2} is more difficult to generalize to a group of vertices, some of which may have fewer children. We will also need to apply an inductive argument in order to estimate how the subtrees formed by the descendants of the children of $1$ evolve over time. For this we will need a complex algorithm (Algorithm \ref{algorithmSimulateTamn}) for simulating $\textbf{U}_{\Tam_n}$ that will document the sizes of the child sub-trees, while allowing each sub-tree to retain enough ``randomness'' to apply the inductive hypothesis.

Using these ideas, we will bound $\mathcal{E}(\Tam_n)$ as follows. Simulate the Ungarian Markov Chain on $\Tam_n$ starting at $\hat{1}$. First, we will bound the probability with which $\hat{1}$ transitions (after some number of steps) to a forest $F$ containing a tree $T$ with many ($\sim n$) vertices. We will do this by modeling the first move in which a vertex $i$ receives an operation with a geometric random variable $g_i$, and then investigating the probability that the maximum value of $g_i$ is uniquely attained for some $i < n/2$. Using the \dfn{n-skyline} (to be defined later) we will also scan the ``peaks'' of $\{g_i\}$ to identify which vertices will become children of the root of $T$. We will consider the sequence $\{a_i\}$ formed by the labels of these peaks, and we will let $t_i$ denote the number of steps (after the $g_1-1^\text{th}$ step) until the vertices $a_i$ and $1$ are disconnected. We will also use the \dfn{n-summary} to ensure that the children of $T$ are sufficiently ``spread out''. We will find that with probability $1/2 - o(1)$, our tree $T$ will satisfy our desired conditions, so it will suffice to bound $\mathbb{P}(T_{n, T} = \Omega(n^{1-o(1)}))$.

Now, given a tree $T$ we with our desired properties, we will bound $\mathbb{P}(T_{n, T} = \Omega(n^{1-o(1)}))$ via an induction argument (on $n$). Using the \dfn{n-summary}, we may identify a subsequence $\{a_{i_j}\} \subset \{a_i\}$ of ``landmark children''  which are spread out and have many children between them. In Lemma \ref{lemmaRapidIncrease}, we will combine Heuristics \ref{heuristic1} and \ref{heuristic2} with the inductive hypothesis to show that for $l \lesssim \log \log n$, the $t_{i_j}$ grow superexponentially. Combined with another application of Heuristic \ref{heuristic1} in Lemma \ref{lemmaExpInContext}, this will establish our bound on $\mathbb{P}(T_{n, T} = \Omega(n^{1-o(1)}))$ (approximately given by Lemma \ref{lemmaMain}). We will then apply Lemma \ref{lemmaMain} in Subsection \ref{ssMainTheoremFinish} to establish Theorem \ref{thmMainTheorem2}.

Note that we will use slightly different notation in the proof for Theroem \ref{thmMainTheorem2} from the notation provided above. Notably we will formulate \ref{lemmaMain} differently, without reference to $T_{n, T}$. Nevertheless the main ideas in the proof of Theorem \ref{thmMainTheorem2} are as described above.

\subsection{Introduction to the Proof of Theorem \ref{thmMainTheorem2}}\label{ssthmMainTheorem2Intro}

The proof of Theorem \ref{thmMainTheorem2} is organized as follows. We will first present several definitions that will be used repeatedly throughout the proof. The bulk of the proof of Theorem \ref{thmMainTheorem2} will be dedicated to proving Lemma \ref{lemmaMain}. Since the proof of Lemma \ref{lemmaMain} is rather involved, we will first prove Theorem \ref{thmMainTheorem2} (and Theorem \ref{thmBigTreeLong}) assuming Lemma \ref{lemmaMain} in Subsection \ref{ssMainTheoremFinish}. We will then return to prove Lemma \ref{lemmaMain} in Subsection \ref{ssLemmaMainProof}.

Now, as discussed in Section \ref{sTamariBackground}, identify the vertices of any ordered forest in $\Tam_n$ with its labels under the preorder traversal labeling. In the proof of Theorem \ref{thmMainTheorem2}, we will simulate the Ungarian Markov chain $\textbf{U}_{\Tam_n}$ using Algorithm \ref{algorithmSimulateTamn}. Throughout this section, while running $\textbf{U}_{\Tam_n}$, we will use the convention that ``on the $t^\text{th}$ step'' refers to before the $t^\text{th}$ random Ungar move has been applied, while ``after the $t^\text{th}$ step'' refers to after the $t^\text{th}$ random Ungar move has been applied. Now, we will fully define Algorithm \ref{algorithmSimulateTamn} in Subsection \ref{ssLemmaMainProof}, but for now we will note that Algorithm \ref{algorithmSimulateTamn} will simulate the first few moves as follows. For any $i \in [n]$ and $t \in \mathbb{Z}^+$, let $S_{i, t}$ and $B_{i, t}$ be i.i.d. Bernoulli variables with parameter $p$ (i.e. $\mathbb{P}(S_{i, t} = 1) = p)$. Given any $t \in \mathbb{Z}^+$, suppose that some $j \in [n]$ satisfies $S_{j, t'} = 0$ for all $t' < t$. Then given any $i \in [n]$, operate on the $t^\text{th}$ step as follows:
\begin{itemize}
    \item if $S_{i, t'} =0$ for all $t' < t$, operate on vertex $i$ if and only if $S_{i, t} = 1$;
    \item else, operate on vertex $i$ if and only if $B_{i, t} = 1$.
\end{itemize}

Now, for any $i \in [n]$, let $g_i$ be the smallest integer $t$ such that vertex $i$ receives an operation on move $t$. If  $\textbf{U}_{\Tam_n}$ is simulated using Algorithm \ref{algorithmSimulateTamn}, then $g_i$ is precisely the smallest $t$ such that $S_{i, t} = 1$. Observe that no matter how $\textbf{U}_{\Tam_n}$ is simulated, the variables $\{g_i\}_{i \in [n]}$ are distributed as i.i.d. geometric random variables with parameter $p$.
\begin{definition}
    For any subinterval $I = [i, j] \subset [n]$ and positive integer $m$, let $E_{I, m}$ be the event that $g_i = m$ and $g_l < m$ for all $l \in [i + 1, j]$.
\end{definition}
Similarly to the proof of Proposition \ref{propTreeProperties} (c), we have that $E_{I, m}$ implies that for the first $m - 1$ steps, all of the vertices with label in $[i + 1, j]$ will be descendants of the vertex $i$. Now, we will record data about the variables $\{g_i\}$ using a collection of two-rowed arrays. We will want our arrays to satisfy the following property.

\begin{definition}
    A two-rowed array $A = \begin{bmatrix} a_1 & a_2 & \ldots & a_l \\ b_1 & b_2 & \ldots & b_l \end{bmatrix}$ is \dfn{$n$-childlike} if:
    \begin{itemize}
        \item For all $i \in [l]$, $a_i, b_i \in [n]$;
        \item $a_1 = 1$;
        \item $n \geq a_2 > \ldots > a_l = 2$, and;
        \item $b_1 > b_2 \geq b_3 \geq b_4 \geq \cdots \geq b_l$.
    \end{itemize}
\end{definition}

We will use the following operation on two-rowed arrays to identify the ``peaks'' of the bottom row.
\begin{definition}
    Consider a positive integer $m$, and let $c_1, \ldots c_m$ be a sequence of positive integers. Given a two-rowed array $A$ of the form $A = \begin{bmatrix} 1 & 2 & \ldots & m \\ c_1 & c_2 & \ldots & c_m \end{bmatrix}$, let its \dfn{$n$-skyline} be the array $\begin{bmatrix} a_1 & a_2 & \ldots & a_l \\ b_1 & b_2 & \ldots & b_l \end{bmatrix}$ defined as follows:
    \begin{enumerate}
        \item $a_1 = 1$;
        \item $a_2$ is the largest $j \in [m]$ satisfying $c_j = \max(c_i : i \in [2, m])$;
        \item For $i > 2$, $a_i$ is inductively defined as the largest $j$ satisfying $c_j = \max(c_u : u \in [2, a_{i-1} - 1])$. In particular, we terminate the sequence at $a_l = 2$;
        \item For all $i \in [l]$, $b_i = c_{a_i}$.
    \end{enumerate}
\end{definition}

Now, define the random array $J$ to be the $n$-skyline of the array $\begin{bmatrix} 1 & 2 & \ldots & n \\ g_1 & g_2 & \ldots & g_n \end{bmatrix}$. We will use $J$ to keep track of (a large subset of) the children of vertex $1$. Observe that $J$ is completely determined by the values of the variables $g_i$. Moreover, note that if $g_1 > g_{a_2}$, then $J$ is $n$-childlike.

\begin{definition}
    Given an $n$-childlike array $A$, define $E_A$ to be the event that $J = A$.
\end{definition}

Consider an $n$-childlike array $A = \begin{bmatrix} a_1 & a_2 & \cdots & a_l \\ b_1 & b_2 & \cdots & b_l \end{bmatrix}$. Then by Proposition \ref{propTreeProperties} (c), $E_A$ implies that after $g_1 - 1$ steps, the tree $\hat{1}$ transitions to a tree with vertex $1$ as its root and with vertices $a_2, \ldots, a_l$ all as children of vertex $1$. Now, we will use the following notions to ensure the arrays we consider are sufficiently long and ``equidistributed''.

\begin{definition}
    Given any decreasing subsequence $I = \{a_1, \ldots, a_l\}$ of integers in $[n]$, its \dfn{$n$-summary} is the subsequence $I' = \{a_{i_1}, \ldots, a_{i_{l'}}\}$ chosen inductively as follows:
    \begin{itemize}
        \item We set $a_{i_1} = a_1$;
        \item For each $j > 1$, pick $a_{i_j}$ to be the smallest integer which satisfies $a_{i_j} \geq \frac{a_{i_{j-1}}}{\log n}$. If no such integer exists, the subsequence ends.
    \end{itemize}
\end{definition}

\begin{definition}
    Given any $n$-childlike array $A = \begin{bmatrix} a_1 & \cdots & a_l \\ b_1 & \cdots & b_l \end{bmatrix}$, let $\{a_{i_1}, \ldots, a_{i_{l'}}\}$ be the $n$-summary of $\{a_2, \ldots a_l\}$. Then the \dfn{$n$-summary} of $A$ is the subarray $A' = \begin{bmatrix} a_1 & a_{i_1} & a_{i_2} & \cdots & a_{i_{l'}} \\ b_1 & b_{i_1} & b_{i_2} & \cdots & b_{i_{l'}} \end{bmatrix}$.
\end{definition}

\begin{definition}
    A two-rowed array $A = \begin{bmatrix} a_1 & \cdots & a_l \\ b_1 & \cdots & b_l \end{bmatrix}$ is \dfn{$n$-good} if it is $n$-childlike, and its $n$-summary $A' = \begin{bmatrix} a_1 & a_{i_1} & a_{i_2} & \cdots & a_{i_{l'}} \\ b_1 & b_{i_1} & b_{i_2} & \cdots & b_{i_{l'}} \end{bmatrix}$ satisfies that $a_{i_1} \geq \frac{n}{\log n}$ and $a_{i_{l'}} \leq (\log n)^3$.
\end{definition}

We will soon bound $\mathcal{E}(\Tam_n)$ by conditioning on the probabilty that $J$ is a given $n$-good array. Intuitively, the $n$-good condition is a length condition that ensures the Ungarian Markov chain $\textbf{U}_{\Tam_n}$ does not terminate too quickly.

Note that by definition, $a_{i_{k + 2}} \leq \frac{a_{i_k}}{\log n}$ for all positive integers $k$. Since $a_{i_1} \leq n$, this implies that:
\begin{align*}
    a_{i_{l'}} \leq a_{i_{2 \lceil l'/2 \rceil - 1}} &\leq (\log n)^{1 - \lceil l'/2 \rceil}a_{i_1}, \\
    (\log n)^{\lceil l'/2 \rceil - 1} &\leq a_{i_1} / a_{i_{l'}} \leq n, \\
    \lceil l'/2 \rceil - 1 &\leq \log n  / \log \log n, \\
    l' &\leq \frac{2 \log n}{\log \log n} + 2.
\end{align*}
Moreover by the bound $a_{i_{k+1}} \geq \frac{a_{i_k}}{\log n},$ we have that
\begin{align*}
    a_{i_1} &\leq a_{i_{l'}}(\log n)^{l'-1}, \\
    (\log n)^{l' - 1} &\geq a_{i_1}/a_{i_{l'}} \geq n/(\log n)^4, \\
    l' &\geq \log n / \log \log n - 3. \numberthis \label{eqGoodArrayBound}
\end{align*}

Now, simulate $\textbf{U}_{\Tam_n}$ starting at $\hat{1}$ using Algorithm \ref{algorithmSimulateTamn}. Let the random variable $T_n$ denote the number of steps until vertex $1$ is disconnected from every other vertex. For an absolute constant $\Cl[abcon]{421} > 4$ to be chosen later, define 
\begin{align*}
    f(x) &= \begin{cases} \max\bigg(1, x\exp\Big(-p^8\exp(\Cr{421}/p^2)(\log \log x)^4\Big)\bigg) & \text{if }x \geq 16, \\ 1 & \text{else}. \end{cases}
\end{align*}
Before we proceed, we will make some preliminary adjustments to $\Cr{421}$ as follows. Pick a sufficiently large $\Cr{421}$ such that for all $n \leq e^{e^4}$, we have that $f(n) = 1$. Then for any $a \in (4, \Cr{421}/p^2]$, we have that $e^a/a^4$ is monotonically increasing with $a$ in this range, so
\begin{align*}
    \frac{\exp(a)}{a^4} &\leq \frac{p^8\exp(\Cr{421}/p^2)}{\Cr{421}^4} \\
    &\leq p^8\exp(\Cr{421}/p^2), \\
    \exp(\exp(a)) &\leq \exp\Big(p^8\exp(\Cr{421}/p^2)a^4\Big),
\end{align*}
so $f(a) \leq 1$. Hence $f(x) \leq 1$ for all $x \in [16, e^{e^{\Cr{421}/p^2}}]$. Notably $f$ is continuous at $x = 16$, and is hence continuous everywhere. Moreover, observe that $h(x) = f(x)/x$ is continuous and decreasing along $[1, \infty)$. Hence for any $a, b, r \in [1, \infty)$, we have that
\begin{align*}
    f(a + b) = (a+b)h(a+b) \leq ah(a) + bh(b) = f(a) + f(b),
\end{align*}
and
\begin{align*}
    f(ra) = (ra)h(ra) \leq rah(a) = rf(a).
\end{align*}
These properties will be useful to us moving forward. Now, the main ingredient in our proof of Theorem $\ref{thmMainTheorem2}$ is the following theorem.
\begin{theorem}\label{thmBigTreeLong}
    For all positive integers $n, m$,
    \begin{align*}
        \mathbb{P}(T_n  \geq f(n) + m - 1 | E_{[1, n], m}) &\geq \frac{1}{2}.
    \end{align*}
\end{theorem}

The proof of Theorem \ref{thmBigTreeLong} mainly relies on the following lemma.

    \begin{lemma}\label{lemmaMain}
        Consider any $n$-good array $A$ whose bottom leftmost entry is $m$. Then
        \begin{align*}
            \mathbb{P}(T_n \geq f(n) + m - 1 | E_A) &\geq 0.85.
        \end{align*}
    \end{lemma}

The proof of Lemma \ref{lemmaMain} is quite lengthy; we will come back to it in Subsection \ref{ssLemmaMainProof}. Instead, we will first prove Theorems \ref{thmBigTreeLong} and \ref{thmMainTheorem2} assuming Lemma \ref{lemmaMain}.

\subsection{Proofs of Theorems \ref{thmBigTreeLong}, \ref{thmMainTheorem2} Assuming Lemma \ref{lemmaMain}}\label{ssMainTheoremFinish}

Throughout this section, let $k = \log n$. Before we prove Theorem \ref{thmBigTreeLong}, we will first prove the following lemma.

\begin{lemma}\label{lemmaUsuallyGood}
        For any integer $n > e^{e^{\Cr{421}/p^2}}$, we have that
        \begin{align*}
            \mathbb{P}\bigg( \Big(\bigcup_{A \textup{ good}} 
            E_A \Big) \thinspace\bigg|\thinspace E_{[1, n], m} \bigg) &\geq 1 - \frac{2}{\log \log n},
        \end{align*}
        where the union is over all $n$-good arrays $A$.
    \end{lemma}
    \begin{proof} 
        Simulate the Ungarian Markov chain $\textbf{U}_{\Tam_n}$ using Algorithm \ref{algorithmSimulateTamn}, and then condition on $E_{[1, n], m}$. It suffices to show that $J = \begin{bmatrix} a_1 & a_2 & \cdots & a_l \\ b_1 & b_2 & \cdots & b_l \end{bmatrix}$ is $n$-good with probability at least $1 - \frac{2}{\log \log n}$. Now, given any $i \geq 2$, we claim that for all reals $r \in [0, 1]$,
        \begin{align*} 
            \mathbb{P}\bigg( \frac{a_{i + 1} - 1}{a_i - 2} \geq r \thinspace \bigg|\thinspace a_i > 2 \bigg) \geq 1 - r.
        \end{align*}
        Indeed, condition by fixing the value of $a_i$. Then given any ordered $(a_i - 2)$-tuple $\mathfrak{t}$, the probability that $\{g_{2}, \ldots g_{a_i - 1}\} = \mathfrak{t}$ as ordered tuples is equal to the probability that $\{g_{2}, \ldots, g_{a_i - 1}\} = \mathfrak{t}'$ for any permutation $\mathfrak{t}'$ of $\mathfrak{t}$. Because $a_{i+1}$ is defined as the \textit{largest} index such that $g_{a_{i+1}} = \max(g_{2}, \ldots, g_{a_i - 1})$, we retrieve the inequality in the case of fixed $a_i$; the inequality above is retrieved from summing over all possible values of $a_i$. Note that conditioning on $E_{[1, n], m}$ does not impact this argument. Moreover, by analogous reasoning we have that
        \begin{align*}
            \mathbb{P}\bigg(\frac{a_2 - 1}{n - 1} \geq r \bigg) &\geq 1 - r.
        \end{align*}

        Now let the $n$-summary of $J$ be $J' = \begin{bmatrix} a_1 & a_{i_1} & a_{i_2} & \cdots & a_{i_{l'}} \\ b_1 & b_{i_1} & b_{i_2} & \cdots & b_{i_{l'}} \end{bmatrix}$. Adjust $\Cr{421}$ to ensure that $\log n > 2$. By the claim above we have that $a_{i_1} = a_2 \geq n/k$ with probability at least
        \begin{align*}
            1 - \frac{n/k-1}{n-1} \geq 1 - \frac{2}{k}.
        \end{align*}
        Moreover, for each $j$ satisfying that $a_{i_j} > k^3$, the probability that $a_{i_{j+1}}$ exists (i.e. the summary has at least $j + 2$ columns), conditioned on the probability that $a_{i_j}$ exists, is at least
        \begin{align*}
            1 - \frac{k^2 - 1}{k^3 - 2} \geq 1 - \frac{k^2 - 2/k}{k^3 - 2} \geq 1 - \frac{1}{k}.
        \end{align*}
        Now, for all $j$ we have that $a_{i_{j+2}} < a_{i_j}/k$, so thus the maximum index $u$ for which $a_{i_u} \geq k^3$ satisfies
        \begin{align*}
            a_{i_{u}} \leq a_{i_{2 \lceil u/2 \rceil - 1}} &\leq k^{1 - \lceil u/2 \rceil}a_{i_1}, \\
            k^{\lceil u/2 \rceil - 1} &\leq a_{i_1} / a_{i_{u}} \leq n/k^3, \\
            \lceil u/2 \rceil +2 &\leq \log n  / \log \log n, \\
            u &\leq \frac{2 \log n}{\log \log n} - 4.
        \end{align*}
        Hence by a union bound we have that
        \begin{align*}
            \mathbb{P}(J \textup{ is good}) &\geq 1 - \frac{2 (\log n/\log \log n) - 2}{\log n} \\
            &\geq 1 - \frac{2}{\log \log n},
        \end{align*}
        as desired.
    \end{proof}

\begin{proof}[Proof of Theorem \ref{thmBigTreeLong}]
    For $n \leq e^{e^{\Cr{421}/p^2}}$ we have $f(n) = 1$ and hence the statement is trivial. Thus assume $n > e^{e^{\Cr{421}/p^2}}$. Adjust $\Cr{421}$ to be at least $10$; then $\log \log n > \Cr{421}/p^2 \geq 10.$ Now, by definition, if an $n$-good array $A$ satisfies that the event $E_A$ is contained in $E_{[1, n], m}$, then the bottom leftmost entry of $A$ must be $m$. Hence, combining Lemma \ref{lemmaMain} and Lemma \ref{lemmaUsuallyGood} gives us that
    \begin{align*}
        \mathbb{P}(T_n \geq f(n) + m - 1 | E_{[1, n], m}) &\geq 0.85\bigg(1 - \frac{2}{\log \log n}\bigg) \\
        &\geq \frac{1}{2},
    \end{align*}
    as desired.
\end{proof}

We now prove Theorem \ref{thmMainTheorem2}.

\begin{proof}[Proof of Theorem \ref{thmMainTheorem2}]
    Simulate the Ungarian Markov chain on $\Tam_n$ starting at $\hat{1}$ using Algorithm \ref{algorithmSimulateTamn}. For any integers $i \in [n]$ and $M \in \mathbb{Z}^+$, let $A_{i, M}$ be the event that $g_i = M > \max_{j \neq i}(g_j)$. By Proposition \ref{propTreeProperties} (b), we have that the induced subgraph formed by the vertices $[i, n]$ transitions as per the Ungarian Markov chain on $\Tam_{n - i + 1}$. Hence
    \begin{align*}
        \E_p(\Tam_n | A_{i, M}) &\geq \E_p(\Tam_{n-i+1} | E_{[1, n-i+1], M}) \\
        &\geq \mathbb{E}[T_{n-i+1} | E_{[1, n-i+1], M}] \\
        &\geq \frac{f(n - i + 1)}{2},
    \end{align*}
    where the last inequality is by Theorem \ref{thmBigTreeLong}. Now, note that $\sum_{M = 1}^{\infty} \mathbb{P}(A_{i, M})$ is the probability that $g_i > g_j$ for all $j \neq i$, which is precisely $\frac{\zeta_p(n)}{n}$ by symmetry (where $\zeta_p(n)$ is the constant discussed in Subsection \ref{ssIntroLattice}). Thus,
    \begin{align*}
        \E_p(\Tam_n) &\geq \sum_{i=1}^n \sum_{M=1}^{\infty} \E_p(\Tam_n | A_{i, M})\mathbb{P}(A_{i, M}) \\
        &\geq \frac{f(\lceil n/2 \rceil)}{2} \sum_{i=1}^{\lceil n/2 \rceil} \sum_{M=1}^{\infty} \mathbb{P}(A_{i, M}) \\ 
        &\geq \frac{\zeta_p(n)}{4}f(\lceil n/2 \rceil) \\
        &\geq \frac{\zeta_p^-f(n)}{8}.
    \end{align*}
    By choosing $\Cr{1}$ to be a sufficiently large constant greater than $\Cr{421}$, this implies Theorem \ref{thmMainTheorem2}, as desired.
\end{proof}

\subsection{Proof of Lemma \ref{lemmaMain}}\label{ssLemmaMainProof}

        We proceed by induction on $n$. As has been standard throughout Section \ref{sTamnProof}, identify the vertices of each ordered forest $F \in \Tam_n$ with its preorder traversal label. Recall that for all $n \leq e^{e^{\Cr{421}/p^2}}$, we have $f(n) = 1$ and hence the theorem is trivial. Henceforth assume $n > e^{e^{\Cr{421}/p^2}}$, and assume throughout the rest of this section that Lemma \ref{lemmaMain} is true for all integers $n'$ less than $n$.

        Now, adjust $\Cr{421}$ such that all $n > e^{e^{\Cr{421}/p^2}}$ satisfy $\log n / \log \log n - 10 > 402(\log \log n)^3$. We now define an algorithm for simulating the Ungarian Markov Chain $\textbf{U}_{\Tam_n}$ starting at $\hat{1}$.

        \begin{algorithm}\label{algorithmSimulateTamn}

        First, define $7$ sequences $\{S_{i, t}\}$, $\{B_{i, t}\}$, $\{B_{i, t}'\}$, $\{C_{j, t}\}$, $\{D_{i, t}\}$, $\{D_{i, t}'\}$, $\{D_{i, t}^\dagger\}$ of i.i.d. Bernoulli variables of parameter $p$ (i.e. $\mathbb{P}(S_{i, t} = 1) = p$). For convenience, let $\mathcal{B}$ denote the set of random variables
        \begin{align*}
            \mathcal{B} = \{S_{i, t}\} \cup \{B_{i, t}\} \cup \{B_{i, t}'\} \cup \{C_{j, t}\} \cup \{D_{i, t}\} \cup \{D_{i, t}'\} \cup \{D_{i, t}^\dagger\}.
        \end{align*}
        In all these sequences, let $i$ range within the interval $[1, n]$ and $t$ range along $\mathbb{Z}^+$. Also, for each ordered forest $F \in \Tam_n$, identify the vertices of $F$ with its preorder traversal labeling.

        Begin with the ordered forest $\hat{1} \in \Tam_n$ before step $1$. Consider any $t \in \mathbb{Z}^+$. Suppose that some $j \in [n]$ satisfies $S_{j, t'} = 0$ for all $t' < t$. Then given any vertex $i \in [n]$, on the $t^\text{th}$ step:
        \begin{itemize}
            \item if $S_{i, t'} = 0$ for all $t' < t$, operate on vertex $i$ if and only if $S_{i, t} = 1$;
            \item else, operate on vertex $i$ if and only if $B_{i, t} = 1$.
        \end{itemize}

        Now, suppose that for all $i \in [n]$, $S_{i, t'} = 1$ for some $t' < t$. For all $i \in [n]$, let $g_i$ denote the smallest $t'$ such that vertex $i$ receives an operation on move $t'$ (i.e. $S_{i, t'} = 1$). Define the random array $J$ as before to be the $n$-skyline of the array $\begin{bmatrix} 1 & 2 & \ldots & n \\ g_1 & g_2 & \ldots & g_n \end{bmatrix}$.
        \begin{itemize}
            \item If $J$ is not $n$-good, then on the $t^\text{th}$ step, operate on a vertex $i$ if and only if $S_{i, t} = 1$.
            \item If $J$ is $n$-good, let $J = \begin{bmatrix} a_1 & a_2 & \cdots & a_l \\ b_1 & b_2 & \cdots & b_l \end{bmatrix}$. Let $J' = \begin{bmatrix} a_1 & a_{i_1} & a_{i_2} & \cdots & a_{i_{l'}} \\ b_1 & b_{i_1} & b_{i_2} & \cdots & b_{i_{l'}} \end{bmatrix}$ be the $n$-summary of $J$.
            \begin{itemize}
                \item On the $t^\text{th}$ step, say $i$ is the smallest index at least $2$ satisfying that vertex $1$ is still connected to vertex $a_i$. Then operate on vertex $1$ as follows:
                    \begin{itemize}
                        \item If $i$ exists and $i \leq i_{2\lceil 201 \log \log n \rceil}$, then operate on vertex $1$ if and only if $D_{i, t} = 1$.
                        \item Else, if $i \in (i_{j-2}, i_j]$ for some even $j \in [2\lceil 201 \log \log n \rceil+2, 2\lceil 201 (\log \log n)^3 \rceil]$, then operate on vertex $1$ if and only if $D_{j, t}' = 1$.
                        \item Else, if $i > 2 \lceil 201 (\log \log n)^3 \rceil$ or $i$ does not exist, operate on vertex $1$ if and only if $D_{1,t}^\dagger = 1$.
                    \end{itemize}
                \item Consider an integer $i \in [2, a_{i_{2\lceil 201 (\log \log n)^3 \rceil}})$. Then on the $t^\text{th}$ step, operate on vertex $i$ if and only if $B_{i, t} = 1$.
                \item Consider any integer $i \in [a_{i_{2\lceil 201 (\log \log n)^3 \rceil}}, a_{i_{2\lceil 201 \log \log n \rceil}})$. Suppose that $i \in [a_{i_{j}}, a_{i_{j-2}})$ for some even integer $j \in [2 \lceil 201 \log \log n \rceil +2, 2\lceil 201 (\log \log n)^3 \rceil]$. Let $F_t \langle a_{i_{j}}, a_{i_{j-2}} - 1 \rangle$ denote the induced subgraph of the current state (interpreted as a graph-theoretic forest) formed by the vertices in $[a_{i_{j}}, a_{i_{j-2}})$. On the $t^\text{th}$ step, operate on vertex $i$ as follows:
                \begin{itemize}
                    \item If after $t - 1$ steps $i$ is the largest root of $F_t \langle a_{i_{j}}, a_{i_{j-2}} - 1 \rangle$ that has at least one child, then operate on vertex $i$ if and only if $C_{j, t} = 1$;
                    \item Else, operate on vertex $i$ if and only if $B_{i, t} = 1$.
                \end{itemize}
                \item Consider any integer $i \geq a_{i_{2\lceil 201 \log \log n \rceil}}$. Suppose that $i \in [a_j, a_{j-1} - 1]$ for some \\$j \in [2, i_{2\lceil 201 \log \log n \rceil}]$ (here when $j = 2$, use the interval $[a_2, n]$ instead of $[a_2, a_1 - 1]$). On the $t^\text{th}$ step, operate on vertex $i$ as follows:
                \begin{itemize}
                    \item If $i \neq a_j$, then operate on $i$ if and only if $B_{i, t} = 1$.
                    \item Else, if vertex $1$ is connected to vertex $a_{j-1}$ after $t - 1$ steps, then operate on $i$ if and only if $B_{i, t} = 1$.
                    \item Else, operate on vertex $i$ if and only if $B_{i, t}' = 1.$
            \end{itemize}
        \end{itemize}
        \end{itemize}
        
        \end{algorithm}
    
    When using Algorithm \ref{algorithmSimulateTamn}, it will be useful to remember the following fact from Corollary \ref{corTamnTreeUngarMoves}. Suppose that on a given turn $t$, the vertices $l_1, l_2, \ldots l_m$ are chosen to be operated on, where the preorder traversal labels satisfy $l_1 < \ldots < l_m$. Then on that step, the vertices are operated on in order of increasing label, i.e. $l_1$ is operated on first and $l_m$ last.

    Now, while using Algorithm \ref{algorithmSimulateTamn}, suppose $\tau - 1$ steps have elapsed. Then on the $\tau^\text{th}$ turn, for any vertex $i$, the status of whether vertex $i$ is chosen for an operation is determined by some variable in $\mathcal{B} = \{S_{i, t}, B_{i, t}, B_{i, t}', C_{j, t}, D_{i, t}, D_{i, t}'\}$ whose indices satisfy $t = \tau$. Moreover, the variables chosen for each vertex are distinct, and the choice of variables for each vertex is determined by the values of \\$\{S_{i, t}, B_{i, t}, B_{i, t}', C_{j, t}, D_{i, t}, D_{i, t}' : i \in [n], \tau < t\}$. Thus each vertex is operated on with independent probability $p$, and so Algorithm \ref{algorithmSimulateTamn} applies a random Ungar move (with parameter $p$) on each step. Hence Algorithm \ref{algorithmSimulateTamn} indeed simulates $\textbf{U}_{\Tam_n}$, starting at $\hat{1}$.

    Now, let our given $n$-good array be
    \begin{align*}
        A = \begin{bmatrix} a_1 & a_2 & \ldots & a_l \\ b_1 & b_2 & \ldots & b_l \end{bmatrix},
    \end{align*}
    and let its $n$-summary be
    \begin{align*}
        A' = \begin{bmatrix} a_1 & a_{i_1} & \ldots & a_{i_{l'}} \\ b_1 & b_{i_1} & \ldots & b_{i_{l'}} \end{bmatrix}.
    \end{align*}
    Note that conditioning on $E_A$ implies that $J = A$ and $J' = A'$. Throughout the following proof, we will always implicitly condition on $E_A$. To do so, we will require that:
    \begin{enumerate}
        \item All $i \in [l]$ satisfy $S_{a_i, t} = 0$ for $t < g_{a_i}$ and $S_{a_i, t} = 1$ for $t = g_{a_j}$;
        \item If $j \in (a_i, a_{i-1})$ for some $i \in [3, l]$, then among the variables $(S_{j, 1}, \ldots S_{j, g_{a_i} - 1})$, at least one is equal to $1$;
        \item If $j \in (a_2,  n]$, then among the variables $(S_{j, 1}, \ldots S_{j, g_{a_2} - 1})$, at least one is equal to $1$;
    \end{enumerate}
    Note that in the above, $l$ denotes the length of the two-rowed array $A$. Also note that if $j$ is not equal to $a_i$ for some $i$, then either $j \in (a_i, a_{i-1})$ for some $i$, or $j \in (a_2, n]$. The above requirements clearly are equivalent to conditioning on $E_A$. Moreover, from the above, it is clear that the event $E_A$ is independent from the state of any Bernoulli variable in $\mathcal{B}$ not mentioned above (e.g. the variables $\{B_{i, t}\}$, $\{B_{i, t}'\}$, $\{C_{j, t}\}$, $\{D_{i, t}\}$, $\{D_{i, t}'\}$, $\{D_{i, t}^\dagger\}$, as well as some of the variables $S_{i, t}$).
    
    Throughout the rest of the proof of Lemma \ref{lemmaMain}, we will treat the variables $\{g_{a_i}\}_{i \in [l]}$ as fixed (rather than random) variables. Also note that by Equation \eqref{eqGoodArrayBound}, the condition that any $n \geq e^{e^{\Cr{421}/p^2}}$ satisfies $\log n / \log \log n - 10 > 402 (\log \log n)^3$ ensures that the index $l'$ in the array $J'$ is greater than \\$2\lceil 201 (\log \log n)^3 \rceil$. So e.g. the term $a_{i_{\lfloor 2\lceil 201 (\log \log n)^3 \rceil \rfloor + 1}}$ exists.
    
    Now, for all $j \in [2, l]$, let $t_j$ be the random variable denoting the number of steps after the $g_1 - 1^\text{th}$ step until vertex $1$ is no longer connected to vertex $a_{j}$. Observe that $t_2 \geq 1$, and since vertex $1$ may only be operated on at most once every step, the number of children it has can only decrease by at most $1$ every step. Hence $t_{i+1} - t_i \geq 1$ for every $i \in [2, l]$. In particular, $t_{i_{\lceil \log \log n \rceil}} \geq i_{\lceil \log \log n \rceil} - 1 \geq \log \log n$. We now establish the following bound on $t_{i_{2\lceil 201 \log \log n \rceil}}$.

    \begin{lemma}\label{lemmaRapidIncrease}
        There exists an event $A_1 \subset E_A$ satisfying that
        \begin{itemize}
            \item $A_1$ is determined by the random variables
            \begin{itemize}
                \item $D_{j, t}$, for $j \leq i_{2\lceil 201 \log \log n \rceil}$ and $t \in \mathbb{Z}^+$, and
                \item $S_{j, t}$, $B_{j, t}$, and $B_{j, t}'$, for $j \geq a_{i_{2\lceil 201 \log \log n \rceil}}$ and $t \in \mathbb{Z}^+$;
            \end{itemize}
            \item $A_1$ is contained in the event that
            \begin{align*}
                t_{i_{2\lceil 201 \log \log n \rceil}} &\geq  \frac{f\Big(\frac{n}{2e^{403 (\log \log n)^2}}\Big)}{1700};
            \end{align*}
            \item $\mathbb{P}(A_1 | E_A) \geq 0.9$.
        \end{itemize}
    \end{lemma}

    The proof of Lemma \ref{lemmaRapidIncrease} will mainly depend on a growth estimate presented in Lemma \ref{lemmaIndiv}. In turn, Lemma \ref{lemmaIndiv} will depend on Lemma \ref{lemmaApplyInductiveHypothesis}, which will allow us to apply the inductive hypothesis of Lemma \ref{lemmaMain} to develop stronger bounds. This induction argument will require us to condition on events of the following form.
    \begin{definition}
        Condition on $E_A$, and consider any $i \leq 2 \lceil 201 \log \log n \rceil$. Then call any event $E$ \dfn{i-independent} if $E$ satisfies the following properties:
        \begin{itemize}
            \item $E$ is independent of the random variables:
            \begin{itemize}
                \item $B_{a_i, t}$ for $t \in \mathbb{Z}^+$;
                \item $B_{j, t}$ for $j \in [a_i, a_{i-1} - 1]$ and $t \in \mathbb{Z}^+$;
            \end{itemize}
            \item For each $j \in [a_i + 1, a_{i-1} - 1]$ and $t \in \mathbb{Z}^+$, $E$ is conditionally independent of $S_{i, t}$ given $E_A$.
        \end{itemize}
    \end{definition}

    \begin{lemma}\label{lemmaApplyInductiveHypothesis}
        Consider an integer $i \leq 2 \lceil 201 \log \log n \rceil$, and any $i$-independent event $E$. Let $t_i'$ be the number of steps after the $g_{a_i}-1^\text{th}$ step until no vertex in $[a_i + 1, a_{i-1} - 1]$ is a child of $a_i$. Then
        \begin{align*}
            \mathbb{P}(t_i' \geq f(a_{i-1} - a_i + 1) | E_A \cap E) \geq 0.85.
        \end{align*}
    \end{lemma}
    \begin{proof}

    Consider any sequence $\gamma = \{c_{a_i + 1}, c_{a_i + 2}, \ldots c_{a_{i-1}-1}\}$ of positive integers, all of which are less than $g_{a_i}$. Let $m = a_{i-1} - a_i +1$, and let $\eta_\gamma = \{d_1, d_2, \ldots d_m\}$ be the corresponding sequence defined by:
    \begin{itemize}
        \item $d_1 = g_{a_i}$;
        \item $d_j = c_{a_i + j - 1}$ for $j \in [2, m-1]$;
        \item $d_m = g_{a_{i-1}}$.
    \end{itemize}
    When simulating $\textbf{U}_{\Tam_n}$, let $E_\gamma$ be the event that $g_j = c_j$ for all $j \in [a_i + 1, a_{i-1} - 1]$. For any sequence $\eta= \{d_1, d_2, \ldots d_m\}$, define the \dfn{$\eta$-indexed Ungarian Markov process} $\textbf{U}^\eta_{\Tam_m}$ as follows. On step $t = 0$, start with the tree $\hat{1} \in \Tam_m$. Then for each vertex $j \in [m]$ and $t \in \mathbb{Z}^+$:
    \begin{enumerate}
        \item If $t < d_j$, then do not let $j$ receive an operation on turn $t$;
        \item If $t = d_j$, let $j$ receive an operation on turn $t$;
        \item If $t > d_j$, let $j$ receive an operation on turn $t$ with independent probability $p$.
    \end{enumerate}
    On each turn $t$, apply an Ungar move by successively operating on the vertices chosen to be operated on, in increasing order of preorder traversal label. Now, for any ordered forest $F \in \Tam_n$ and sub-interval $[i, j] \subset [n]$, let $F \langle i, j \rangle$ denote the induced sub-forest formed by the vertices with labels in $[i, j]$. Using Algorithm \ref{algorithmSimulateTamn}, simulate $\textbf{U}_{\Tam_n}$ starting at $\hat{1}$ conditioned on $E_A \cap E_\gamma \cap E$. This prescribes the value of $g_m$ for all $m \in \{a_j\}_{j \in [l]} \cup [a_i, a_{i-1}]$; for any such $m$, we may condition the value of $g_m$ by setting $S_{m, t'} = 0$ for $t'$ less than the prescribed value of $g_m$, and $S_{m, t'} = 0$ for $t'$ equal to the prescribed value of $g_m$. For any $t \geq 0$, let $F_t$ denote the forest we retrieve after running this Markov chain for $t$ steps; here we let $F_0 = \hat{1} \in \Tam_n$. Our main claim is that for any $\gamma \subset [g_{a_i}-1]^{m-1}$, the sequence of random forests $\{F_t \langle a_i, a_{i-1} \rangle\}_{t \geq 0}$ in $\Tam_m$ is identically distributed to the sequence of random forests $\{H_t\}_{t \geq 0}$ retrieved by starting at $H_0 = \hat{1} \in \Tam_m$ and then running the $\eta_\gamma$-indexed Ungarian Markov Process $U^{\eta_\gamma}_{\Tam_m}$.

    As above, simulate $\textbf{U}_{\Tam_n}$ starting at $\hat{1}$ using Algorithm \ref{algorithmSimulateTamn}, conditioned on $E_A \cap E \cap E_\gamma$. Consider any $j \in [a_i, a_{i-1} - 1]$; recall that $E_A \cap E \cap E_\gamma$ fixes the value of $g_j$. Since $E$ is $i$-independent, we find that vertex $j$ receives no operations on the first $g_j - 1$ moves, receives an operation on the $g_j^{\text{th}}$ move, and receives an operation with probability $p$ for every move afterwards. Note that any operations on the vertex $a_{i-1}$ do not affect the tree $F_t \langle a_i, a_{i-1} \rangle$ (similarly any operations on the vertex $m \in H_t$ do not affect $H_t$). Hence to prove the claim above, it suffices to show that for any pair $(j, t) \in [n] \times \mathbb{Z}^+$ such that vertex $j$ may receive an operaton on step $t$ (i.e. if the value of $g_j$ is prescribed then $t \geq g_j$),
    \begin{align}\label{eqInducedSubgraphCommute}
        (F_t[j])\langle a_i, a_{i-1} \rangle = (F_t\langle a_i, a_{i-1} \rangle)[j],
    \end{align}
    where again $F[j]$ denotes the forest obtained by operating on vertex $j$ of a forest $F$. Notably, for any $j \not \in [a_i, a_{i-1}]$, we must show that $F_t[j]\langle a_i, a_{i-1} \rangle = F_t \langle a_i, a_{i+1} \rangle$.

    To show the above, first note that as per the proof of Proposition \ref{propTreeProperties} (b), any operation on a vertex $j$ for $j < a_i$ does not affect the induced subgraph on $[a_i, a_{i-1}]$. So it suffices to prove Equation \ref{eqInducedSubgraphCommute} for $j \geq a_i$. Now, observe that for any ordered forest $F'$ and vertex $j' \in V(F')$, the set of descendants $\mathfrak{d}_{j''}$ of any vertex $j'' \in V(F')$ is left unchanged if $j'' \neq j'$, while $\mathfrak{d}_{j'}$ is replaced with a subset of itself (proper if $j'$ is not a leaf). Also note that given any ordered forest $F'$, we can retrieve $F'$ using the collection of sets $\{\mathfrak{d}_{j'}\}_{j' \in V(F')}$. With this in mind:
    \begin{itemize}
        \item For $j \in [a_i, a_{i-1}-1]$, observe that $g_j < g_{a_i} \leq g_{a_{i-1}}$, so after the first operation on $j$, the vertices in $[a_{i-1}, n]$ are no longer descendants of $j$. Thus for all $t > g_j$, all the descendants of $j$ are vertices in $[a_i, a_{i-1} - 1]$. Since Equation \eqref{eqInducedSubgraphCommute} is clearly true for $(j, g_j)$, we have that Equation \eqref{eqInducedSubgraphCommute} is in fact true for all $t$ when $j \in [a_i, a_{i-1} - 1]$.
        \item For all other $j$, since $g_{a_i} \leq g_{a_{i-1}}$, we have that on the $g_{a_i}^\text{th}$ move, either $a_{i-1}$ is not operated on or it is operated on after $a_i$ is. Either way, all operations on $a_{i-1}$ occur after $a_i$ and $a_{i-1}$ are disconnected (at which point $[a_i, a_{i-1}]$ is downward-closed, in the sense that any descendant of any vertex in $[a_i, a_{i-1}]$ is also in $[a_i, a_{i-1}]$. Notably no vertex in $[a_{i-1} + 1, n]$ can ever be connected to a vertex in $[a_i, a_{i-1}-1]$ via an edge. Hence we have that Equation \eqref{eqInducedSubgraphCommute} is true (for all $t$) for $j = a_i$ as well as all $j \in [a_{i-1}, n]$, as desired. Thus the claim is proved.
    \end{itemize}
    With the claim in hand, consider any $\gamma \in [g_{a_i} - 1]^{m-1}$. Run $\textbf{U}_{\Tam_m}$ starting at $\hat{1}$, and let $E_{\eta_\gamma}'$ be the event that $g_i = d_i$ for all $i \leq m - 1$, and $g_m < g_1$. Let $\mathfrak{t}$ denote the number of steps after the $(g_{a_i} - 1)^\text{th}$ in $\textbf{U}_{\Tam_m}$ until the vertex $1$ has no more children. Recall that for any forest in $\Tam_m$, the vertex labeled $m$ is always a leaf, so operating on $m$ will not change the forest. Thus the sequence of forests $\{H_t'\}_{t \geq 0}$ obtained by starting at $H_0' = \hat{1} \in \Tam_m$ and running $\textbf{U}_{\Tam_m}$ conditioned on $E_{\eta_\gamma}'$, is identically distributed to the sequence of forests $\{H_t\}_{t \geq 0}$ obtained by starting at $\hat{1}$ and running the process $U^\gamma_{\Tam_m}$. The claim above implies these forests are identically distributed to the sequence $\{F_t \langle a_i, a_{i-1} \rangle \}_{t \geq 0}$.  Thus for any $\gamma$, we have that
    \begin{align*}
        \mathbb{P}(t_i' \geq f(m) | E_A \cap E \cap E_\gamma)= \mathbb{P}(\mathfrak{t} \geq f(m) | E_{\eta_\gamma}').
    \end{align*}
    Observe that the events $\{E_A \cap E_\gamma \cap E: \gamma \in [g_{a_i} - 1]^{m-1} \}$ partition $E_A \cap E$, while the events $\{E_{\eta_\gamma}' : \gamma \in [g_{a_i}-1]^{m-1}$ partition the event $E_{[1, m], g_{a_i}}$. By the $i$-independence of $E$, we have
    \begin{align*}
        \mathbb{P}(E_A \cap E \cap E_\gamma | E_A \cap E) = \mathbb{P}(E_A \cap E_\gamma | E_A) =\mathbb{P}(E_{\eta_\gamma}' | E_{[1, m], g_{a_i}}),
    \end{align*}
    where the last equality is clear. Thus we have that
    \begin{align*}
        \sum_\gamma \mathbb{P}(t_i' \geq f(m) | E_A \cap E \cap E_\gamma) \cdot \mathbb{P}(E_A \cap E \cap E_\gamma | E_A \cap E) &= \sum_{\eta_\gamma} \mathbb{P}(\mathfrak{t} \geq f(m) | E_{\eta_\gamma}')\mathbb{P}(E_{\eta_\gamma}' | E_{[1, m], g_{a_i}}), \\
        \sum_\gamma \frac{\mathbb{P}(\{t_i' \geq f(m)\} \cap E_A \cap E \cap E_\gamma)}{\mathbb{P}(E_A \cap E)} &= \sum_{\eta_\gamma} \frac{\mathbb{P}(\{\mathfrak{t} \geq f(m)\} \cap E_{\eta_\gamma}' \cap E_{[1, m], g_{a_i}})}{\mathbb{P}(E_{[1, m], g_{a_i}})} \\
        \mathbb{P}(t_i' \geq f(m) | E_A \cap E) &= \mathbb{P}(\mathfrak{t} \geq f(m) | E_{[1, m], g_{a_i}})
    \end{align*}
    Now since $m < n$, the inductive hypothesis on Lemma \ref{lemmaMain} implies that
    \begin{align*}
        \mathbb{P}(\mathfrak{t} \geq f(m) | E_{[1, m], g_{a_i}}) \geq 0.85,
    \end{align*}
    as desired.
    \end{proof}

    We now prove the following growth estimate on the random variables $\{t_i\}$.

    \begin{lemma}\label{lemmaIndiv}
        There exists an absolute constant $\Cl[abcon]{hitconst}$ such that for every $i_{\lceil \log \log n \rceil} < i \leq i_{2\lceil 201 \log \log n \rceil}$, there exists an event $E_i \subset E_A$ satisfying that:
        \begin{enumerate}
            \item The event $E_i$ is determined by the values of the random variables
            \begin{itemize}\item $D_{j, t}$, for $j \leq i$ and $t \in \mathbb{Z}^+$, and
                \item $S_{j, t}$, $B_{j, t}$, and $B_{j, t}'$, for $j \geq a_{i}$ and $t \in \mathbb{Z}^+$.
            \end{itemize}
            Notably, $E_i$ is independent of all other variables in $\{B_{i, t}\} \cup \{B_{i, t}'\} \cup \{C_{j, t}\} \cup \{D_{i, t}\} \cup \{D_{i, t}'\} \cup \{D_{i, t}^\dagger\}$, and for all $j \in [2, a_i)$ and $t \in \mathbb{Z}^+$, $E_i$ is conditionally independent of $S_{j, t}$ given $E_A$.
            \item $E_i$ is contained in the event that
            \begin{align*}
                t_i - t_{i-1} \geq \min\bigg(\Cr{hitconst}(pt_{i-1} - \sqrt{t_{i-1}})^2, \frac{f(a_{i-1} - a_i)}{17}\bigg);
            \end{align*}
            \item We have that $\mathbb{P}(E_i)$, conditioned on $E_A$ as well as the values of the random variables
            \begin{itemize}
                \item $D_{j, t}$, for $j \leq i-1$ and $t \in \mathbb{Z}^+$, and
                \item $S_{j, t}$, $B_{j, t}$, and $B_{j, t}'$, for $j \geq a_{i-1}$ and $t \in \mathbb{Z}^+$,
        \end{itemize}
            is always at least $0.15$.
        \end{enumerate}
    \end{lemma}

    \begin{proof}
        Observe that conditioned on $E_A$, $t_{i-1}$ is determined by the values of the random variables
        \begin{itemize}
                \item $D_{j, t}$, for $j \leq i-1$ and $t \in \mathbb{Z}^+$, and
                \item $S_{j, t}$, $B_{j, t}$, and $B_{j, t}'$, for $j \geq a_{i-1}$ and $t \in \mathbb{Z}^+$.
        \end{itemize}
        Denote this set of random variables as $\mathcal{B}_i$. Now, throughout this proof, condition on $E_A$ as well as the values of the random variables in $\mathcal{B}_i$; note that this conditioning is $i$-independent. Define the following events:
        \begin{enumerate}
            \item Let $E_{i, 1}$ be the event that 
            \begin{align*}
                \sum_{u=g_{a_i} + 1}^{g_1 - 1 + t_{i-1}} B_{a_i, u} \geq p(t_{i-1} + g_1 - g_{a_i} - 1) - \sqrt{t_{i-1} + g_1 - g_{a_i} - 1}.
            \end{align*}
            \item Define $\tau_0$ to be the smallest $t$ satisfying that
            \begin{align}\label{eq2}
                \sum_{u=g_1 + t_{i-1}}^{g_1 - 1 + t_{i-1}+t} D_{i, u} - B_{a_i, u}' = \left \lceil p(t_{i-1} + g_1 - g_{a_i} - 1) - \sqrt{t_{i-1} + g_1 - g_{a_i} - 1} -1\right \rceil.
            \end{align}
            If no such $t$ exists, define $\tau_0 = \infty$. Note that if some $t$ satisfies that the LHS of \eqref{eq2} is at least the RHS, then $\tau_0$ must be finite by discrete continuity. Now, for any $r \in \mathbb{R}^+$, let $E_{i, 2, r}$ be the event that $\tau_0$ satisfies
            \begin{align*}
                \tau_0 &\geq r(p(t_{i-1} + g_1 - g_{a_i} - 1) - \sqrt{t_{i-1} + g_1 - g_{a_i} - 1}-1)^2.
            \end{align*}
            Note that the events $E_{i, 1}$ and $E_{i, 2, r}$ measure the effect of the estimate on $t_i - t_{i-1}$ discussed in Heuristic \ref{heuristic2}.
            \item Let $t_i'$ be the number of steps after the $g_{a_i} - 1^\text{st}$ step until no vertex in $[a_i + 1, a_{i-1} - 1]$ is a child of $a_i$. Let $E_{i, 3}$ be the event that $t_i' \geq f(a_{i-1} - a_i+1)$.
            \item Like in the proof of Lemma \ref{lemmaApplyInductiveHypothesis}, let the random variable $F_t \in \Tam_n$ denote the forest obtained after running the Ungarian Markov chain for $t$ steps. Let $F_t \langle a_i, a_{i-1} - 1 \rangle$ denote the induced sub-forest of $F_t$ formed by the vertices in $[a_i, a_{i-1} - 1]$. Given any $t > g_{a_i}$, let the random variable $j_t$ denote the largest label of a vertex $j$ satisfying that $j$ is a root of $F_t \langle a_i, a_{i-1} - 1 \rangle$ with at least one child (if no such vertex exists set $j_t = a_i$). Moreover, let $u_0 = t_{i-1} + g_1 - 1$, and for $z \geq 1$, let $u_z$ be the random variable denoting the $z^\text{th}$ smallest $u$ satisfying $u > t_{i-1} + g_1 - 1$ and $D_{i, u} = 1$. Finally, let $Y_z = u_z - u_{z-1}$.
            \begin{itemize}
                \item Given any real $0 < \gamma_2 < 1$, let $E_{i, 4, \gamma_2}$ be the event that
                \begin{align*}
                    \sum_{t = g_{a_i} + 1}^{g_1-1 + t_{i-1}} B_{j_t, t} > \gamma_2 p(t_{i-1} + g_1 - g_{a_i}-1).
                \end{align*}
                \item Let $R = \min(\lceil \gamma_2 p (t_{i-1} + g_1 - g_{a_i} - 1) \rceil, a_{i-1} - a_i)$. Then given any real $0 < \gamma_3 < 1$, let $E_{i, 5, (\gamma_2, \gamma_3)}$ be the event that
                \begin{align*}
                    \sum_{i=1}^{R} Y_i \geq \gamma_3 R/p.
                \end{align*}
                Note that $E_{i, 4, \gamma_2}$ and $E_{i, 5, (\gamma_2, \gamma_3)}$ are an adaptation of the estimate on $t_i - t_{i-1}$ discussed in Heuristic \ref{heuristic1}.
            \end{itemize}
        \end{enumerate}

        Note that conditioned on $E_A$ and the values of the variables in $\mathcal{B}_i$, the events above are determined by the values of the random variables in $\{D_{i, t}\}_{i, t} \cup \{S_{j, t}, B_{j, t}, B_{j, t}'\}_{j, t}$, where $j \in [a_i, a_{i-1} - 1]$ and $t \in \mathbb{Z}^+$. Now, for any $r \in \mathbb{R}^+$ and $\gamma_2, \gamma_3 \in (0, 1)$, define the event $E_{i, r, \gamma_2, \gamma_3}$ as the intersection
        \begin{align*}
            E_{i, r, \gamma_2, \gamma_3} = E_{i, 1} \cap E_{i, 2, r} \cap E_{i, 3} \cap E_{i, 4, \gamma_2} \cap E_{i, 5, (\gamma_2, \gamma_3)}.
        \end{align*}
        Suppose $E_{i, r, \gamma_2, \gamma_3}$ is true; we will show that for appropriate choices of $r, \gamma_2, \gamma_3$, $E_{i, r, \gamma_2, \gamma_3}$ will satisfy our desired properties. Observe that $E_{i, r, \gamma_2, \gamma_3}$ is determined by the random variables described in the statement of the lemma. Now we will show that $E_{i, r, \gamma_2, \gamma_3}$ implies the desired bound on $t_i - t_{i-1}$. We will lower bound $t_i - t_{i-1}$ in two ways; for the first bound, assume the events $E_{i_1}$, $E_{i, 2, r}$, and $E_{i, 3}$. First, suppose that $t_i+ g_1 \leq t_i' + g_{a_i}$. We claim that:
        \begin{enumerate}
            \item If $t \in [g_{a_i} + 1, g_1 - 1 + t_{i-1}]$, then after $t$ steps,
            \begin{align*}
                \#\{ j \in [a_i, a_{i-1} - 1] : j\text{ child of }1\} \geq \sum_{u=g_{a_i} + 1}^{t} B_{a_i, u};
            \end{align*}
            \item If $t \in [g_1 + t_{i-1}, g_1 - 2 + t_i]$, then after $t$ steps,
            \begin{align}\label{eq514ineq2}
                 \#\{ j \in [a_i, a_{i-1} - 1] : j\text{ child of }1\} \geq \sum_{u=g_{a_i} + 1}^{g_1 - 1 + t_{i-1}} B_{a_i, u} -\bigg( \sum_{u=g_1 + t_{i-1}}^{t} D_{i, u} - B_{a_i, u}'\bigg).
            \end{align}
        \end{enumerate}
        Indeed, all choices of $t$ above are in the interval $[g_{a_i} + 1, g_{a_i} - 1 + t_i']$, so on the $t^\text{th}$ step, $a_i$ is a root in $F_t \langle a_i, a_{i-1} - 1 \rangle$ with at least one child. Meanwhile, any $t \leq t_i + g_1 - 1$ satisfies that right before the $t^\text{th}$ step, $a_i$ is still a child of the vertex $1$. Thus for any $t \in [g_{a_1} + 1, g_1 + t_{i-1} - 1]$, we have that each operation on $a_i$ increases the number of children of $1$ by one, implying the first inequality. Meanwhile, for any $t \in [g_1 + t_{i-1}, g_1 - 2 + t_i]$, we have that the number of children of $1$ in $[a_i, a_{i-1} - 1]$ increases by at least one whenever $a_i$ is operated on, and can only decrease (by one) whenever $1$ is operated on. This implies the second inequality, and thus the claim.
        
        Now, suppose $E_{i, 1}$, $E_{i, 2, r}$ and $t_i + g_1 \leq t_i' + g_{a_i}$ are all true. We know that after $g_1 - 1 + t_i$ steps, since $a_i$ is no longer connected to $1$, neither is any other vertex in $[a_i, a_{i-1} - 1]$ (by the definition of the preorder traversal). Since $g_1 - 1 + t_i \leq g_{a_i} - 1 + t_i'$, we have that after $g_1 - 2 + t_i$ steps, the rightmost child of $1$ must be the vertex $a_i$, and we must have $D_{i, g_1 - 1 + t_i} = 1$. Hence Inequality \eqref{eq514ineq2} implies:
        \begin{align*}
            1 &\geq \sum_{u=g_{a_i} + 1}^{g_1 - 1 + t_{i-1}} B_{a_i, u} -\bigg( \sum_{u=g_1 + t_{i-1}}^{g_1+t_i-1} D_{i, u} - B_{a_i, u}'\bigg), \\
            \sum_{u=g_1 + t_{i-1}}^{g_1+t_i-1} D_{i, u} - B_{a_i, u}' &\geq \sum_{u=g_{a_i} + 1}^{g_1 - 1 + t_{i-1}} B_{a_i, u} - 1.
        \end{align*}
        Since $E_{i, 1}$ is true, we find that
        \begin{align*}
            \sum_{u=g_1 + t_{i-1}}^{g_1+t_i-1} D_{i, u} - B_{a_i, u}' &\geq \left \lceil p(t_{i-1} + g_1 - g_{a_i} - 1) - \sqrt{t_{i-1} + g_1 - g_{a_i} - 1} - 1 \right \rceil,
        \end{align*}
        where taking the ceiling is justified since the left hand side is an integer. By discrete continuity, we thus have that $\tau_0$ exists and is at most $t_i - t_{i-1}$. Since $E_{i, 2, r}$ is true, we thus find
        \begin{align*}
            t_i - t_{i-1} \geq \tau_0 \geq r(p(t_{i-1} + g_1 - g_{a_i}-1) - \sqrt{t_{i-1} + g_1 - g_{a_i}-1})^2.
        \end{align*}
        Now, suppose $t_i + g_1 > t_i' + g_{a_i}$. Since $E_{i, 3}$ is true, we find that
        \begin{align*}
            t_i + g_1 &> t_i' + g_{a_i} \\
            t_i - t_{i-1} &\geq t_i' + g_{a_i} - g_1 - t_{i-1} + 1 \\
            &\geq f(a_{i-1} - a_i + 1) - (g_1 - g_{a_i}) - t_{i-1} + 1.
        \end{align*}
        where the second inequality is because both sides consist of integers. Hence, $E_{i, 1} \cap E_{i, 2, r} \cap E_{i, 3}$ always implies
        \begin{align*}
            t_i - t_{i-1} &\geq \min(r(p(t_{i-1} + g_1 - g_{a_i}-1) - \sqrt{t_{i-1} + g_1 - g_{a_i}-1})^2, f(a_{i-1} - a_i + 1) - (g_1 - g_{a_i}) - t_{i-1} + 1).
        \end{align*}
        We will now create a second lower bound for $t_i - t_{i-1}$ as follows. Let the random variable $\mathfrak{c}_i$ denote the number of vertices in $[a_i, a_{i-1} - 1]$ that are children of $1$ after $t_{i-1} + g_1-1$ steps. Suppose $\mathfrak{c}_i < a_i - a_{i-1}$. Then for all $t \leq t_{i-1} + g_1-1$ satisfying $B_{j_t, t} = 1$, we have that the number of vertices in $[a_i, a_{i-1} - 1]$ increases after the $t$th step. Assuming $E_{i, 4, \gamma_2}$ is true, this implies that
        \begin{align*}
            \mathfrak{c}_i > \gamma_2p(t_{i-1} + g_1 - g_{a_i} - 1).
        \end{align*}
        Hence regardless of whether $\mathfrak{c}_i < a_i - a_{i-1}$ or not, we have that
        \begin{align*}
            \mathfrak{c}_i \geq \min(\lceil\gamma_2p(t_{i-1} + g_1 - g_{a_i} - 1)\rceil, a_{i-1} - a_i) = R,
        \end{align*}
        where the ceiling may be taken since $\mathfrak{c}_i$ is an integer. Hence vertex $1$ must receive at least $R$ operations after $t_{i-1}$ to become disconnected from $a_i$. Now, if $E_{i, 5, (\gamma_2, \gamma_3)}$ is true, then
        \begin{align*}
            u_R - (t_{i-1} + g_1 - 1) = u_R - u_0 = \sum_{i=1}^R Y_i \geq \gamma_3 R/p.
        \end{align*}
        So it takes vertex $1$ at least $\gamma_3 R/p$ steps to receive those $R$ operations, and hence $t_i - t_{i-1} \geq \gamma_3 R/p$. Combining the above estimates, we have that if $E_{i,r, \gamma_2, \gamma_3}$ is true, then
        \begin{align*}
            t_i - t_{i-1} &\geq \max\Bigg(\min\bigg(\gamma_3 \gamma_2 (t_{i-1} + g_1 - g_{a_i}-1), \frac{\gamma_3 (a_{i-1} - a_i)}{p}\bigg), \\
            &\min\Big(r(p(t_{i-1} + g_1 - g_{a_i}-1) - \sqrt{t_{i-1} + g_1 - g_{a_i}-1})^2, f(a_i - a_{i-1}+1) - (t_{i-1} + g_1 - g_{a_i})+1\big)\Bigg). 
        \end{align*}
        Set $\gamma_2 = \gamma_3 = 1/4$, and consider any $r \leq 1/16$. Observe that
        \begin{align*}
            \max\bigg(\frac{t_{i-1} + g_1 - g_{a_i}-1}{16}, f(a_i - a_{i-1} + 1) - (t_{i-1} + g_1 - g_{a_i}-1)\bigg) \geq \frac{f(a_i-a_{i-1}+1)}{17}.
        \end{align*}
        Every integer $n \geq 1$ satisfies $f(n + 1) \leq n$, so $f(a_i - a_{i-1} + 1)/17 < (a_{i-1}-a_i)/4p$. Hence we have that
        \begin{align*}
            t_i - t_{i-1} \geq \min\bigg(r(p(t_{i-1} + g_1 - g_{a_i}-1) - \sqrt{t_{i-1} + g_1 - g_{a_i}-1})^2, \frac{f(a_i - a_{i-1} + 1)}{17}\bigg)
        \end{align*}
        Finally, the function $q(x) = r(px - \sqrt{x})^2$ is increasing if $x > p^2 / 2$, and positive if $x > 1/p^2$. Adjust $\Cr{421}$ to be at least $2$; then $t_{i-1} \geq t_{i_{\lceil \log \log n \rceil}} \geq \log \log n \geq \Cr{421}/p^2 > \max(p^2/2, 1/p^2)$, so $q(x)$ is increasing in $[t_{i-1}, \infty)$. Since $g_1 - g_{a_i} - 1 \geq 0$ and $f$ is nondecreasing, we find that if $E_{i, r, 1/4, 1/4}$ is true, then
        \begin{align*}
            t_i - t_{i-1} \geq \min\bigg(r(pt_{i-1} - \sqrt{t_{i-1}})^2, \frac{f(a_i - a_{i-1})}{17}\bigg)
        \end{align*}
        Hence it now suffices to prove that for some (sufficiently small) absolute constant $\Cr{hitconst} \in (0, 1/16)$ we have that $\mathbb{P}(E_{i, \Cr{hitconst}, 1/4, 1/4}) \geq 0.15$. To do this, we bound the probabilities of the aforementioned events for general values of $r, \gamma_2, \gamma_3$ as follows. Throughout we again condition on $E_A$ as well as the random variables in $\mathcal{B}_i = \{D_{j, t}\}_{j \leq i - 1}\cup\{S_{j, t}, B_{j, t}, B_{j, t}'\}_{j \geq a_{i-1}}$; note that this conditioning fixes $t_{i-1}$ and $R$ and is $i$-independent.
        \begin{enumerate}
            \item By Hoeffding's inequality, we have that $\mathbb{P}(E_{i, 1}) \geq 1 - e^{-2}$.
            \item Let $D^\circ_{a_i, u} = D_{i, u + g_1 + t_{i-1}-1} - B_{a_i, u + g_1 + t_{i-1}-1}'$. Then note that each $D^\circ_{a_i, u}$ is independent with probability distribution
            \begin{align*}
                D^\circ_{a_i, u} &= \begin{cases} 1 & \text{ probability }p(1-p) \\ 0 & \text{ probability }p^2 + (1-p)^2 \\ -1 & \text{ probability }p(1-p). \end{cases}
            \end{align*}
            Thus the sequence $\{D^\circ_{a_i, 1}, D^\circ_{a_i, 1} + D^\circ_{a_i, 2}, \ldots\}$ defines a lazy simple random walk on $\mathbb{Z}$ with parameter $p(1-p)$, and hence we may apply Corollary \ref{corLazyLemmaRandomWalk} to obtain that
            \begin{align*}
                \mathbb{P}(E_{i, 2, r}) &\geq 1 - e^{-\Cr{hittingtime}/\sqrt{r}}.
            \end{align*}
            \item Since our conditioning is $i$-independent, we have by Lemma \ref{lemmaApplyInductiveHypothesis} that $\mathbb{P}(E_{i, 3}) \geq 0.85$.
            \item Recall that for any $i$, $t_{i-1} \geq \log \log n \geq \Cr{421}/p^2$. Note also that $g_1 - g_{a_i} - 1 \geq 0$. Now by applying a multiplicative Chernoff bound to the $B_{j_t, t}$, we retrieve that
            \begin{align*}
                \mathbb{P}(E_{i, 4, \gamma_2}) \geq 1 - \exp(-(1-\gamma_2)^2p(t_{i-1} + g_1 - g_{a_i} - 1))\geq 1 - \exp(-(1 - \gamma_2)^2\Cr{421}/2).
            \end{align*}
            \item Recall that our conditioning fixed $t_{i-1}$ and $R$. Since $t_{i-1} \geq \Cr{421}/p^2$, we have that $R \geq \min(\gamma_2\Cr{421}/p, 1)$. Now, under our conditioning, each $Y_i$ is an independent geometric variable of parameter $p$ taking values in $\mathbb{Z}^+$. Thus by applying Lemma \ref{lemmaSumGeomBound}, we have:
            \begin{align*}
                \mathbb{P}(E_{i, 5, \gamma_2, \gamma_3}) &= 1 - \mathbb{P}\bigg(\sum_{i=1}^{R} Y_i < \frac{R}{p} - (1-\gamma_3)(\sqrt{Rp})\sqrt{\frac{R}{p^3}}\bigg),\\
                &\geq 1 - \exp\bigg(\frac{-Rp(1-\gamma_3)^2}{2p - p(1-\gamma_3)}\bigg), \\
                &= 1 - \exp\bigg(\frac{-R(1-\gamma_3)^2}{1+\gamma_3}\bigg), \\
                &\geq 1 - \exp\bigg(\frac{-\min(\gamma_2\Cr{421}/p, 1)(1-\gamma_3)^2}{1+\gamma_3}\bigg).
            \end{align*}
        \end{enumerate}
        Finally, by plugging in $\gamma_2 = \gamma_3 = 1/4$ and applying a union bound we find that
        \begin{align*}
            \mathbb{P}(E_{i, \Cr{hitconst}, 1/4, 1/4}) \geq 1 - e^{-2} - e^{-\Cr{hittingtime}/\sqrt{\Cr{hitconst}}} - (1-0.85) - e^{-9\Cr{421}/32} - e^{-\min(9\Cr{421}/80, 9/20)}.
        \end{align*}
        Adjusting $\Cr{421}$ to be sufficiently large and $\Cr{hitconst}$ to be sufficiently small now gives us that $\mathbb{P}(E_{i, \Cr{hitconst}, 1/4, 1/4}) \geq 0.15$, as desired.
    \end{proof}
    
    We'll now prove the following lemma, which will allow us to apply the results of Lemma \ref{lemmaIndiv} to groups of indices $a_i$ at a time, and hence derive a steadier growth estimate for the sequence $\{t_{i_j}\}$. From this lemma onward, we will reuse the notation from Subsection \ref{ssMainTheoremFinish} that $k = \log n$.
    \begin{lemma}\label{lemmaLongRapidIncrease}
        For every $\lceil \log \log n \rceil \leq j \leq 2\lceil 201 \log \log n \rceil - 2$, there exists an event $G_j \subset E_A$ satisfying that:
        \begin{itemize}
            \item $G_j$ is determined by the random variables:
            \begin{itemize}
                \item $D_{u, t}$, for $u \leq i_{j+2}$ and $t \in \mathbb{Z}^+$, and
                \item $S_{u, t}$, $B_{u, t}$, and $B_{u, t}'$, for $u \geq a_{i_{j+2}}$ and $t \in \mathbb{Z}^+$;
            \end{itemize}
            \item $G_j$ is contained in the event that
            \begin{align*}
                t_{i_{j+2}} - t_{i_j} &\geq 0.01\min\bigg(\frac{f(n/k^j - n/k^{j+1})}{17}, \Cr{hitconst}(pt_{i_j} - \sqrt{t_{i_j}})^2\bigg);
            \end{align*}
            \item We have that $\mathbb{P}(G_j)$, conditioned on $E_A$ as well as the variables
            \begin{itemize}
                \item $D_{u, t}$, for $u \leq i_{j}$ and $t \in \mathbb{Z}^+$, and
                \item $S_{u, t}$, $B_{u, t}$, and $B_{u, t}'$, for $u \geq a_{i_{j}}$ and $t \in \mathbb{Z}^+$;
            \end{itemize}
            is always at least $0.1$.
        \end{itemize}
    \end{lemma}
    \begin{proof}
        For each $a \in \mathbb{R}_{\geq 0}$, define
        \begin{align*}
            h_j(a) &= \min \bigg( \Cr{hitconst}(pt_{i_j} - \sqrt{t_{i_j}})^2, \frac{f(a)}{17} \bigg).
        \end{align*}
        For any $i \in (i_{\lceil \log \log n \rceil}, i_{2\lceil 201 \log \log n \rceil}]$, let $E_i$ be the event from Lemma \ref{lemmaIndiv}. Note that for every $i \in [i_j + 1, i_{j+2}]$, the event $E_i$ is contained in the event that
        \begin{align*}
            t_i - t_{i-1} \geq h_j(a_{i-1} - a_i).
        \end{align*}
        Now, any nonnegative reals $a, b, c$ satisfy
        \begin{align*}
            \min(c, a) + \min(c, b) = \min(2c, a + c, b + c, a + b) \geq \min(c, a + b).
        \end{align*}
        Recall that for all nonnegative reals $a, b$, we have that $f(a) + f(b) \geq f(a + b)$. Hence $h_j(a) + h_j(b) \geq h_j(a + b)$ as well. Now, for each $i \in [1, i_{j+2} - i_j]$, let $X_i$ be the indicator function for $E_{i_j+i}$. Let $G_j$ be the event that
        \begin{align*}
            \sum_{i=1}^{i_{j+2} - i_j} h_j(a_{i+i_j-1} - a_{i+i_j})X_i \geq 0.01 h_j(a_{i_j} - a_{i_{j+2}}).
        \end{align*}
        We claim that $G_j$ satisfies the desired properties. Conditioned on $E_A$, $G_j$ clearly is determined by the random variables described. Now throughout the rest of this proof, condition on $E_A$, as well as the variables
        \begin{itemize}
                \item $D_{u, t}$, for $u \leq i_{j}$ and $t \in \mathbb{Z}^+$, and
                \item $S_{u, t}$, $B_{u, t}$ and $B_{u, t}'$, for $u \geq a_{i_{j}}$ and $t \in \mathbb{Z}^+$.
        \end{itemize}
        By Lemma $\ref{lemmaIndiv}$ we know that when conditioning on the variables above and on the variables
        \begin{itemize}
            \item $D_{u, t}$, for $i_j < u \leq i_{j} + i - 1$ and $t \in \mathbb{Z}^+$, and
            \item $S_{u, t}$, $B_{u, t}$ and $B_{u, t}'$, for $a_{i_j} > u \geq a_{i_{j} + i - 1}$ and $t \in \mathbb{Z}^+$,
        \end{itemize}
        we have that $\mathbb{P}(X_i = 1) \geq 0.15$. In particular, $\mathbb{P}(X_i = 1| X_1, \ldots X_{i-1}) \geq 0.15$ for all possible values of $\{X_1, \ldots X_{i-1}\}$ and $i \in [1, i_{j+2} - i_j]$, so $\mathbb{E}[X_i] \geq 0.15$.
        Now, let $q = 0.15$, and consider any real $r \in (0, q]$. We have by the Markov inequality that
        \begin{align*}
            \mathbb{P}\bigg(\sum_{i=1}^{i_{j+2} - i_j} h_j(a_{i+i_j-1} &- a_{i+i_j})(1 - X_i) \geq (1-r)\sum_{i=1}^{i_{j+2} - i_j} h_j(a_{i+i_j-1} - a_{i+i_j})\bigg) \\
            &\leq \frac{\mathbb{E}\bigg[\displaystyle\sum_{i=1}^{i_{j+2} - i_j} h_j(a_{i+i_j-1} - a_{i+i_j})(1-X_i)\bigg]}{(1-r)\displaystyle\sum_{i=1}^{i_{j+2} - i_j} h_j(a_{i+i_j-1} - a_{i+i_j})}, \\
            \mathbb{P}\bigg(\sum_{i=1}^{i_{j+2} - i_j} h_j(a_{i+i_j-1} &- a_{i+i_j})X_i \leq r\sum_{i=1}^{i_{j+2} - i_j} h_j(a_{i+i_j-1} - a_{i+i_j})\bigg) \\
            &\leq \frac{(1-q)\bigg(\displaystyle\sum_{i=1}^{i_{j+2} - i_j} h_j(a_{i+i_j-1} - a_{i+i_j})\bigg)}{(1-r)\displaystyle\sum_{i=1}^{i_{j+2} - i_j} h_j(a_{i+i_j-1} - a_{i+i_j})} = \frac{1-q}{1-r}, \\
            \mathbb{P}\bigg(\sum_{i=1}^{i_{j+2} - i_j} h_j(a_{i+i_j-1} &- a_{i+i_j})X_i > r\sum_{i=1}^{i_{j+2} - i_j} h_j(a_{i+i_j-1} - a_{i+i_j})\bigg) \geq \frac{q-r}{1-r}. \\
        \end{align*}
        Now, note that 
        \begin{align*}
            \sum_{i=1}^{i_{j+2} - i_j} h_j(a_{i+i_j-1} - a_{i+i_j}) &\geq h_j(a_{i_{j+2}} - a_{i_j}).
        \end{align*}
        Thus, setting $r = 0.01$ gives us that $\mathbb{P}(G_j) \geq 0.14/0.99 > 0.1$, as desired. Now, note that we always have $t_{i + i_j} - t_{i + i_j - 1} \geq h_j(a_{i+i_j-1} - a_{i+i_j})X_i$. Also, since $a_{i_j} \geq n/k^j$ (where $k = \log n$), $a_{i_{j+2}} \leq a_{i_j}/k$, and $h_j$ is nondecreasing, we have that $h_j(a_{i_j} - a_{i_{j+2}}) \geq h_j(n/k^j - n/k^{j+1})$. Thus, $G_j$ is contained in the event that
        \begin{align*}
            t_{i_{j+2}} - t_{i_j} = \sum_{i=i_j+1}^{i_{j+2}} t_i - t_{i-1} \geq 0.01 h_j(a_{i_j} - a_{i_{j+2}}) \geq 0.01 h_j(n/k^j - n/k^{j+1}),
        \end{align*}
        as desired.
    \end{proof}

    We now prove Lemma \ref{lemmaRapidIncrease}.

    \begin{proof}[Proof of Lemma \ref{lemmaRapidIncrease}]
        First, adjust $\Cr{421}$ to be at least $4$. Then by picking any (small) absolute constant $\Cl[abcon]{422} \in (0, \min(\Cr{hitconst}/400, 1))$, we find that for all $t \geq \Cr{421}/p^2$,
        \begin{align*}
            0.01\Cr{hitconst}(pt-\sqrt{t})^2 \geq \Cr{422}p^2t^2.
        \end{align*}
        Fixing $\Cr{422}$, adjust $\Cr{421}$ again such that $\Cr{421} \geq e/\Cr{422}$. Moving forward, each time we decrease $\Cr{422}$ we will correspondingly increase $\Cr{421}$ to ensure this inequality holds. Now, let $N$ be the number of even indices $j \in (\lceil \log \log n \rceil, 2\lceil 201 \log \log n \rceil - 2]$ for which $G_j$ is true. Let $A_1$ be the event that $N \geq 2 \log \log n$. We claim that $A_1$ satisfies the desired properties. Evidently, $A_1$ is determined by the variables listed. We now bound $\mathbb{P}(A_1)$. Condition on $E_A$. For each even $j \in (\lceil \log \log n \rceil, 2\lceil 201 \log \log n \rceil -2]$, let $X_{j/2 + 1 - \lceil (\log \log n)/2 \rceil}$ be the indicator function of $G_j$. Also, suppose there are $J$ such even $j$ in total. Then $N = \sum_{j=1}^J X_j$, and by an analogous argument to the one in Lemma \ref{lemmaLongRapidIncrease} we find that for all $j \in [J]$, $\mathbb{P}(X_j = 1 | X_{j-1}, \ldots X_1) \geq 0.1$. Now, let $Y_i = X_i - 0.1$, and let $Z_i = \sum_{k \leq i} Y_k$. Then
        \begin{align*}
            \mathbb{E}[Z_i | Z_1, \ldots Z_{i-1}] &= Z_{i-1} + \mathbb{E}[Y_i | X_1, \ldots X_{i-1}] \\
            &\geq Z_{i-1} - 0.1 + \mathbb{E}[X_i | X_1, \ldots X_{i-1}] \\
            &\geq Z_{i-1}.
        \end{align*}
        Hence $\{Z_i\}_i$ forms a super-martingale. Observe that $|Z_k - Z_{k-1}| = |Y_k| \leq 0.9$ for all $k$. Thus by the Azuma-Hoeffding inequality, we have that
        \begin{align*}
            \mathbb{P}\bigg(\sum_{i=1}^J X_i \geq 0.01J \bigg) &= \mathbb{P}\bigg(Z_J \geq -0.09J \bigg) \\
            &\geq 1 - \exp\bigg(\frac{-0.09J^2}{2 \cdot 0.9^2 \cdot J}\bigg) \\
            &= 1 - \exp (-J/18).
        \end{align*}
        Now, note that $J \geq \frac{401 \log \log n - 4}{2} - 1$, so $0.01 J \geq 2 \log \log n + 0.005 \log \log n - 0.03$. Adjust $\Cr{421}$ to be at least $6$; then $\log \log n \geq \Cr{421}/p^2  > 6$ so $0.01 J \geq 2 \log \log n$. This also gives us that $J \geq 200  \log \log n > 18 \log 10$, so $\exp(-J/18) < 0.1$. Hence $\mathbb{P}(A_1) \geq 0.9$, as desired. Finally, suppose $A_1$ is true. Let $s_i$ be the $i$th $j$ for which $G_j$ is true. First, note that for all $j \leq 2\lceil 201 \log \log n \rceil$, 
        \begin{align*}
            f(n/k^j - n/k^{j+1}) &\geq f(n/k^{403 \log \log n} - n/k^{403 \log \log n + 1}) \\
            &\geq f(n/2k^{403 \log \log n}) = f(n/2e^{403(\log \log n)^2}).
        \end{align*}
        Hence, each $G_j$ is contained in the event that $t_{i_{j+2}} - t_{i_j} \geq \min(f(n/2e^{403(\log \log n)^2})/1700, \Cr{422}p^2t_{i_j}^2)$. Now, recalling that $\log \log n \geq \Cr{421}/p^2 \geq e/\Cr{422}p^2$, we may find by induction that for all $i$,
        \begin{align*}
            t_{i_{s_i}} \geq \min(f(n/2e^{403(\log \log n)^2})/1700, e^{2^i}/\Cr{422}p^2).
        \end{align*}
        Observe that since the variables $\{X_i\}$ are integers, $A_1$ implies that $N \geq \lceil 2 \log \log n \rceil$. Thus, $A_1$ is contained in the event that
        \begin{align*}
            t_{i_{2\lceil 201 \log \log n \rceil}} &\geq t_{i_{s_{\lceil 2 \log \log n \rceil}}} \\
            &\geq \min \bigg( \frac{f(n/2e^{403 (\log \log n)^2})}{1700}, \frac{e^{2^{\lceil 2 \log \log n \rceil}}}{\Cr{422}p^2} \bigg ) \\
            &\geq \min \bigg( \frac{f(n/2e^{403 (\log \log n)^2})}{1700}, \frac{n}{\Cr{422}p^2} \bigg ) \\
            &\geq \frac{f(n/2e^{403 (\log \log n)^2})}{1700},
        \end{align*}
        as desired.
    \end{proof}

        We now prove a bound on $t_{i_{\ell}} - t_{i_{\ell-2}}$, for any even $\ell \in [2\lceil 201 \log \log n \rceil+2, 2\lceil 201 (\log \log n)^3 \rceil]$. Combined with the initial bound on $t_{i_{2\lceil 201 \log \log n \rceil}}$ given by Lemma \ref{lemmaRapidIncrease}, this will give us the key bound on $T_n$ that will allow us to prove Lemma \ref{lemmaMain}. Throughout the proof we will again let $k = \log n$.
        \begin{lemma}\label{lemmaExpInContext}
            For every even $\ell$ with $2\lceil 201 \log \log n \rceil + 2 \leq \ell \leq 2\lceil 201 (\log \log n)^3 \rceil$, there exists an event $V_\ell \subset E_A$ such that
            \begin{enumerate}
                \item $V_\ell$ is determined by the random variables
                \begin{itemize}
                    \item $D_{j, t}$, for all $j \in [n], t \in \mathbb{Z}^+$,
                    \item $S_{j, t}$, $B_{j, t}$ and $B_{j, t}'$, for $j \geq a_{i_{\ell}}$ and $t \in \mathbb{Z}^+$,
                    \item $D_{j, t}'$ and $C_{j, t}$, for $j \leq \ell$ and $t \in \mathbb{Z}^+$.
                \end{itemize}
                \item $V_\ell$ is contained in the event that
                \begin{align*}
                    t_{i_{\ell}} - t_{i_{\ell-2}} \geq \frac{\min(t_{i_{\ell-2}}, n/k^{\ell-2})}{4}.
                \end{align*}
                \item Conditioned on $E_A$, as well as the random variables
                \begin{itemize}
                    \item $D_{j, t}$, for all $j \in [n], t \in \mathbb{Z}^+$,
                    \item $S_{j, t}$, $B_{j, t}$ and $B_{j, t}'$, for $j \geq a_{i_\ell - 2}$ and $t \in \mathbb{Z}^+$,
                    \item $D_{j, t}'$ and $C_{j, t}$, for $j \leq \ell - 2$ and $t \in \mathbb{Z}^+$,
                \end{itemize}
                we have that $\mathbb{P}(V_\ell) \geq 1 - \exp(-49pt_{i_{\ell - 2}}/2560) - \exp(-\min(n/k^{\ell - 2},pt_{i_{\ell-2}})/12)$.
        \end{enumerate}
        \end{lemma}

        \begin{proof}
            This is roughly another application of the estimate presented in Heuristic \ref{heuristic1}. Condition on the variables $\{D_{j, t}\}_{j \in [n]}$, $\{D_{j, t}', C_{j, t}\}_{j \leq \ell - 2}$, $\{S_{j, t}, B_{j, t}, B_{j, t}'\}_{j \geq a_{i_{\ell-2}}}$ as discussed above. Observe that this conditioning fixes $t_{i_{\ell-2}}$. For any integer $m \in \mathbb{Z}^+$ satisfying $g_1 < m \leq t_{i_{\ell-2}} + g_1 - 1$, let $N_\ell(m)$ be the number of children of the vertex $1$ with indices in $[a_{i_{\ell}}, a_{i_{\ell-2}})$ on the $m^\text{th}$ step. For $m \leq t_{i_{\ell-2}} + g_1 - 1$, we have that on the $m^{\text{th}}$ step, all vertices in $[a_{i_{\ell}}, a_{i_{\ell-2}})$ are descendants of $1$. Hence on the $m^\text{th}$ step, any operation on a non-leaf child of $1$ in $[a_{i_{\ell}}, a_{i_{\ell-2}})$ increases the number of children of $1$ in $[a_{i_{\ell}}, a_{i_{\ell-2}})$. Thus either $N_\ell(m) = a_{i_{\ell-2}} - a_{i_{\ell}}$, or $N_\ell(m + 1) \geq N_\ell(m) + C_{\ell, m}$. Hence by induction, we have
            \begin{align*}
                N_\ell(t_{i_{\ell-2}} + g_1 - 1) &\geq \min\bigg((a_{i_{\ell-2}} - a_{i_{\ell}}),\sum_{i=g_1+1}^{t_{i_{\ell-2}}+g_1-2} C_{\ell, i} \bigg).
            \end{align*}
            Now note that $t_{i_{\ell - 2}} \geq 400 \log \log n \geq 400 \Cr{421}/p^2 \geq 400$, so $t_{i_{\ell - 2}} - 2 > 0.8t_{i_{\ell-2}}$. For any $\gamma_1 \in (0, 0.8)$, let $E_{6, \gamma_1}$ be the event that $\sum_{i=g_1+1}^{t_{i_{\ell-2}}+g_1-2} C_{\ell, i} > \gamma_1 p(t_{i_{\ell-2}}-2)$. By a multiplicative Chernoff bound, we have that for all $\gamma_1 \in (0, 0.8)$,
            \begin{align*}
                \mathbb{P}\Bigg(\sum_{i=g_1+1}^{t_{i_{\ell-2}}+g_1-2} C_{\ell, i} \leq \gamma_1 pt_{i_{\ell-2}}\Bigg) &\leq \mathbb{P}\Bigg(\sum_{i=g_1+1}^{t_{i_{\ell-2}}+g_1-2} C_{\ell, i} \leq 1.25\gamma_1 p(t_{i_{\ell-2}}-2)\Bigg) \\
                &\leq \exp\Big(-0.5(1.25\gamma_1-1)^2p(t_{i_{\ell-2}}-2)\Big) \\
                &\leq \exp\Big(-0.4(1.25\gamma_1-1)^2pt_{i_{\ell-2}}\Big) \\
                \mathbb{P}(E_{6, \gamma_1}) &\geq 1 - \exp\Big(-0.4(1.25\gamma_1-1)^2pt_{i_{\ell-2}}\Big).
            \end{align*}
            Now, let $X_0 = 0$, and for $i > 0$, let $X_i$ be the $i$th smallest index $j > 0$ such that $D_{\ell, g_1 - 1 + t_{i_{\ell-2}} + j}' = 1$. For any positive integer $i$, let $Y_i = X_i - X_{i-1}$. Note that the sequence of $\{Y_i\}$'s are distributed as a sequence of mutually independent geometric random variables with parameter $p$. Now, condition on $E_{6, \gamma_1}$; this gives us that $N_\ell(t_{i_{\ell - 2}} + g_1 - 1) \geq \min(0.8\gamma_1pt_{i_\ell}, a_{i_{\ell-2}} - a_{i_{\ell}})$. Letting $R_{\gamma_1} = \min(\lceil0.8\gamma_1pt_{i_\ell}\rceil, a_{i_{\ell-2}} - a_{i_\ell})$, we find that vertex $1$ must receive at least $R_{\gamma_1}$ operations after the turn $g_1 - 1 + t_{i_{\ell - 2}}$ until it is disconnected from all vertices in $[a_{i_{\ell}}, a_{i_\ell-2})$. Thus
            \begin{align*}
                t_{i_{\ell}} - t_{i_\ell-2} \geq \sum_{i=1}^{R_{\gamma_1}} Y_i.
            \end{align*}
            Now for any $\gamma_2 \in (0, 1)$, let $E_{7, \gamma_1, \gamma_2}$ be the event that $\sum_{i=1}^{R_{\gamma_1}} Y_i \geq \gamma_2R_{\gamma_1}/p$. Applying Lemma \ref{lemmaSumGeomBound} (b), we have that for all $\gamma_2 < 1$,
            \begin{align*}
                \mathbb{P}\bigg(\sum_{i=1}^{R_{\gamma_1}} Y_i < \frac{\gamma_2R_{\gamma_1}}{p}\bigg)  &= \mathbb{P}\bigg(\sum_{i=1}^{R_{\gamma_1}} Y_i < \frac{R_{\gamma_1}}{p} - (1-\gamma_2)\sqrt{R_{\gamma_1}p}\sqrt{\frac{R_{\gamma_1}}{p^3}}\bigg)\\
                &< \exp\Big(-(1-\gamma_2)^2R_{\gamma_1}/(1 + \gamma_2)\Big),\\
                \mathbb{P}(E_{7, \gamma_1, \gamma_2}) &\geq 1-\exp\Big(-(1-\gamma_2)^2R_{\gamma_1}/(1 + \gamma_2)\Big).
            \end{align*}
            Now, let the event $V_\ell$ be the intersection $V_{\ell} = E_{6, 5/8} \cap E_{7, 5/8, 1/2}$. Note that $a_{i_{\ell-2}} - a_{i_{\ell}} \geq a_{i_{\ell-2}}(1 - k) \geq n/2k^{\ell - 2}$, hence $R_{5/8} = \min(\lceil 0.5 pt_{i_\ell} \rceil, a_{i_{\ell - 2}} - a_{i_\ell}) \geq \min(pt_{i_\ell}, n/k^{\ell - 2})/2$. Thus, we have
            \begin{align*}
               \mathbb{P}(V_\ell) &\geq 1 - \exp(-49pt_{i_{\ell-2}}/2560) - \exp(-R_{5/8}/6), \\
               &\geq 1 - \exp(-49pt_{i_{\ell - 2}}/2560) - \exp(-\min(n/k^{\ell - 2},pt_{i_{\ell-2}})/12),
            \end{align*}
            as desired. Moreover, $V_\ell$ is contained in the event that
            \begin{align*}
                t_{i_{\ell}} - t_{i_{\ell-2}} \geq R_{5/8}/2p \geq \min(t_{i_{\ell-2}}, n/k^{\ell-2})/4.
            \end{align*}
            Finally, $V_{\ell}$ is clearly determined by variables listed in the statement of the Lemma, so we are done.
        \end{proof}

    We now return to the proof of Lemma \ref{lemmaMain}. Condition on $E_A$. Define $A_2$ to be the intersection of the events $V_\ell$ across all even indices $2\lceil 201 \log \log n \rceil \leq \ell \leq 2\lceil 201 (\log \log n)^3 \rceil$. Suppose $A_1 \cap A_2$ is true. By Lemma \ref{lemmaRapidIncrease}, we know that for all $\ell \geq 2\lceil 201 \log \log n \rceil$, $t_{i_\ell} \geq f(n/2e^{403(\log \log n)^2})/1700$. Thus for all even $2\lceil 201 \log \log n \rceil \leq \ell \leq 2\lceil 201 (\log \log n)^3 \rceil - 2$, we have by Lemma \ref{lemmaExpInContext} that
    \begin{align*}
        t_{i_{\ell+2}} \geq t_{i_\ell} + \frac{1}{4} \min\bigg(t_{i_\ell}, \frac{n}{k^\ell}\bigg) \geq \min\bigg(\frac{5}{4}t_{i_\ell},\frac{n}{4e^{402 (\log \log n)^4}} \bigg).
    \end{align*}
    Hence by induction, we have
    \begin{align*}
        t_{i_\ell} &\geq \min\bigg(\frac{n}{4e^{402(\log \log n)^4}}, \frac{f(n/2e^{403 (\log \log n)^2})}{1700}\bigg(\frac{5}{4}\bigg)^{\frac{\ell-2\lceil201 \log \log n\rceil}{2}}\bigg) \\
        t_{2\lceil 201 (\log \log n)^3 \rceil} &\geq \min\bigg(\frac{n}{4e^{402(\log \log n)^4}}, \frac{f(n/2e^{403 (\log \log n)^2})}{1700}\bigg(\frac{5}{4}\bigg)^{200(\log \log n)^3}\bigg) \\
        &\geq \min\bigg(\frac{n}{4e^{402(\log \log n)^4}}, \frac{f(n/2e^{403 (\log \log n)^2})}{1700}e^{44(\log \log n)^3}\bigg).
    \end{align*}
    Now, observe that as a function of $p$, $p^8\exp(\Cr{421}/p^2)$ achieves its minimum at $\sqrt{\Cr{421}}/2$, where it is equal to $\Cr{421}^4e^4/2^8$. Thus we may adjust $\Cr{421}$ such that for all $n > e^{e^{\Cr{421}/p^2}}$,
    \begin{align*}
        \frac{n}{4e^{402(\log \log n)^4}} &\geq f(n).
    \end{align*}
    Adjusting $\Cr{421}$ again, we may ensure
    \begin{align*}
        44 (\log \log n)^3 &> 403 (\log \log n)^2 + \log (3400), \\
        \frac{f(n/2e^{403 (\log \log n)^2})}{1700}e^{44(\log \log n)^3} &\geq f(n)\frac{e^{44(\log \log n)^3}}{3400e^{403 (\log \log n)^2}} \\
        &\geq f(n).
    \end{align*}
    Thus we find $t_{2\lceil 201 (\log \log n)^3 \rceil} \geq f(n)$. Hence $A_1 \cap A_2$ implies that
    \begin{align*}
        T_n &\geq t_{2\lceil 201 (\log \log n)^3 \rceil} + g_1 - 1 \\
        &\geq f(n) + g_1 - 1.
    \end{align*}
    Now it suffices to compute $\mathbb{P}(A_1 \cap A_2)$. Condition on the random variables
    \begin{itemize}
        \item $D_{j, t}$, for $j \leq i_{2\lceil 201 \log \log n \rceil}$ and $t \in \mathbb{Z}^+$;
        \item $S_{j, t}$, $B_{j, t}$ and $B_{j, t}'$, for $j \geq a_{i_{2\lceil 201 \log \log n \rceil}}$ and $t \in \mathbb{Z}^+$.
    \end{itemize}
    Recall from Lemma \ref{lemmaRapidIncrease} that these variables determine $A_1$. For any even $\ell \in (2\lceil 201 \log \log n \rceil, \\ 2 \lceil 201 (\log \log n)^3 \rceil]$, note that $V_{2 \lceil 201 \log \log n \rceil}, \ldots, V_l$ are all determined by the random variables
    \begin{itemize}
        \item $D_{j, t}$, for all $j \in [n], t \in \mathbb{Z}^+$,
        \item $S_{j, t}$, $B_{j, t}$ and $B_{j, t}'$, for $j \geq a_{i_{\ell}}$ and $t \in \mathbb{Z}^+$,
        \item $D_{j, t}'$ and $C_{j, t}$, for $j \leq \ell$ and $t \in \mathbb{Z}^+$.
    \end{itemize}
    Now by Lemma \ref{lemmaExpInContext} we always have that
    \begin{align*}
        \mathbb{P}(V_{\ell+2} | V_{2 \lceil 201 \log \log n \rceil}, \ldots V_l) &\geq 1 - \exp(-49pt_{i_{\ell - 2}}/2560) - \exp(-\min(n/k^{\ell - 2},pt_{i_{\ell-2}})/12) \\
        &\geq 1 - 2\exp(-49p(402 \log \log n)/2560) \\
        &\geq 1 - 2\exp(-p\log \log n).
    \end{align*}
    Thus $\mathbb{P}(V_{\ell + 2} | A_1) \geq 1 - 2\exp(-p\log \log n)$. Applying a union bound, we find
    \begin{align*}
        \mathbb{P}(A_2 | A_1) \geq 1 - 804 (\log \log n)^3 \exp(-p\log \log n).
    \end{align*}
    By Lemma \ref{lemmaRapidIncrease} we know $\mathbb{P}(A_1 | E_A) \geq 0.9$. Hence
    \begin{align*}
        \mathbb{P}(A_1 \cap A_2 | E_A) \geq 0.9 - (0.9 \cdot 804) (\log \log n)^3 \exp(-p\log \log n).
    \end{align*}
    Thus it suffices to show that we may adjust $\Cr{421}$ such that all $n > e^{e^{\Cr{421}/p^2}}$ satisfy
    \begin{align}\label{eqFinalInequality}
        (\log \log n)^3\exp(-p\log \log n) < 10^{-4}.
    \end{align}
    We have that $a^3\exp(-pa)$ is decreasing on $(3/p, \infty)$. Since $\Cr{421} \geq 4$, we have that for $n > e^{e^{\Cr{421}/p^2}}$,
    \begin{align*}
        (\log \log n)^3\exp(-p\log \log n) \leq \Cr{421}\exp(-\Cr{421}/p)/p^2 < \Cr{421}(\exp(-1/p)/p^2)^{\Cr{421}} < \Cr{421}(0.6^{\Cr{421}}),
    \end{align*}
    where in the last inequality we use $\exp(-1/x)/x^2 < 0.6$ for all $x \in (0, 1)$. Hence taking a sufficiently large $\Cr{421}$ gives us Inequality \eqref{eqFinalInequality}. This proves Lemma \ref{lemmaMain} as desired.

\end{document}